\documentclass[11pt]{article}
\usepackage{amsmath,amsfonts,amssymb,amsthm,mathrsfs,graphics,graphicx,mathabx,setspace}

\allowdisplaybreaks

\newcommand{\I}{{\bf 1}}
\newtheorem{proposition}{Proposition}[section]
\newtheorem{theorem}[proposition]{Theorem}
\newtheorem{corollary}[proposition]{Corollary}
\newtheorem{lemma}[proposition]{Lemma}
\newtheorem{remark}[proposition]{Remark}

\newcommand{\nc}{\newcommand}

\nc{\BHd}{\mathbb{H}^d}
\nc{\R}{{\mathbb R}}
\nc{\bS}{{\mathbb S}^{d-1}}
\nc{\N}{{\mathbb N}}
\nc{\BB}{{\mathbb B}}
\nc{\C}{{\mathbb C}}
\nc{\Z}{{\mathbb Z}}
\nc{\BP}{\mathbf{P}}
\nc{\BH}{\mathbb{H}}
\nc{\BQ}{\mathbf{Q}}
\nc{\bN}{{\mathbf N}}
\nc{\BX}{{\mathbb X}}
\nc{\BY}{{\mathbb Y}}
\nc{\cB}{{\mathcal B}}
\nc{\cE}{{\mathcal E}}
\nc{\cK}{{\mathcal K}}
\nc{\cC}{{\mathcal C}}
\nc{\cL}{{\mathcal L}}
\nc{\cM}{{\mathcal M}}
\nc{\cR}{{\mathcal R}}
\nc{\cX}{{\mathcal X}}
\nc{\cY}{{\mathcal Y}}
\nc{\dint}{{\rm d}}
\nc{\D}{\Delta}
\nc{\g}{\gamma}
\nc{\cI}{\mathcal{I}}
\nc{\cZ}{\mathcal{Z}}
\nc{\cum}{\operatorname{cum}}
\nc{\sZ}{{\mathscr{Z}}}
\nc{\origin}{o}
\nc{\bI}{\mathbf{1}}
\nc{\1}{\bI}
\nc{\E}{\mathbf{E}}
\newcommand{\defeq}{\mathrel{\mathop:}=}
\nc{\sfp}{\mathsf{p}}

\newcommand{\horo}{\mathbb{B}}

\DeclareMathOperator{\Cov}{Cov}

\nc{\Vol}{V_d}
\DeclareMathOperator{\Var}{Var}
\DeclareMathOperator{\calI}{\mathcal{I}}
\nc{\Ih}{{\mathcal{I}_d}}
\DeclareMathOperator{\hb}{\mathbb{B}_{\rm h}^{\text{$d$}}}

\DeclareMathOperator{\msfp}{\mathsf{p}}
\DeclareMathOperator{\TG}{G}

\usepackage{colortbl}

\numberwithin{equation}{section}

\begin{document}
\title{Boolean models in hyperbolic space}


\renewcommand{\thefootnote}{\fnsymbol{footnote}}

\author{Daniel Hug\footnotemark[1], G\"unter Last\footnotemark[2], and Matthias Schulte\footnotemark[3]\,}

\footnotetext[1]{Institute of Stochastics, Karlsruhe Institute of Technology, daniel.hug@kit.edu}

\footnotetext[2]{Institute of Stochastics, Karlsruhe Institute of Technology, guenter.last@kit.edu}

\footnotetext[3]{Institute of Mathematics, Hamburg University of Technology, matthias.schulte@tuhh.de}

\maketitle

\begin{abstract}
The union of the particles of a stationary Poisson process of compact (convex) sets in Euclidean space is called Boolean model and is a classical topic of stochastic geometry. In this paper, Boolean models in hyperbolic space are considered, where one takes the union of the particles of a stationary Poisson process in the space of  compact (convex) subsets of the hyperbolic space. Geometric functionals such as the volume of the intersection of the Boolean model with a compact convex observation window are studied. In particular, the asymptotic behavior for balls with increasing radii as observation windows is investigated. Exact and asymptotic formulas for expectations, variances, and covariances are shown and univariate and multivariate central limit theorems are derived. Compared to the Euclidean framework, some new phenomena can be observed.

\medskip

\end{abstract}


\section{Introduction}

Boolean models are an important class of random sets in Euclidean
space. We study them in hyperbolic space and consider geometric
quantities such as volume or surface area within an observation window. In
particular we derive asymptotic results for balls with increasing
radii as observation windows. In comparison to the Euclidean
framework, we observe new phenomena that emerge from fundamental
properties of hyperbolic space.

To introduce the Boolean model in the Euclidean setting, let
$\eta_{{\rm Euc}}$ be a Poisson process on the space of compact convex
subsets of $\mathbb{R}^d$. The points of this point process are
compact convex sets that are referred to as particles or grains, and
$\eta_{{\rm Euc}}$ is called a  particle process. The union set
$$
Z_{{\rm Euc}} := \bigcup_{K\in\eta_{\rm Euc}} K
$$
is a random set in $\mathbb{R}^d$ that is known as the Boolean model derived from $\eta_{{\rm Euc}}$. Usually, $\eta_{\rm Euc}$ is assumed to be stationary, that is,
$$
\eta_{{\rm Euc}} \overset{d}{=} \eta_{{\rm Euc}}+t \quad \text{ for each }  t\in\mathbb{R}^d,
$$
where $\overset{d}{=}$ stands for equality in distribution and
$\eta_{{\rm Euc}}+t$ is the particle process obtained by translating
each particle of $\eta_{{\rm Euc}}$ by $t$. The stationarity of
$\eta_{\rm Euc}$ implies that $Z_{\rm Euc}$ is also stationary in the
sense that its distribution is invariant under translations.
The Boolean model is a fundamental object of stochastic geometry (see \cite{LP18,SW08})
and continuum percolation (see \cite{MeesterRoy}) and can serve as a model for disordered spatial
structures in physics and other applied sciences; see, e.g.,  \cite{AKM2009,BS2022,CSKM13,HHKM2014,Jeulin2021,Mecke2000,MS2000,Mecke2001,MS2002,Molchanov1997,RRSKG2021}.

Throughout this paper, we work in the $d$-dimensional hyperbolic space $\mathbb{H}^d$, the unique simply connected Riemannian manifold of constant sectional curvature $-1$. There exist several models for hyperbolic space such as the Poincar\'e halfspace model, the Poincar\'e ball model (also known as the conformal ball model), the hyperboloid model, and the Beltrami--Klein model, which are all isometric and, thus, equivalent, whence we do not restrict to any particular of these models (for an introduction to hyperbolic geometry, see, e.g., \cite{Benedetti.1992,Ratcliffe}). We denote by $\mathcal{K}^d$ the set of compact convex subsets of $\mathbb{H}^d$ and let $\eta$ be a Poisson process in $\mathcal{K}^d$. From the particle process $\eta$ we derive the Boolean model
$$
Z := \bigcup_{K\in\eta} K,
$$
i.e., the random set $Z$ in $\mathbb{H}^d$ is the union of all particles of $\eta$ (see Figure \ref{Fig1}).
\begin{figure}[ht]
\centering
    \hspace*{-1cm}
 \includegraphics[width=0.5\textwidth, angle=0]{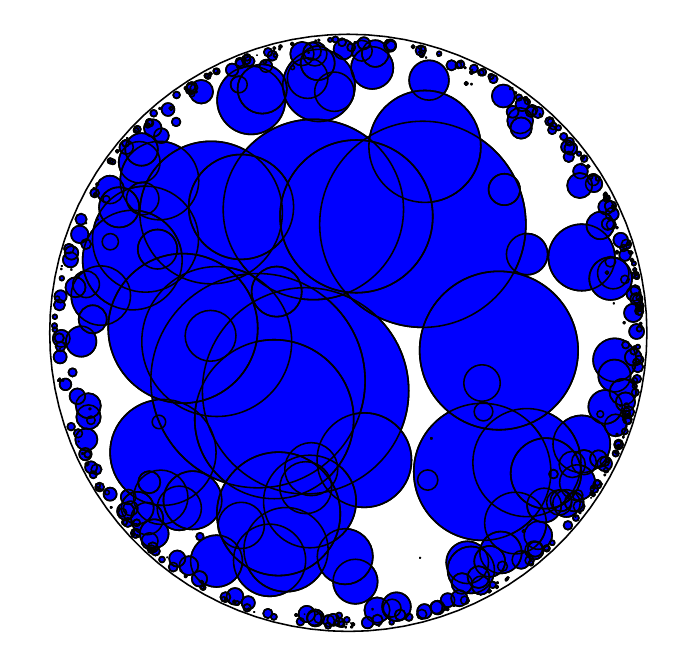}\includegraphics[width=0.5\textwidth, angle=0]{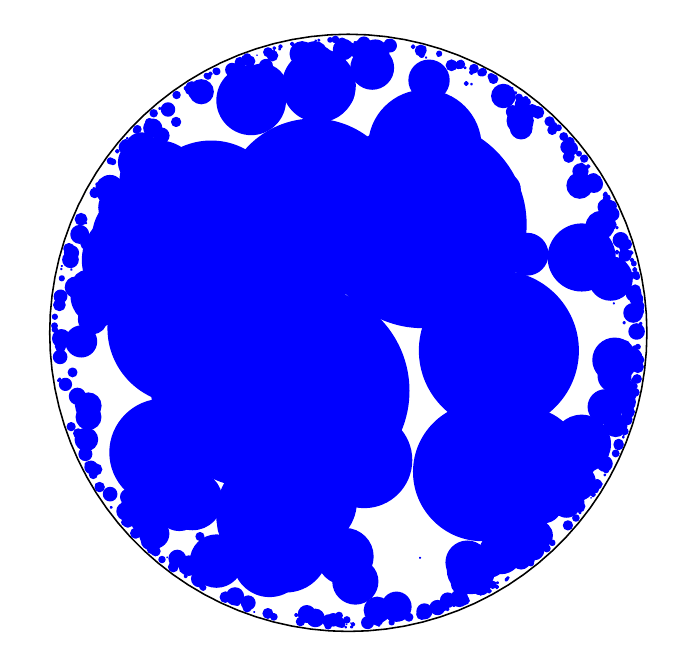}
 \caption{Realization of a Poisson particle process of random balls (left panel) and its union set (right panel) in the conformal ball model of $\mathbb{H}^2$.  The intensity of the underlying Poisson process is $\gamma=500/\mathcal{H}^2(\mathbb{B}_5)$, the radius distribution is uniform in $[0,1]$. }\label{Fig1}
\end{figure}

As in the Euclidean case one would like $Z$ to have some invariance properties. To this end, we denote by $\Ih$ the class of all isometries of $\mathbb{H}^d$ and require that
\begin{equation}\label{eq:stationaryeta}
\eta \overset{d}{=} \varrho \eta \quad \text{ for each }   \varrho\in\Ih.
\end{equation}
Here, by $\varrho \eta$ we mean the particle process obtained by
applying $\varrho$ to every particle of $\eta$. Then the Boolean model
$Z$ becomes also invariant under isometries, which means that
$Z\overset{d}{=}\varrho Z$ for all $\varrho\in\Ih$. Any particle process satisfying \eqref{eq:stationaryeta} will be called stationary.

Let us briefly summarize the main findings of this work. First we initiate a study of  general stationary processes of compact particles in hyperbolic space. For the intensity measure $\Lambda$ of such a stationary particle process and a given center function we derive a crucial decomposition in terms of the uniquely defined distribution of the so-called typical particle (or typical grain) ${\rm G}$, an intensity parameter $\gamma$, and an integration over isometries (see Theorem \ref{thmdisintegration}), which is frequently applied throughout the work. The typical particle ${\rm G}$ is a random compact set centered at a fixed point $\sfp$, which can be considered as the counterpart of the origin in the Euclidean case. The intensity measure $\Lambda$ is locally finite if and only if $\gamma$ is finite and the typical particle satisfies a basic first moment condition (see Theorem \ref{thmchar}).

In the main part of this paper we consider the Boolen model $Z$ constructed from an underlying stationary Poisson process $\eta$ of compact convex particles in $\mathbb{H}^d$ with finite intensity $\gamma$. A different way to construct $\eta$ is to take i.i.d.\ copies of the typical particle and to move them by random isometries coming from an isometry invariant Poisson process on $\Ih$. We are interested in geometric quantities $\phi$ such as
the volume or the surface area of the Boolean model within a compact convex observation
window $W\in\mathcal{K}^d$, whence we consider $\phi(Z\cap W)$. We
often choose for $W$ a $d$-dimensional ball $\mathbb{B}_R$ around
$\mathsf{p}$ with radius $R$ and let $R\to\infty$. More generally, we study functionals $\phi$ defined on finite unions of compact convex sets that are assumed to be geometric, which means that they are measurable, additive, locally bounded, and isometry invariant. Sometimes continuity is also required. Our findings are subject to natural moment conditions on $\overline{\Vol}({\rm G})$, the volume of the $1$-parallel set of the typical particle ${\rm G}$. In particular, we contribute to the following topics:

\begin{itemize}
\item {\em Expectations.} We derive an exact formula for $\mathbf{E}\phi(Z\cap W)$ and an asymptotic formula for $\mathbf{E}\phi(Z\cap \mathbb{B}_R)/\Vol(\mathbb{B}_R)$ as $R\to\infty$. The latter involves an isometry invariant integration over the space of horoballs (see Theorem \ref{thm:mean}). In the cases where $\phi$ is volume, surface area, a general intrinsic volume or the Euler characteristic more explicit expressions are provided.
\item {\em Variances and covariances.} For two geometric functionals $\phi$ and $\psi$, we establish an exact formula for $\mathbf{Cov}(\phi(Z\cap W),\psi(Z\cap W))$ and compute the limit of $\mathbf{Cov}(\phi(Z\cap \mathbb{B}_R),\psi(Z\cap \mathbb{B}_R))/\Vol(\mathbb{B}_R)$ as $R\to\infty$, which involves again an integration over horoballs. Special emphasis is given to volume and surface area, for which exact and asymptotic variance and covariance formulas are derived (see Theorems \ref{th:localcovariance} and \ref{th:asympcovariance}). Moreover, we provide conditions which ensure that $\mathbf{Var}(\phi(Z\cap \mathbb{B}_R))$ is of order $\Vol(\mathbb{B}_R)$ (see Theorem \ref{thm:lower_bound_variance}), which is an essential step towards establishing central limit theorems.
\item {\em Central limit theorems.} We prove a univariate central limit theorem for $\phi(Z\cap \mathbb{B}_R)$ as $R\to\infty$ as well as a multivariate central limit theorem for several geometric functionals. In the univariate case we obtain rates of convergence for the Wasserstein and the Kolmogorov distance.
\end{itemize}

Our proofs rely on general tools for Poisson processes such as the Mecke formula, the Fock space representation and the Malliavin-Stein method and additionally on geometric arguments. The same probabilistic tools for Poisson processes have previously been employed for the analysis of Euclidean Boolean models in e.g.\ \cite{BST,HLS,LP18,SY19}, but the geometric arguments need to be adapted and extended significantly. One  reason is the presence of asymptotically non-vanishing boundary effects discussed further below which lead to additional and different contributions in comparison with the Euclidean framework. Another reason is that in Euclidean space it is often exploited that space can be  partitioned by congruent half-open cubes. In hyperbolic space, we show the existence of a covering of the whole space by ``identical sets'' such that not too many of them overlap, which is obvious in $\mathbb{R}^d$, but becomes non-trivial in $\mathbb{H}^d$. This economic covering result is used to establish bounds for geometric functionals
and iterated difference operators of the Boolean model in an observation window. For the results concerning expectations of general  intrinsic volumes, iterated kinematic formulas are required which take a convenient form if a suitable renormalisation of the intrinsic volumes is used.

The Euclidean counterparts of our results outlined above are
well known. For an introduction to Boolean models in $\mathbb{R}^d$
and formulas for exact and asymptotic expectations, we refer e.g.\ to 
\cite[Chapter 9]{SW08}, while the covariance structure and central limit theorems
for general geometric functionals were first considered in
\cite{HLS}. A recent survey is provided in \cite{HLW2023}. Previously, central limit theorems for volume, surface
area and related functionals were explored e.g.\ in
\cite{Baddeley,Heinrich2005,HeinrichMolchanov,Mase,Molchanov1995}. The
findings of \cite{HLS} were further refined in \cite{HKLS,LP18,SY19}. 
Poisson cylinder processes were treated in \cite{BST,HeinrichSpiess2009,HeinrichSpiess2013}.

To illustrate new phenomena that arise in hyperbolic space, let us consider the asymptotic variance of the volume. We will show that
\begin{equation}\label{eqn:Asymptotic_Variance_Volume}
\lim_{R\to\infty} \frac{\Var \Vol(Z\cap \mathbb{B}_R)}{\Vol(\mathbb{B}_R)} = e^{-2\gamma\E V_d(\TG)} \int_{\BHd} \left(e^{\gamma C(\sfp,    z)}-1\right) \BP(z\in \horo_{U,T}) \, \mathcal{H}^d(\dint z)
\end{equation}
with independent random variables
$U\sim\operatorname{Uniform}(\mathbb{S}^{d-1}_\sfp)$ and
$-T\sim\operatorname{Exp}(d-1)$. Here, $\gamma$ is the intensity
parameter, $C(\cdot,\cdot)$ is the so-called covariogram function of the typical
particle (the definition has to be adjusted in hyperbolic space), and
$V_d$ stands for the volume functional in $\BHd$ (that is,
$d$-dimensional Hausdorff measure $\mathcal{H}^d$). By $\mathbb{B}_{u,r}$ we denote
the closed and convex horoball in direction $u$ with signed distance
$r$ to $\mathsf{p}$, which is obtained as the limit of an increasing
sequence of (compact convex geodesic) balls with radius $R$ and with center
in direction $u$ and distance $R+r$ from $\mathsf{p}$, as
$R\to\infty$. In other words, horoballs are closures of suitably
chosen nested sequences of compact convex metric balls (see
\cite[Theorem 3.4]{BR16} for a generic argument); see Section
\ref{sec:2.1} and Figure \ref{Fig3} for an illustration. While the
latter is a halfspace in the Euclidean case, this is no longer true in
hyperbolic space. For a Boolean model in Euclidean space $\mathbb{R}^d$ one has the
simpler description of the asymptotic variance by
$$
\lim_{R\to\infty} \frac{\Var \Vol(Z_{\rm Euc}\cap \mathbb{B}_R)}{\Vol(\mathbb{B}_R)} = e^{-2\gamma\E V_d(\TG)} \int_{\R^d} \left(e^{\gamma C(0,    z)}-1\right)  \, \mathcal{H}^d(\dint z),
$$
 which is almost the same formula as in \eqref{eqn:Asymptotic_Variance_Volume}, except for the probability involving the random horoball. This difference is caused by boundary effects (see Figure \ref{Fig2} for an illustration).
 \begin{figure}[ht]
 \includegraphics[width=0.5\textwidth, angle=0]{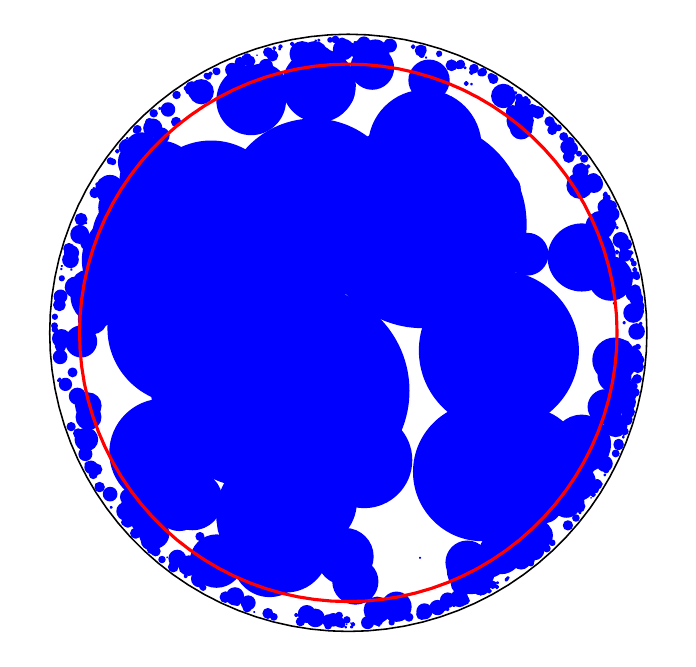}\includegraphics[width=0.5\textwidth, angle=0]{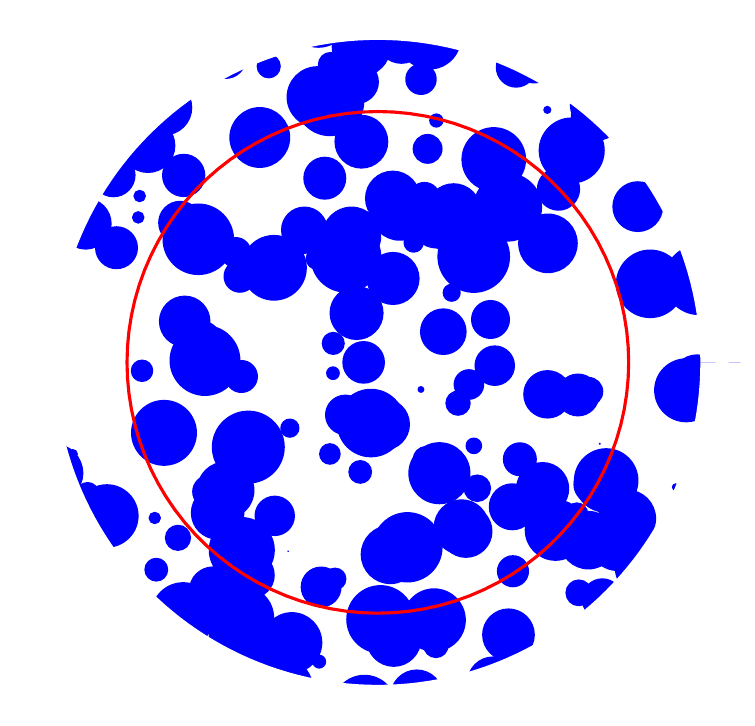}
 \caption{Visual comparison of realizations of Boolean models observed in a ball of radius $R$ in the conformal ball model of $\mathbb{H}^2$ (left panel) and in Euclidean space (right panel).}\label{Fig2}
\end{figure}
 For $Z_{\rm Euc}$ the behavior close to the boundary of the observation window $\mathbb{B}_R$ can be neglected as the fraction of the particles hitting the boundary of $\mathbb{B}_R$ of all particles intersecting $\mathbb{B}_R$ tends to zero as $R\to\infty$. In the present paper this is not true as the volume and the surface area of $\mathbb{B}_R$ in $\mathbb{H}^d$ are of the same order in $R$, which is a basic fact of hyperbolic geometry. 

The literature contains many other interesting (and partially surprising) properties of models of stochastic geometry in hyperbolic space. Examples from continuum percolation are the seminal paper \cite{BenjaminiSchramm01}, studying Poisson--Voronoi percolation in the hyperbolic plane, and \cite{Tykesson07}, dealing with the spherical Boolean model. The recent contributions \cite{HM22,HM23} on Poisson--Voronoi percolation in the hyperbolic plane explore the behavior of the critical probability for an infinite cluster as the intensity of the underlying Poisson process approaches infinity or zero.  Visibility properties of the spherical Boolean model are studied in
\cite{BJST09,TykessonCalka13}. Central and non-central limit theorems for hyperplane processes and processes of $\lambda$-geodesics or lower-dimensional flats are  derived in \cite{BHT,HHT2021,KRT23+,KRT24+}. Properties of the typical cell of a Poisson--Voronoi tessellation and the zero cell of a Poisson hyperplane tessellation in hyperbolic space are considered in \cite{GKT22,HeHu}. New phenomena arising in the study of intersection probabilities for uniform random  flats in hyperbolic space are discovered in \cite{SST24}.  The recent preprint \cite{ACELUe23} shows that the low-density Poisson--Voronoi tessellation has a non-trivial limit. In all these examples the non-amenability of hyperbolic space leads to properties which are significantly
different from the Euclidean case.  Finally, we mention that the hyperbolic plane has been identified as a natural framework for studying random geometric graphs, which serve as models of complex networks; see, e.g.,   \cite{CFS22,FHMS21,FY20} and the literature cited there (going back at least to the introduction of the KPKVB model  in \cite{KPKVB2010}).

This paper is organized as follows. In Section \ref{sec:2} we consider stationary  processes of compact particles in hyperbolic space and establish, in particular, a disintegration result for their intensity measures. Then we introduce the concept of a geometric functional in hyperbolic space and discuss intrinsic volumes as well as their localisations, the curvature measures, as important examples in Section \ref{sec:GF}. Our main results for geometric functionals of hyperbolic Boolean models within observations windows are presented in Sections \ref{sec:4}--\ref{sec:6}. Most of the proofs are postponed to later sections. Section \ref{sec:4} concerns local and asymptotic mean value formulas for geometric functionals, whose proofs are given in Section \ref{sec:10}. In Section \ref{sec:5} 
we provide formulas for local and asymptotic  variances and covariances of
geometric functionals as well as of volume and surface area. The proof for general geometric functionals is deferred to Section \ref{sec:11}, while the results for volume and surface area are shown in Sections \ref{sec:12} and \ref{sec:cov_vol_surf}. A lower variance bound is stated in Section \ref{sec:lower_bound_variance} and proven in Section \ref{sec:14}. Univariate and multivariate central limit theorems are discussed in Section \ref{sec:6} and finally established in Section \ref{sec:15}. Sections \ref{sec:7}--\ref{sec:9} contain various auxiliary
results, which could also be of independent interest, namely several new integral geometric results in hyperbolic space in Section \ref{sec:7}, an economic covering result in Section \ref{sec:8}, and some intermediate bounds on geometric functionals and difference operators in Section \ref{sec:9}.

 \section{Stationary particle processes and Boolean models}\label{sec:2}

In the present section, we first provide some general facts on stationary   processes of compact (convex) particles in hyperbolic space, which do not require any Poisson assumption. As a preparation, we first collect some useful results and facts from hyperbolic geometry. Then we introduce and discuss the Boolean model.

\subsection{Basic facts from hyperbolic geometry}\label{sec:2.1}

We consider the hyperbolic space $\BHd$ with intrinsic metric
(distance function) $d_h$. Let $\mathsf{p}$ denote a fixed point of
$\mathbb{H}^d$.  For $x\in\mathbb{H}^d$ and $r\ge 0$, we denote by
$\mathbb{B}(x,r)\defeq\{z\in\mathbb{H}^d\colon d_h(z,x)\le r\}$ the
 (geodesic) ball with center $x$ and radius $r\ge 0$; in particular, we
set $\mathbb{B}_r\defeq\mathbb{B}(\mathsf{p},r)$. We write
$\mathbb{S}^{d-1}_{\msfp}$ for the unit sphere in the tangent space
 $T_{\msfp}\mathbb{H}^d$ of $\mathbb{H}^d$ at $\msfp$ and
$\mathcal{H}^{d-1}_{\msfp}$ for the $(d-1)$-dimensional Hausdorff measure  on $T_{\msfp}\mathbb{H}^d$ (with respect to the scalar product
induced by the Riemannian metric of $\mathbb{H}^d$). The
$s$-dimensional Hausdorff measure in $\mathbb{H}^d$ (with respect to
the metric $d_h$) is denoted by $\mathcal{H}^s$ for $s\ge 0$. In
particular, $\mathcal{H}^d$ is normalized in such a way that on Borel
measurable sets it coincides with the Riemannian volume measure, for
which we write $\Vol$.  The exponential map
$\exp_\sfp:T_{\msfp}\mathbb{H}^d\to \mathbb{H}^d$ is a diffeomorphism
between the tangent space at $\sfp$ and the hyperbolic space. In the
introduction, the notion of a horoball in hyperbolic space has already
been introduced in an informal way.  For
$u\in \mathbb{S}^{d-1}_{\msfp}$ and $t\in\R$, we denote by
$\horo_{u,t}\subset\mathbb{H}^d$ the (closed and convex) horoball
which contains $\exp_{\msfp}(tu)\in\BHd$ in its boundary and satisfies
$\BB(\exp_{\sfp}((t+R)u),R)\to \horo_{u,t}$ as $R\to\infty$ (with
respect to the Fell
topology). In Section \ref{sec:4} we will introduce a measure on the space of horoballs, which is invariant with respect to the natural operation of the isometry group $\Ih$ (see below).

\begin{figure}[ht]
    \centering
 \includegraphics[width=0.5\textwidth, angle=0]{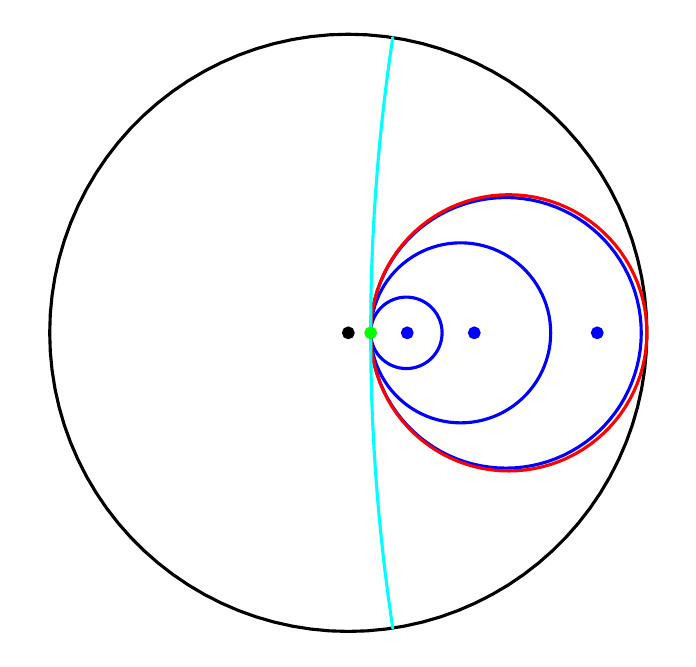}
 \caption{Nested sequence of balls (blue), common boundary  point $\exp_{\msfp}(tu)$ (green) and boundary of the limiting horoball (red) in the conformal ball model of $\mathbb{H}^2$.
Also shown is the geodesic (cyan) that touches the horoball at $\exp_{\msfp}(tu)$.}\label{Fig3}
\end{figure}

In the following, we will repeatedly use formulas for the surface area and the volume of a hyperbolic ball of radius $r$,
 which are given by
\begin{equation}\label{eq:surfaceball}
\mathcal{H}^{d-1}(\partial \mathbb{B}_r)=\omega_d \sinh^{d-1}(r),
\end{equation}
where 
$\omega_n\defeq n \kappa_n \defeq {2 \pi^{n/2}}/{\Gamma(n/2)}$, $n\in\N$,
is the surface area and   $\kappa_n$ is the volume of an $n$-dimensional unit ball in the Euclidean space $\R^n$, and
\begin{equation}\label{volbalr}
\mathcal{H}^{d}(\mathbb{B}_r)= \omega_d \int_{0}^{r} \sinh^{d-1}(s) \, \dint s.
\end{equation}
As usual, $\sinh$ and $\cosh$ denote the hyperbolic sine and cosine function.
Formula \eqref{volbalr} is a special case of the following general polar coordinate formula \eqref{eq:polarcoord} (see
  Sections 3.3 and 3.4 and especially  formulas (3.25) and (3.26) in the monograph \cite{Chavel}). For a measurable function  $f\colon \mathbb{H}^d\to[0,\infty]$, one has
 \begin{equation}\label{eq:polarcoord}
 \int_{\mathbb{H}^d}f(x)\, \mathcal{H}^d(\dint x)
 =\int_{\mathbb{S}^{d-1}_\sfp}\int_0^\infty f(\exp_\sfp(su))\sinh^{d-1}(s) \,\dint s \,\mathcal{H}^{d-1}_\sfp(\dint u).
\end{equation}
For  $d=2$ we thus get $\mathcal{H}^1(\partial\mathbb{B}_r)=2\pi\sinh(r)$ and $\mathcal{H}^2(\mathbb{B}_r)= 2 \pi(\cosh(r)-1)$. In general, it follows from \eqref{eq:surfaceball}, \eqref{volbalr} and  l'Hospital's rule that
\begin{equation}\label{eq:limit1}
\lim_{r\to\infty} \frac{\mathcal{H}^{d-1}(\partial \mathbb{B}_r)}{\mathcal{H}^{d}(\mathbb{B}_r)}= d-1 ,
\end{equation}
which is in contrast to the growth behavior in Euclidean space (see \eqref{eq:limit2} for a more general relation).

Let $\Ih$ denote the group of isometries of $\BHd$. 
With respect to the compact open topology, $\Ih$ is a locally compact topological group with countable
base  of the topology (see \cite{AM10}, \cite[pp.~46--49]{NoKoI}, \cite[Chapter 5]{Ratcliffe}, \cite[III.6, IV.1, IV.6.2]{Sakai}).  Up to a multiplicative constant there exists a uniquely
determined Haar measure $\lambda$ on $\Ih$ (see \cite[Theorem 9.2.6]{Cohn}).
Since $\Ih$ is unimodular
(see \cite[Chapter X, Proposition 1.4]{Helgason.1962} or \cite[Proposition C
4.11]{Benedetti.1992}), $\lambda$ is left invariant, right invariant,
and inversion invariant. The isometry group operates continuously on $\mathbb{H}^d$. For each $x\in\mathbb{H}^d$, the map $\varrho\mapsto \varrho x$ from $\Ih$ to $\mathbb{H}^d$ is  open and proper (inverse images of compact sets are compact). In other words, $\mathbb{H}^d$ is a homogeneous $\Ih$-space (see \cite{Nachbin,SW08}).
The measure $\mathcal{H}^{d}$ is locally finite and invariant under isometries (by its definition), that is
$\mathcal{H}^{d}(\varrho B)=\mathcal{H}^{d}(B)$
for each Borel set $B\subseteq \BH^d$ and each
$\varrho\in \Ih$, where  $\varrho B \defeq \{\varrho x \colon x \in B\}$. Here we write $\varrho x$ instead of $\varrho (x)$ for
$\varrho\in \Ih$ and $x\in \BHd$.
Hence, up to a factor, $\mathcal{H}^{d}$ is the only isometry invariant, locally finite measure
on $\mathbb{H}^d$; see \cite[p.~138, Theorem 1 (Weil)]{Nachbin} or \cite[Theorem 13.3.1]{SW08} .

We will choose the normalization of
$\lambda$ such that
\begin{equation}\label{eq:hdlambda}
\mathcal{H}^{d}= \int_{\Ih} \1\{\varrho x\in \cdot\}\, \lambda(\dint \varrho)
\end{equation}
for all $x\in\BHd$. Note that the isometry invariant  measure defined on the right-hand side is locally finite, since $\varrho\mapsto \varrho x$ is proper (for each $x\in \mathbb{H}^d$).
Proceeding as in \cite{Hug.2017,Last.2010}, we consider the isotropy group 
$\Ih(\sfp) \defeq \{\varrho \in \Ih \colon \varrho \sfp=\sfp\}$ of $\sfp$, consisting of
isometries fixing $\sfp$. Since $\Ih(\sfp)$ is a compact subgroup of $\Ih$, there is a unique $\Ih(\sfp)$ invariant probability measure $\kappa(\sfp, \cdot)$ on $\Ih(\sfp)$. By
defining $\kappa(\sfp,\Ih \setminus \Ih(\sfp))\defeq 0$, we extend
$\kappa(\sfp,\cdot)$ to $\Ih$. More generally, for $x \in \BHd$ we
define
$$\Ih(\sfp,x) \defeq \{ \varrho \in \Ih \colon  \varrho \sfp=x\},$$
the set of isometries that map $\sfp$ to $x$. Choosing an arbitrary
$\varrho_x \in \Ih(\sfp,x)$, we define
\begin{align}\label{e:kappa}
\kappa(x,B) \defeq \int_{\Ih(\sfp)} \1 \{\varrho_x \circ \varrho \in B\} \, \kappa(\sfp,\dint \varrho), \quad B \in \mathcal{B}(\Ih).
\end{align}
This definition is independent of the choice of $\varrho_x$ (see
\cite{Last.2010}). The map $x\mapsto \varrho_x$ can be chosen in a
measurable way, see \cite[Section 4.5, p.~121]{Ratcliffe} for an
explicit description of a ``hyperbolic translation'' in the hyperbolic
ball model. Hence $\kappa$ is a stochastic transition kernel from
$\BHd$ to $\Ih$. Moreover, $\kappa(x,\cdot)$ is concentrated on
$\Ih(\sfp,x)$.
The kernel $\kappa$ disintegrates $\lambda$ in the sense that
\begin{align}\label{disintlambda}
\int_{\mathbb{H}^d}\int_{\Ih} \I\{\varrho\in\cdot\}\,\kappa(x,\dint \varrho)\,\mathcal{H}^d(\dint x)=\lambda.
\end{align}
To check \eqref{disintlambda}, note that for $x\in \BH^d$  and $\tau\in\Ih$ we have $\tau\varrho_x\in \Ih(\sfp,\tau x)$ and hence, applying twice \eqref{e:kappa},
\begin{equation}\label{eq:2.neu1}
\kappa(\tau x,\cdot)=\int_{\Ih} \I\{\tau\circ\varrho\in\cdot\}\,\kappa(x,\dint \varrho).
\end{equation}
Using \eqref{eq:hdlambda},  \eqref{eq:2.neu1},   Fubini's theorem and the right invariance of $\lambda$,  we get\begin{align*}
&\int_{\mathbb{H}^d}\int_{\Ih} \I\{\varrho\in\cdot\}\,\kappa(x,\dint \varrho)\,\mathcal{H}^d(\dint x)=
\int_{\Ih}\int_{\Ih} \I\{\varrho\in\cdot\}\,\kappa(\tau \sfp,\dint \varrho)\,\lambda(\dint \tau)\\
&=\int_{\Ih}\int_{\Ih(\sfp)} \I\{\tau\circ \sigma\in\cdot\}\,\kappa(\sfp ,\dint \sigma)\,\lambda(\dint \tau)
=\int_{\Ih(\sfp)}\int_{\Ih} \I\{\tau\in\cdot\}\,\lambda(\dint \tau)\,\kappa(\sfp ,\dint \sigma)=\lambda,
\end{align*}
since $\kappa(\sfp,\cdot)$ is a probability measure.

\subsection{Stationary particle processes}\label{subsec2.2}

Let $\cC^d$ denote the space of nonempty compact
  subsets of $\mathbb{H}^d$, endowed with the Hausdorff metric or  the Fell topology, which is a locally compact Hausdorff space with countable base of the topology.
By a particle process on $\BH^d$ we mean a point process $\zeta$ on  $\cC^d$.
We require $\zeta$ to be locally finite in the sense that
$\zeta([C])<\infty$ almost surely for each
$C\in \cC^d$.  Here $[C]$ denotes the set of all $K\in\cC^d$
with $K\cap C\neq\varnothing$.
More formally we can introduce $\zeta$ as a measurable mapping
from $\Omega$ to $\mathbf{N}$, where
$(\Omega,\mathcal{F},\BP)$ is the underlying probability space
and $\mathbf{N}$ is the space of all measures on
$\cC^d$ which take values in $\N_0\cup\{\infty\}$ and are finite if evaluated at $[C]$  for
each compact set $C$. Here measurability refers to the smallest
$\sigma$-field on $\mathbf{N}$ making the mappings $\mu\mapsto\mu(A)$
(from $\mathbf{N}$ to $[0,\infty]$) measurable for
each Borel set $A\subseteq \cC^d$.
Later we will mostly consider a particle process $\zeta$ concentrated on
the space $\cK^d$ of nonempty compact convex subsets of
$\mathbb{H}^d$, that is $\BP(\zeta(\cC^d\setminus \cK^d)=0)=1$.
The isometry group $\Ih$ acts continuously on $\cC^d$.
For $\mu\in\mathbf{N}$
and  $\varrho\in \Ih$ let $\varrho \mu$ denote the image measure of
$\mu$ under $\varrho$. Then the map $(\varrho,\mu)\mapsto \varrho\mu$
is measurable. A particle process $\zeta$ is called {\em stationary}
if $\varrho\zeta\overset{d}{=}\zeta$ for each
$\varrho\in \Ih$.

Suppose that $\zeta$ is a stationary particle process on $\BH^d$.
Then the {\em intensity measure} $\Lambda\defeq\E\zeta$
is isometry invariant, that is we have
$\Lambda(A)=\Lambda(\varrho A)$ for each $\varrho\in \Ih$
and each measurable $A\subseteq \cC^d$,
where $\varrho A\defeq\{\varrho K\colon K\in A\}$. In the main body of this
work we will assume that $\zeta$ is a Poisson process.
Then the distribution of $\zeta$ is completely determined by $\Lambda$;
see \cite[Proposition 3.2]{LP18}. In this case
$\Lambda$ is locally finite, which means that
$\Lambda([C])<\infty$ for each compact set $C\subset\mathbb{H}^d$.
In the following we shall
establish an important disintegration of $\Lambda$.
To do so we need to fix a measurable {\em center function}
$c_h\colon \cC^{d} \rightarrow \BHd$ which is isometry covariant in the
sense that
\begin{align}\label{eq:c_h_property}
c_h(\varrho K)= \varrho c_h(K), \quad K \in \cC^{d},  \varrho \in \Ih.
\end{align}
An example is provided by the circumcenter (see \cite{HeHu}), which
also satisfies $c_h( K)\in K$ for $K \in \cK^{d}$.
Let $\mathbb{B}\subset\BH^d$ be a ball of volume $1$.
The number
\begin{equation}\label{eq:gammadef}
\gamma \defeq
\int_{\cC^{d}} \I\{c_h(K)\in \mathbb{B}\} \, \Lambda(\dint K)
\end{equation}
is said to be the {\em intensity} of $\zeta$. 
By stationarity, this definition does not depend on the
center  of  $\mathbb{B}$. 
We shall mostly assume that $\gamma$ is positive and finite.
Since $\Lambda$ is invariant under isometries and by \eqref{eq:c_h_property}, the measure
$\int_{\cC^{d}} \I\{c_h(K)\in\cdot \}\,\Lambda(\dint K)$ is invariant under isometries
as well. If $\gamma<\infty$, then a simple covering argument shows that this measure is also locally finite, hence
\begin{align}\label{e:intmeasure}
\int_{\cC^{d}} \I\{c_h(K)\in\cdot \}\,\Lambda(\dint K)=\gamma\,\mathcal{H}^d.
\end{align}
In particular, this implies that the value of $\gamma$ is independent of the specific choice of a set $\mathbb{B}$ of volume $1$, once we know that $\gamma<\infty$. Moreover, considering $A_n\defeq\{ K\in\cC^d\colon c_h(K)\in \BB_n\} $ we deduce from \eqref{e:intmeasure} that $\Lambda(A_n)=\gamma\mathcal{H}^d(\BB_n)<\infty$, which shows that $\Lambda$ is a $\sigma$-finite measure.
In Lemma \ref{l:2.3} we shall see that the value of $\gamma$
is also independent of the center function $c_h$.
Let
$$
\cC_{\sfp}^{d} \defeq \{K \in \cC^{d}\colon \ c_h(K)=\sfp\}
$$
be the set of all compact sets with center at $\sfp$.

\begin{theorem}\label{thmdisintegration}
  Let $\zeta$ be a stationary particle process on $\BHd$ with
  intensity $\gamma \in (0,\infty)$ (with respect to the given center
  function $c_h$).  Then there exists a unique   probability measure
$\BQ$ on $\cC^d$ which is concentrated
on $\cC_\sfp^{d}$, invariant under isometries fixing $\sfp$, and satisfies
\begin{equation}\label{eqdisintegration}
\Lambda =\gamma \int_{\cC^{d}}\int_{\Ih} \I\{\varrho K\in\cdot\} \,\lambda(\dint \varrho) \,\BQ(\dint K).
   \end{equation}
   This measure is given by
\begin{align}\label{e:BQ}
\BQ = \frac{1}{\gamma} \int_{\cC^{d}} \int_{\Ih}
 \1 \{\varrho^{-1} K \in \cdot\}w(c_h(K))\, \kappa(c_h(K),\dint \varrho) \, \Lambda( \dint K),
\end{align}
where $w\colon\BH^d\to [0,\infty)$ is any measurable function with 
\begin{equation}\label{eq:integralone}
\int_{\mathbb{H}^d} w(x)\,\mathcal{H}^d(\dint x)=1.
\end{equation}
\end{theorem}

\begin{proof}
Since  $\kappa(c_h(K),\cdot)$ is a probability measure, it follows from
 \eqref{e:intmeasure}  that $\BQ$ is  a probability measure.  The identity $\BQ(\cC^d\setminus \cC_{\sfp}^{d})=0$
is a direct consequence of \eqref{e:BQ} since $\kappa(c_h(K),\cdot)$ is concentrated on $\Ih(\sfp,c_h(K))$ and $c_h$ is isometry covariant. For the asserted $\Ih(\sfp)$ invariance we need
to take into account that
\begin{align*}
\int \I\{\varrho\circ\tau\in\cdot\}\,\kappa(x,\dint\varrho)=\kappa(x,\cdot),
\quad \tau \in \Ih(\sfp),\,x\in\BH^d,
\end{align*}
which is implied by \eqref{e:kappa} and
the invariance of the Haar measure $\kappa(\sfp,\cdot)$.

Next we verify that \eqref{eqdisintegration} holds if $\BQ$ is chosen
as at \eqref{e:BQ}. Using the right invariance of $\lambda$, the fact
that $\kappa(c_h(K),\cdot)$ is a probability measure, Fubini's
theorem, the isometry invariance of $\Lambda$, the inversion
invariance of $\lambda$, and  \eqref{eq:c_h_property}, and finally
again Fubini's theorem, we get
\begin{align*}
    &\gamma\int_{\cC^d_\sfp}\int_{\Ih}  \I\{\varphi K\in\cdot\} \, \lambda(\dint \varphi)\, \BQ(\dint K)\\
    &=\int_{\cC^d}\int_{\Ih}\int_{\Ih}\I\{\varphi\circ \sigma^{-1} K\in\cdot\}\,\lambda(\dint\varphi)\, w(c_h(K))\,\kappa(c_h(K),\dint \sigma)\, \Lambda(\dint K)\\
    &=\int_{\cC^d}\int_{\Ih}\int_{\Ih}\I\{\varrho  K\in\cdot\}\,\lambda(\dint\varrho)\, w(c_h(K))\,\kappa(c_h(K),\dint \sigma)\, \Lambda(\dint K)\\
    &=\int_{\cC^d}\int_{\Ih} \I\{\varrho  K\in\cdot\}\,\lambda(\dint\varrho)\, w(c_h(K))\, \Lambda(\dint K)\\
    &=\int_{\Ih}\int_{\cC^d} \I\{\varrho  K\in\cdot\}w(c_h(K))
    \, \Lambda(\dint K)\,\lambda(\dint\varrho)\\
    &=\int_{\Ih}\int_{\cC^d}\I\{K\in\cdot\} w(c_h(\varrho^{-1}K))
    \, \Lambda(\dint K)\,\lambda(\dint\varrho)\\
    &=\int_{\cC^d} \I\{K\in\cdot\}\int_{\Ih}w(\varrho c_h( K))
    \, \lambda(\dint\varrho)\, \Lambda(\dint K)\\
    &=\int_{\cC^d} \I\{K\in\cdot\} \, \Lambda(\dint K),
\end{align*}
where we also used that $\Lambda$ is $\sigma$-finite.

Suppose that $\BQ^*$ is a probability measure on $\cC^d$ which  is concentrated on $\cC_\sfp^{d}$, invariant under isometries fixing $\sfp$, and satisfies  \eqref{eqdisintegration}.   Using   Fubini's theorem,   $c_h(K)=\sfp$ for $\BQ^*$-a.e.~$K\in\cC^{d}$, \eqref{eq:2.neu1} with $x=\sfp$, the right invariance of $\lambda$, and the  $\Ih(\sfp)$ invariance  of $\BQ^*$, we get
\begin{align}\label{eqdisintegration0}
&\int_{\cC^{d}}\int_{\Ih}  \I\{(\varrho,\varrho^{-1}K)\in\cdot\}\,\kappa(c_h(K),\dint \varrho)\,\Lambda(\dint K)\nonumber\\
&
=\gamma\int_{\cC^{d}}\int_{\Ih}\int_{\Ih}  \I\{(\varrho,\varrho^{-1}\circ \sigma K)\in\cdot\}\,\kappa(c_h(\sigma K),\dint \varrho)\,\lambda(\dint\sigma)\,\BQ^*(\dint K)\nonumber\\
&
=\gamma\int_{\cC^{d}}\int_{\Ih}\int_{\Ih}  \I\{(\sigma\circ\varrho,(\sigma\circ\varrho)^{-1}\circ \sigma K)\in\cdot\}\,\kappa(\sfp,\dint \varrho)\,\lambda(\dint\sigma)\,\BQ^*(\dint K)\nonumber\\
&
=\gamma\int_{\cC^{d}}\int_{\Ih}\int_{\Ih}  \I\{(\sigma, \varrho^{-1}  K)\in\cdot\}\,\lambda(\dint\sigma)\,\kappa(\sfp,\dint \varrho)\,\BQ^*(\dint K)\allowdisplaybreaks\nonumber\\
&
=\gamma\int_{\Ih}\int_{\Ih} \int_{\cC^{d}} \I\{(\sigma,   K)\in\cdot\}\,\BQ^*(\dint K)\,\lambda(\dint\sigma)\,\kappa(\sfp,\dint \varrho) \allowdisplaybreaks\nonumber\\
&
=\gamma\int_{\cC^{d}}\int_{\Ih}   \I\{(\sigma,   K)\in\cdot\}\,\lambda(\dint\sigma)\,\BQ^*(\dint K) .
\end{align}
Applying \eqref{eqdisintegration0} to the function $f(\sigma,K)\defeq w(\sigma c_h(K))\I\{K\in \cdot\}$, we conclude that
  $\BQ^*=\BQ$.
\end{proof}

\begin{remark}{\rm 
Let $\overline{\kappa}$ be an arbitrary probability kernel from $\BHd$ to $\Ih$. If $\overline{\kappa}(x,\cdot)$ is concentrated on $\Ih(\sfp,x)$ for $\mathcal{H}^d$-almost all $x\in\BHd$, then the measure 
\begin{align}\label{e:BQ2a}
\overline{\BQ} = \frac{1}{\gamma} \int_{\cC^{d}} \int_{\Ih}
 \1 \{\varrho^{-1} K \in \cdot\}w(c_h(K))\, \overline{\kappa}(c_h(K),\dint \varrho) \, \Lambda( \dint K)
\end{align}
is equal to $\BQ$. Hence, while $\kappa$ as given at \eqref{e:kappa} is a natural choice, any other kernel $\overline{\kappa}$ leads to the same measure. In fact, we can choose $\overline{\kappa}(x,\cdot)=\delta_{\varrho_x}$ provided that $\varrho_x\in \Ih(\sfp,x) $ is a measurable function of $x\in \BHd$.

In order to show that $\overline{\BQ}=\BQ$, we start from \eqref{e:BQ2a} and replace $\Lambda$ by the right side of \eqref{eqdisintegration}  (using also that $\BQ$ is concentrated on $\cC_\sfp^{d}$), then we use that $c_h(\varphi K)=\varphi \sfp$ for $K\in \cC_{\sfp}^d$, \eqref{disintlambda}, \eqref{e:kappa}, the fact that $\varrho^{-1}\circ \varrho_x\in \Ih(\sfp)$ if $\varrho\in \Ih(\sfp,x)$ together with Fubini's theorem and the $\Ih(\sfp)$ invariance of $\kappa(\sfp,\cdot)$, \eqref{eq:integralone}, and again Fubini's theorem and the $\Ih(\sfp)$ invariance of $\BQ$. Thus we obtain
\begin{align*}
\overline{\BQ}&=\int_{\cC_\sfp^{d}}\int_{\Ih}\int_{\Ih} 
\1\{\varrho^{-1}(\varphi K)\in\cdot\}\, 
\overline{\kappa}(c_h(\varphi K),\dint \varrho)
w(c_h(\varphi K))\, \lambda(\dint \varphi)\, \BQ(\dint K)\\
&=\int_{\cC_\sfp^{d}}\int_{\Ih}\int_{\Ih} 
\1\{\varrho^{-1}\circ \varphi K\in\cdot\} \,
\overline{\kappa}( \varphi(\sfp),\dint \varrho)
w(\varphi(\sfp))\, \lambda(\dint \varphi)\, \BQ(\dint K)\\
&=\int_{\cC_\sfp^{d}}\int_{\mathbb{H}^d}\int_{\Ih} \int_{\Ih}
\1\{\varrho^{-1}\circ \varphi K\in\cdot\} \,
\overline{\kappa}( \varphi(\sfp),\dint \varrho)
w(\varphi(\sfp))\, \kappa(x,\dint\varphi)\, \mathcal{H}^d(\dint x)\, \BQ(\dint K)\\
&=\int_{\cC_\sfp^{d}}\int_{\mathbb{H}^d}\int_{\Ih(\sfp)} 
\int_{\Ih}\1\{(\varrho^{-1}\circ \varrho_x)\circ \varphi K\in\cdot\} \, 
\overline{\kappa}( x,\dint \varrho)
w(x)\, \kappa(\sfp,\dint\varphi)\, \mathcal{H}^d(\dint x)\, \BQ(\dint K)\\
&=\int_{\cC_\sfp^{d}}\int_{\mathbb{H}^d }
\int_{\Ih}\int_{\Ih(\sfp)}\1\{ \varphi K\in\cdot\}  \, \kappa(\sfp,\dint\varphi)\, 
\overline{\kappa}( x,\dint \varrho)
w(x)\, \mathcal{H}^d(\dint x)\, \BQ(\dint K)\\
&=\int_{\cC_\sfp^{d}}\int_{\mathbb{H}^d }
 \int_{\Ih(\sfp)}\1\{ \varphi K\in\cdot\}  \, \kappa(\sfp,\dint\varphi) 
w(x)\, \mathcal{H}^d(\dint x)\, \BQ(\dint K)\\
&=
 \int_{\Ih(\sfp)}\int_{\cC_\sfp^{d}}\1\{ \varphi K\in\cdot\}  
\, \BQ(\dint K)\, \kappa(\sfp,\dint\varphi) =\BQ.
\end{align*}
In particular, it thus follows that $\overline{\BQ}$ is invariant under isometries fixing $\sfp$. 
}    
\end{remark}

The probability measure $\BQ$ given in \eqref{e:BQ}
is said to be the distribution of the {\em typical particle}.
It is convenient to introduce a random element
$\TG$ with this distribution. Intuitively,  the typical particle
is obtained by choosing one of the particles of $\zeta$ at random
and moving it to $p$. In contrast to the Euclidean case, a canonical selection of an isometry mapping $\sfp$ to  $c_h(K)$ is not available 
in the hyperbolic setting. Therefore we  use the probability kernel
$\kappa(c_h(K),\cdot)$ for a uniform random choice of an isometry mapping $\sfp$ to $c_h(K)$.

In the following theorem, we state equivalent conditions for a particle process to have a locally finite intensity measure. For this, we denote by $\BB(A,r)\defeq\{x\in\BHd\colon d_h(A,x)\le r\}$  the parallel set of  $\varnothing\neq A\subseteq \BHd$ at distance $r\ge 0$, where $d_h(A,x)\defeq\inf\{d_h(a,x)\colon a\in A\}$. Recall that $\Vol$ denotes  the volume functional, which is given by the Riemannian volume measure or the Hausdorff measure on $\BHd$. We introduce the short notation
$$
\overline{\Vol}(A)\defeq\Vol(\mathbb{B}(A,1)),
$$
which will be used throughout this work.

\begin{theorem}\label{thmchar}
Let $\zeta$ be a stationary particle process on $\BHd$. Let $\gamma\in (0,\infty]$ be defined as at \eqref{eq:gammadef} (for a given center function  $c_h$). Then the following conditions are equivalent.
\begin{enumerate}
\item[{\rm (a)}] $\Lambda([C])<\infty$ for all compact sets $C\subset\mathbb{H}^d$.
\item[{\rm (b)}] $\Lambda([\mathbb{B}_1])<\infty$ (instead of $\mathbb{B}_1$ any compact set with nonempty interior may be chosen).
\item[{\rm (c)}] $\gamma<\infty$ and $ \E\overline{\Vol}(\TG)<\infty$.
\end{enumerate}
\end{theorem}

\begin{proof}
(a) clearly implies (b). For the converse, assume that (b) holds. If $C\subset\mathbb{H}^d$  is compact, then there exist isometries $\varrho_1,\ldots,\varrho_m\in\Ih$ such that $C\subseteq \bigcup_{i=1}^m\varrho_i \mathbb{B}_1$. Then
$$
[C]\subseteq \bigcup_{i=1}^m[\varrho_i \mathbb{B}_1]=\bigcup_{i=1}^m \varrho_i[ \mathbb{B}_1],
$$
hence $\Lambda([C])<\infty$   follows from the isometry invariance of $\Lambda$ and subadditivity.

Next we show that (c) implies (b). If $\gamma<\infty$, then Theorem \ref{thmdisintegration} holds. Hence
\begin{align}
\Lambda([\mathbb{B}_1])&=\int_{\cC^d} \I\{\mathbb{B}_1\cap K\neq \varnothing\}\, \Lambda(\dint K)=\gamma\E\int_{\Ih} \I\{\mathbb{B}_1\cap \varrho \TG\neq\varnothing\}\, \lambda(\dint \varrho)\nonumber\\
&=\gamma\E\int_{\Ih} \I\{\varrho\mathbb{B}_1\cap \TG\neq\varnothing\}\, \lambda(\dint \varrho)
=\gamma\E\int_{\Ih} \I\{\varrho\msfp \in \mathbb{B}(\TG,1)\}\, \lambda(\dint \varrho)\nonumber\\
&=\gamma\E \mathcal{H}^d( \mathbb{B}(\TG,1))=\gamma \E\overline{\Vol}(\TG)<\infty.\label{eq:doppelstern}
\end{align}
Here we used the inversion invariance of $\lambda$ for the third equality.

Finally we show that (a) implies (c). If $\mathbb{B}$ denotes again a ball of volume $1$, then
\begin{align}
\Lambda([\mathbb{B}(\mathbb{B},1)])&=\int_{\cC^d} \I\{K\cap \mathbb{B}(\mathbb{B},1)\neq \varnothing\}\, \Lambda(\dint K)=\int_{\cC^d} \I\{ \mathbb{B}(K,1)\cap\mathbb{B}\neq \varnothing\}\, \Lambda(\dint K)\notag\\
&\ge \int_{\cC^d} V_d(\mathbb{B}(K,1))^{-1} V_d(\mathbb{B}(K,1)\cap \mathbb{B}) \, \Lambda(\dint K)\notag\\
&= \int_{\cC^d}\int_{\Ih}V_d(\mathbb{B}(K,1))^{-1}\I\{\varrho c_h(K)\in \mathbb{B}(K,1)\cap \mathbb{B}\}\, \lambda(\dint \varrho)\, \Lambda(\dint K)\allowdisplaybreaks\notag\\
&= \int_{\cC^d}\int_{\Ih}V_d(\mathbb{B}(K,1))^{-1}\I\{ c_h(\varrho K)\in \mathbb{B}(K,1) \}\I\{ c_h(\varrho K)\in \mathbb{B} \}\, \lambda(\dint \varrho)\, \Lambda(\dint K)\allowdisplaybreaks\notag\\
&= \int_{\Ih}\int_{\cC^d} V_d(\mathbb{B}(\varrho^{-1} K,1))^{-1}\I\{ c_h(  K)\in \mathbb{B}(\varrho^{-1} K,1) \}\I\{ c_h(  K)\in \mathbb{B} \}\, \Lambda(\dint K)\, \lambda(\dint \varrho)\allowdisplaybreaks\notag\\
&= \int_{\cC^d}\int_{\Ih} V_d(\mathbb{B}( K,1))^{-1}\I\{ \varrho c_h(  K)\in \mathbb{B}( K,1) \}\I\{ c_h(  K)\in \mathbb{B} \}\, \lambda(\dint \varrho)\, \Lambda(\dint K)\notag\\
&= \int_{\cC^d}\I\{ c_h(  K)\in \mathbb{B} \}\, \Lambda(\dint K)=\gamma, \notag
\end{align}
hence (a) implies that  $\gamma<\infty$. Therefore, relation \eqref{eq:doppelstern} yields that $\E\overline{\Vol}(\TG)<\infty$.
\end{proof}

We finish this subsection by remarking that the intensity measure $\Lambda$  of  a stationary particle process $\zeta$ on $\BHd$ with intensity $\gamma \in (0,\infty)$ (with respect to a given center function) is diffuse. We also show that $\gamma$ is in fact independent of the chosen (isometry covariant) center function.

\begin{lemma}\label{l:2.3}
Let $\zeta$ be a stationary particle process on $\BHd$ with intensity $\gamma \in (0,\infty)$ (with respect to a given center function $c_h$) and intensity measure $\Lambda=\E\zeta$. Then the following is true.
\begin{enumerate}
    \item[{\rm (a)}] $\Lambda$ is diffuse.
    \item[{\rm (b)}] $\gamma$ is independent of the chosen center function.
\end{enumerate}
\end{lemma}

\begin{proof}
(a) Let $K\in\cC^d$. Then Theorem \ref{thmdisintegration} implies
\begin{align*}
    \Lambda(\{K\})&=\gamma \int_{\cC^d}\int_{\Ih} \I\{\varrho G=K\}\, \lambda (\dint \varrho)\, \BQ(\dint G)\\
    &\le \gamma \int_{\cC^d}\int_{\Ih} \I\{ \varrho \sfp=c_h(K)\}\, \lambda (\dint \varrho)\, \BQ(\dint G)=0,
\end{align*}
where we used that if 
$\varrho G=K$ 
and $G\in \cC^d_{\sfp}$, then $\varrho \sfp=\varrho c_h(G)=c_h(\varrho G)=c_h(K)$ as well as \eqref{eq:hdlambda} to see that the inner integral is zero.

(b) Let $c'$ be another center function with corresponding intensity $\gamma'\in (0,\infty)$ and typical particle $\BQ'$. Let $w\colon\BH^d\to [0,\infty)$ be a measurable function satisfying \eqref{eq:integralone}.
 Using first \eqref{e:BQ} for $c'$, $\gamma'$, and $\BQ'$, and then \eqref{eqdisintegration}    with respect to the center function $c$ and corresponding intensity $\gamma$ and distribution $\BQ$ of the typical particle, we get
\begin{align*}
    \gamma'&= \gamma'\BQ'(\cC^d)=\int_{\cC^d}w(c'(K))\, \Lambda(\dint K)
    =\gamma\int_{\cC^d_{\sfp}}\int_{\Ih} w(c'(\varrho G))\, \lambda(\dint \varrho)\, \BQ(\dint G)\\
    &=\gamma\int_{\cC^d_{\sfp}}\int_{\Ih} w(\varrho c'(  G))\, \lambda(\dint \varrho)\, \BQ(\dint G)
    =\gamma,
\end{align*}
where we used that $\int_{\Ih} w(\varrho c'(G))\, \lambda(\dint \varrho)=1$ for $G\in \cC^d_{\sfp}$.
\end{proof}

\begin{remark} {\rm The distribution of the typical particle depends on the choice of the center function.}
\end{remark}

\subsection{Boolean models}\label{sec:2.3neu}

 Based on the results from Subsection \ref{subsec2.2}, we introduce stationary
Boolean models in hyperbolic space as the union sets of stationary
Poisson particle processes. Then we provide an alternative viewpoint
according to which the Boolean model is derived from a stationary
Poisson process in the space of isometries by independent
marking with centered particles.

Let $\eta$ be a stationary Poisson particle process in $\BH^d$.
This means that $\eta$ is both a stationary particle process
and a Poisson process. The intensity measure $\Lambda\defeq\E \eta$
is then locally finite, that is $\Lambda([C])<\infty$ for
each compact set $C\subset\BH^d$.
Then each of the equivalent conditions of Theorem
\ref{thmchar} is satisfied. In particular $\eta$ has a finite
intensity $\gamma$. To avoid trivialities we assume throughout
that $\gamma>0$.  We fix a center function $c_h$ and let
$\TG$ denote a typical particle with distribution $\BQ$;
see Theorem \ref{thmdisintegration}.

Define
$$
Z\defeq\bigcup_{K\in\eta}K,
$$
where $K\in\eta$ means that $K$ is in the
support of $\eta$. Since $\Lambda$ is diffuse,
$\eta$ is simple (i.e., without multiple points).
Then $Z$ is a {\em random closed set}.
Indeed, just as in the Euclidean case
$Z$ is a measurable mapping from $\Omega$ into the space
of closed subsets of $\BH^d$, equipped with the Fell topology;
see \cite[Section 12.2]{SW08}. Moreover, since $\eta$ is stationary, $Z$ is
stationary as well, that is $\varrho Z\overset{d}{=}Z$ for $\varrho\in\Ih$.
We refer to $Z$ as a stationary (hyperbolic) {\em Boolean model}.
Let $C\subset \mathbb{H}^d$ be a compact set. It follows from
the definition of $Z$ and the disintegration \eqref{eqdisintegration} that
\begin{align}\label{eq:argument}
\BP(Z\cap C= \varnothing)
=\exp\left(-\gamma\int_{\cC^d_{\sfp}}\int_{\Ih}\I\{\varrho G\cap C\neq\varnothing\}\, \lambda(\dint\varrho)\,
\BQ(\dint G)\right).
\end{align}
As in the Euclidean case this determines the {\em capacity functional}
$C\mapsto \BP(Z\cap C\neq\varnothing)$ of $Z$ and hence the distribution of $Z$. Conversely, the distribution of $Z$ determines the intensity measure $\Lambda$ of the underlying Poisson particle process $\eta$ and therefore the distribution of $\eta$ (use \cite[Lemma 2.3.1]{SW08} and note that up to the minus sign the expression in the argument of the exponential function in \eqref{eq:argument} is $\Lambda(\mathcal{F}_C)$).

\begin{remark}\rm In the Euclidean case, a Boolean model is mostly defined
in terms of an independent marking of a stationary Poisson process on $\R^d$.
In the present hyperbolic setting we can proceed as follows.
Suppose we are given $\gamma\in(0,\infty)$ and a
probability measure $\BQ$  on $\cC^d_\sfp$ which is invariant under
isometries fixing $\sfp$ and such that
$\int\overline{\Vol}(K)\,\BQ(\dint K)<\infty$.
Let $\Upsilon$ be a stationary Poisson process
in $\Ih$ with intensity measure $\gamma\cdot \lambda$.
Let $\Psi$ be an
independent $\BQ$-marking of $\Upsilon$; see \cite[Chap.~5]{LP18}.
By the marking theorem $\Psi$ is a
Poisson process on $\Ih\times \cC^d_\sfp$ with intensity measure
$\gamma \lambda\otimes \BQ$. The map $T\colon\Ih\times \cC^d_\sfp\to \cC^d$
given by $(\varrho,G)\mapsto \varrho G$ is measurable. Let
$\widetilde\eta\defeq T(\Psi)$ denote the image measure of $\Psi$ under
$T$. By the mapping theorem (\cite[Theorem 5.1]{LP18})
$\widetilde\eta$ is a Poisson process on $\cC^d$ with
intensity measure
$$
\gamma\int_{\cC^d_\sfp}\int_{\Ih}\I\{\varrho G\in\cdot\}\, \lambda(\dint\varrho)\, \BQ(\dint G)
=\Lambda.
$$
Therefore $\widetilde\eta\overset{d}{=}\eta$, so that
$$
\widetilde Z\defeq \bigcup_{(\varrho,K)\in\Psi}\varrho K
$$
has the same distribution as $Z$.
\end{remark}

\section{Geometric functionals}\label{sec:GF}

For the analysis of the Boolean model $Z$, geometric functionals are applied to the intersection of $Z$ with observation windows. In this section we deal with general properties requested of such functionals and with a class of examples that arise from a Steiner formula in hyperbolic space. In the following, we will focus on Boolean models $Z$ which are induced by stationary particle processes $\eta$ that are supported by convex particles, that is, $\BP(\eta(\cC^d\setminus\cK^d)=0)=1$ (if not stated otherwise).

\medskip

Recall that $\cK^d$ denotes the set  of compact convex subsets of
$\mathbb{H}^d$. All common notions of convexity coincide in hyperbolic space (but may differ in more general Riemannian spaces); see \cite[Section 1]{Waltersurvey1981}, \cite[p.~222]{Sakai}, \cite{Ballmann}. Together with the Hausdorff metric, $\cK^d$ is a complete, separable metric space in which the empty set is an isolated point. Let $\cR^d$ denote the set of all finite unions of sets from $\cK^d$ (the \emph{convex ring}).  
We consider a function $\phi:\cR^d\to\mathbb{R}$. It is said to be {\em invariant} under all isometries of  $\BHd$ if
$$
\phi(\varrho A) = \phi(A) \quad \text{for }  A\in\mathcal{R}^d, \varrho\in\calI_d.
$$
We say that $\phi$ is {\em additive} (a valuation) if $\phi(\varnothing)=0$ and
\begin{equation}\label{eq:propval}
\phi(U\cup V) =\phi(U) + \phi(V) - \phi(U\cap V)
\end{equation}
for all $U,V\in\cR^d$. Similarly, a function $\phi:\cK^d\to\R$ is said to be additive (a valuation) if \eqref{eq:propval} holds
for all $U,V\in\cK^d$ with $U\cup V\in\cK^d$ and $\phi(\varnothing)=0$. If  $\phi\colon \mathcal{R}^d\to\R$ is additive and the restriction  of $\phi$ to $\cK^d$ is isometry invariant, then $\phi$ is isometry invariant on $\cR^d$. In order to check that an additive map $\phi:\mathcal{R}^d\to\mathbb{R}$ is measurable, it is sufficient to show that the restriction of $\phi$ to $\cK^d$ is measurable (compare \cite[Theorem 14.4.4]{SW08}).  Moreover, we call $\phi$ {\em locally bounded} if
\begin{equation}\label{eqn:M_phi}
M(\phi) \defeq \sup\{|\phi(K)|\colon K\in\mathcal{K}^d, \varrho\in\mathcal{I}_d, K\subseteq \varrho\mathbb{B}_1\}<\infty.
\end{equation}
In the following,  a measurable, isometry invariant, additive, and locally bounded function $\phi:\mathcal{R}^d\to\mathbb{R}$ will be called a {\em geometric functional} on $\mathcal{R}^d$. Finally we remark that if $\phi:\cK^d\to\R$ is additive and continuous, then $\phi$ has an additive extension to $\cR^d$. The preceding statements are well known in Euclidean space (see \cite[Chapter 14.4]{SW08}) and can be transferred to hyperbolic space by considering the Beltrami--Klein (projective disc) model (as explained in \cite[Chapter 6]{Ratcliffe}).

\medskip

Prominent examples of geometric functionals  are volume and surface area.  
These are just two geometric functionals from a list of $d+1$ natural basic functionals that arise from a common principle, the {\em Steiner formula} in hyperbolic space. Other functionals that are defined via   Crofton type integrals can be expressed as linear combinations of these natural basic functionals. Whereas in Euclidean space Hadwiger's celebrated characterization theorem shows that all geometric and continuous  functionals are linear combinations of these basic functionals,  it is unknown whether a corresponding fact holds in hyperbolic space with $d\ge 3$ (see \cite[Section 3]{Klain2006} for the hyperbolic plane). If continuity is relaxed to upper semicontinuity, then other geometric functionals are known (in hyperbolic space  the floating area considered in \cite[Theorem 1.3]{BW2018} is an interesting example).

To describe the Steiner formula, let $\varnothing\neq A\in\cK^d$.  As in Section \ref{sec:2}, let $\BB(A,r)=\{x\in\BHd\colon d_h(A,x)\le r\}$
 be the parallel set of  $A$  at distance $r\ge 0$. For $j\in\{0,\ldots,d-1\}$, let
$$
l_{d,j}(r)\defeq\binom{d-1}{j}\int_0^r \cosh^j(t)\sinh^{d-1-j}(t)\, \dint t,\quad r\ge 0.
$$
It is easy to check that the functions $l_{d,0},\ldots,l_{d,d-1}$ are linearly independent.
 The Steiner formula \eqref{eq:Steiner} involves the existence of (uniquely determined) continuous, isometry invariant, and additive functionals $V_0,\ldots,V_{d-1}$ on $\cK^d$
such that
\begin{align}\label{eq:Steiner}
V_d(\BB(A,r))=V_d(A)+\sum_{j=0}^{d-1}l_{d,j}(r)V_j(A),\quad r\ge 0,
\end{align}
for all $\varnothing \neq A\in\cK^d$. 
As a consequence, for $i\in\{0,\ldots,d\}$  the functionals $V_i$ are locally bounded. Moreover, $V_i$ is  increasing with respect to set inclusion \cite[Corollary 9]{Sol2005}.
A more general local Steiner formula involving curvature measures  $C_j(A,\cdot)$, $j\in\{0,\ldots,d-1\}$, is established in \cite[Theorem 2.7]{Kohlmann1991}. These measures are Borel measures on $\mathbb{H}^d$ that are concentrated on $\partial A$, weakly continuous \cite[Theorem 2.4]{Kohlmann1994}, additive with respect to $A$ (which can be verified as in the Euclidean framework or by using properties of the unit normal cycle associated with $A$), and isometry covariant in the sense that $C_j(\varrho A,\varrho B)=C_j(A,B)$ for  Borel sets  $B\subseteq\mathbb{H}^d$ and  $\varrho\in\Ih$.  In the following, we will only  use  the top order curvature measure $C_{d-1}(A,\cdot)$. These properties of the curvature measures immediately imply the corresponding  properties of the functionals $V_j(A)=C_j(A,\mathbb{H}^d)$. As usual, we define $V_i(\varnothing):=0$ and $C_i(\varnothing,\cdot)$ as the zero measure for $i\in\{0,\ldots,d\}$.

As in the Euclidean case, $V_0,\ldots,V_{d}$ are called
{\em intrinsic volumes} (the normalization differs from the one in Euclidean space and the functionals are not ``intrinsic''). Alternatively, the functionals are called {\em total curvature integrals} of $A$. They can be additively extended to the convex ring, which can be shown as in \cite{Groemer1978} (see in particular the comment in \cite[Section 5]{Groemer1978} on spherical space).
If $A$ is convex and has nonempty interior (for which we write  $A^\circ\neq\varnothing$), then
\begin{equation}\label{eq:3:curvd-1}
C_{d-1}(A,\cdot)=\mathcal{H}^{d-1}(\partial A\cap\cdot),
\end{equation}
hence
\begin{equation*}
V_{d-1}(A)=\mathcal{H}^{d-1}(\partial A)
\end{equation*}
(in the Euclidean case the corresponding relation involves an additional factor $1/2$ on the right side, due to the different normalization of $V_{d-1}$) and $V_0(A)$ is the (total) Gauss curvature integral of $A$ (in Euclidean space the normalization of $V_0$ is chosen such that $V_0$ is equal to the Euler characteristic $\chi$, in hyperbolic space the relation between $V_0$ and $\chi$ is described below).
If $A\subset \BHd$ is compact, convex, and at most $(d-1)$-dimensional, then
\begin{equation}\label{eq:3:curvd-1lower}
C_{d-1}(A,\cdot)=2\mathcal{H}^{d-1}(\partial A\cap\cdot),
\end{equation}
and correspondingly we have $V_{d-1}(A)=2\mathcal{H}^{d-1}(\partial A)$. If $A$ has dimension at most $d-2$, then we get $C_{d-1}(A,\cdot)=\mathcal{H}^{d-1}(\partial A\cap\cdot)=0$.

Let $r,s\ge 0$. From $\mathbb{B}_{r+s}=\mathbb{B}(\mathbb{B}_r,s)$,  \eqref{eq:polarcoord}, and the Steiner formula \eqref{eq:Steiner}, we obtain
\begin{equation}\label{eq:neua}
\omega_d\int_0^{r+s}\sinh^{d-1}(t)\, \dint t=V_d(\mathbb{B}_r)+\sum_{j=0}^{d-1}\ell_{d,j}(s)V_j(\mathbb{B}_r).
\end{equation}
Taking in \eqref{eq:neua} the derivative with respect to $s$, using the hyperbolic angle sum formula for $\sinh$ on the left-hand side and the linear independence of the functions $s\mapsto \ell'_{d,i}(s)$, $i=0,\ldots,d-1$, we obtain
$$
V_j(\mathbb{B}_r)
=\omega_d\, \cosh^{d-1-j}(r)\sinh^j(r),\quad j\in\{0,\ldots,d-1\}.
$$
As a consequence, for instance by an application of l'Hospital's rule, we get
\begin{equation}\label{eq:limit2}
\lim_{r\to\infty}\frac{V_j(\mathbb{B}_r)}{V_d(\mathbb{B}_r)}=d-1,\quad j\in\{0,\ldots,d-1\},
\end{equation}
which extends \eqref{eq:limit1} and is in striking contrast to the Euclidean situation. In fact, the strict lower bound
$$
\frac{V_j(A)}{V_d(A)}>d-1
$$
for $A\in\cK^d$ with nonempty interior and $j\in\{0,\ldots,d-1\}$ is established in \cite[Corollary 3.2]{GalSol2005}.

The {\em Euler characteristic} $\chi$ is another important functional
on the convex ring $\cR^d$. It is the additive functional
determined by $\chi(\varnothing)=0$ and
$\chi(A)=1$ for a nonempty  $A\in\cK^d$.
The Euler characteristic $\chi$ can be expressed as a linear combination of some of the intrinsic
volumes.
For $d=2$, the relation
\begin{align}\label{eq:chi2}
2\pi\chi(A)=V_0(A)-V_2(A)
\end{align}
holds for $A\in \mathcal{R}^2$.
Replacing $V_0$ by $\chi$ in the $2$-dimensional Steiner formula, for $A\in \mathcal{K}^2$ we get
$$
V_2(\BB(A,r))=2\pi(\cosh(r)-1)\chi(A)+\sinh(r)V_1(A) +\cosh(r)V_2(A),\qquad r\ge 0.
$$

It will be convenient to use a suitable renormalization of the intrinsic volumes, since this will simplify integral geometric formulas.
For this we put
\begin{align}\label{eq:vk0}
V_k^0(A)\defeq \frac{\omega_{d+1}}{\omega_{k+1}\omega_{d-k}}\,\binom{d-1}{k} V_k(A),\quad k\in\{0,\ldots,d-1\},
\end{align}
and $V_d^0(A)\defeq V_d(A)$ for $A\in\mathcal{R}^d$. 
Using the Legendre
duplication formula and setting $\kappa_0:=1$, we can rewrite the coefficient on the right-hand
side of \eqref{eq:vk0} as
\begin{equation}\label{eq:sec2rel}
 \frac{\omega_{d+1}}{\omega_{k+1}\omega_{d-k}}\, \binom{d-1}{k}=\frac{\pi}{d}\frac{\kappa_k\kappa_{d-1-k}}{\kappa_d}.
\end{equation}
For $d=2$ we get $V^0_i=V_i$ for $i=0,1,2$. In arbitrary dimension $d\ge 2$ and for $A\in\mathcal{R}^d$, the relation between the Euler characteristic and the intrinsic volumes can be expressed  in the form
\begin{equation}\label{eq:euler}
\chi(A)=\frac{2}{\omega_{d+1}}\sum_{l\ge 0}^{\lfloor \frac{d}{2}\rfloor} \epsilon ^l\, V_{2l}^0(A),
\end{equation}
upon setting $\epsilon=-1$; relation \eqref{eq:euler} can be deduced from the
more general result \cite[Theorem 3.1]{Kohlmann1991} by introducing
the current normalization.  The Euclidean case is obtained for
$\epsilon=0$, whereas $\epsilon=1$ yields the relation in the
spherical case, if the normalization of $V_i$ is determined by
\eqref{eq:Steiner} with $(\cosh t,\sinh t)$ replaced by $(1,t)$ in the
Euclidean case and by $(\cos t, \sin t)$ in the spherical case.

The main point of introducing the renormalization is the simplification that results for the statement of  the {\em kinematic formulas} which hold for the intrinsic volumes (a first version is given in \cite{Santalo} with different notation, see also the arguments in \cite{Fu2011,Fu2012}). In terms of $V_0,\ldots,V_d$ and for $A,B\in\mathcal{R}^d$, the kinematic formulas read
$$
\int V_d(A\cap \varrho B)\, \lambda(\dint \varrho)=V_d(A)V_d(B),
$$
$$
\int V_{d-1}(A\cap \varrho B)\, \lambda(\dint \varrho)=V_{d-1}(A)V_d(B)+V_d(A)V_{d-1}(B)
$$
and
\begin{align}
\int V_k(A\cap \varrho B)\, \lambda(\dint \varrho)&=V_k(A)V_d(B)+V_d(A)V_k(B)\label{eq3:kin1}\\
&\quad +\sum_{j=k+1}^{d-1}\frac{1}{d\kappa_d}\binom{d-k}{j-k}\frac{\kappa_{d-k}}{\kappa_k}
\frac{\kappa_{j}}{\kappa_{d-j}}\frac{\kappa_{d+k-j}}{\kappa_{j-k}}V_j(A)V_{d+k-j}(B),\nonumber
\end{align}
for $k\in\{0,\ldots,d-2\}$. The formulas for $V_d$ and $V_{d-1}$ can be subsumed to \eqref{eq3:kin1} if this formula is suitably  interpreted. 
For a more general version of the kinematic formula, see \cite[Theorem, p. 1152]{Fu1990} and the discussion in \cite[Section 2]{BFS2014}.
Using the modified functionals $V_k^0$, the kinematic formulas turn into
 \begin{align}\label{eq:kinematic}
\int V_k^0(A\cap \varrho B)\, \lambda(\dint\varrho)=\sum_{i+j=d+k}V_i^0(A)V_j^0(B),\quad k\in\{0,\ldots,d\},
\end{align}
where the summation is extended over all $i,j\in \{0,\ldots,d\}$ such that $i+j=d+k$.
In view of iterations of these formulas that will be needed, the benefit of working with the renormalized functionals is apparent.

From the preceding formulas, a kinematic formula for the Euler characteristic can be deduced.

\section{Mean value formulas}\label{sec:4}

Throughout this section, $\eta$ is a stationary Poisson particle process concentrated on convex particles and $Z$ denotes the induced stationary Boolean model. Some of our general mean value formulas involve an integration over the space of horoballs (see Section \ref{sec:2.1}), with respect to an isometry invariant measure on this space.
Recall that for $u\in \mathbb{S}^{d-1}_{\msfp}$ and $t\in\R$,  we write $\horo_{u,t}\subset\mathbb{H}^d$ for the closed and convex horoball for which $\exp_{\msfp}(tu)\in\partial \horo_{u,t}$  is the unique point on $\partial \horo_{u,t}$ which minimizes the distance to $\sfp$ and that satisfies $\exp_{\msfp}(su)\in  \horo_{u,t}$ for $s\ge t$.  Moreover, $\sfp\in \horo_{u,t}$ if and only if $t\le 0$. The set of horoballs of $\BH^d$ is an isometry invariant Borel subset of the space of all closed subsets of $\BH^d$ (with respect to the Fell topology), which will be denoted by $\hb$. 
We define an infinite (but $\sigma$-finite) 
measure $ {\mu}_{\rm hb}$ on Borel sets of horoballs by
\begin{equation}\label{eq:muhb}
 {\mu}_{\rm hb}\defeq\frac{d-1}{\omega_d}\int_{\mathbb{S}^{d-1}_{\msfp}}\int_\R\1\{\horo_{u,t}\in\cdot\}e^{(d-1)t}\, \dint t\, \mathcal{H}^{d-1}_{\msfp}(\dint u).
\end{equation}
It is shown in Proposition 2.2 and the subsequent remark (for $\lambda=1$) in \cite{Solanes2005a} that  ${\mu}_{\rm hb}$
is the unique (up to a constant factor) isometry invariant measure on $\hb$.

Recall that a function $\phi:\mathcal{R}^d\to\mathbb{R}$ is said to be geometric if it is measurable, isometry invariant, additive, and locally bounded (see the beginning of Section \ref{sec:GF}). The following theorem provides a very general series representation for the mean value of an additive function of $Z$ restricted to an observation window $W$. For increasing balls as observation windows,  we prove the existence of the asymptotic mean and obtain a series representation which involves an integration over horoballs.

\begin{theorem}\label{thm:mean}
\begin{itemize}
\item [{\rm (a)}] If $W\in\mathcal{K}^d$ and $\phi:\mathcal{R}^d\to\mathbb{R}$ is measurable, additive, and locally bounded, then
\begin{equation}\label{eqn:limit_meana}
\E \phi(Z\cap W) = \sum_{n=1}^\infty \frac{(-1)^{n-1}}{n!} \int_{(\mathcal{K}^d)^n} \phi(K_1\cap\hdots\cap K_n\cap W) \, \Lambda^n(\dint(K_1,\hdots,K_n)).
\end{equation}
\item [{\rm (b)}] Let $\phi:\mathcal{R}^d\to\mathbb{R}$ be a geometric functional that is continuous on $\mathcal{K}^d$. Then the limit
\begin{equation}\label{eqn:limit_mean}
m_{\phi,Z}\defeq\lim_{R\to\infty} \frac{\E \phi(Z\cap \mathbb{B}_R)}{\Vol(\mathbb{B}_R)}
\end{equation}
exists and is given by
\begin{align}
m_{\phi,Z}
& = \gamma  \sum_{n=1}^\infty \frac{(-1)^{n-1}}{n!} \int_{\hb}
 \int_{\mathcal{K}^d_\sfp} \int_{(\mathcal{K}^d)^{n-1}} \phi(G\cap K_2 \cap \hdots\cap K_n\cap B)\nonumber \\
&\qquad \times  \Lambda^{n-1}(\dint(K_2,\hdots,K_n)) \, \BQ(\dint G) \,  \mu_{\rm hb}(\dint B).\label{eq:specmeanlimit}
\end{align}
\end{itemize}
\end{theorem}

\begin{remark}{\rm The proof will show that the right-hand side of
    \eqref{eq:specmeanlimit} is
    still finite if we drop the alternating sign and replace the
    integrand by its absolute value, in particular, the series converges  absolutely. Also note that if $\phi$ is isometry invariant and continuous, then $\phi$ is locally bounded.}
\end{remark}

Let us now compare the findings of Theorem \ref{thm:mean} with the Euclidean case. We denote the stationary Boolean model in $\mathbb{R}^d$ by $Z_{\rm Euc}$ and  apart from that use the same notation as in the hyperbolic framework. The exact mean value formula \eqref{eqn:limit_meana} in Theorem \ref{thm:mean} is precisely as in the Euclidean case (see, e.g.,  \cite[Theorem 9.1.2]{SW08}). For additive functionals of Boolean models in $\mathbb{R}^d$ it is well known (see, e.g.,   \cite[Theorem 9.2.1]{SW08}) that the limit in \eqref{eqn:limit_mean} exists and equals
$$
m_{\phi,Z_{\rm Euc}}=\E \phi(Z_{\rm Euc}\cap [0,1)^d).
$$
As the half-open cube $[0,1)^d$ is not an element of the convex ring, $\phi(Z_{\rm Euc}\cap [0,1)^d)$ has to be understood as the expression for the closed cube minus the contribution from its upper boundary which as a union of closed faces belongs to the convex ring as well. Moreover, in the Euclidean case one can take homothetic rescalings of any compact convex set with positive volume. Similarly to \eqref{eq:specmeanlimit},  for a geometric functional $\phi$ in the Euclidean setup we have
\begin{equation}\label{eqn:mean_Euclidean}
m_{\phi,Z_{\rm Euc}}
 = \gamma  \sum_{n=1}^\infty \frac{(-1)^{n-1}}{n!}
 \int_{\cK_\sfp^d} \int_{(\mathcal{K}^d)^{n-1}} \phi(G\cap K_2 \cap \hdots\cap K_n)\,\Lambda^{n-1}(\dint(K_2,\hdots,K_n)) \, \BQ(\dint G),
\end{equation}
where we may choose $\sfp=o\in \R^d$.  This formula can be shown along
the lines of the proof of \cite[Theorem 3.1]{HLS}. We are not aware of
a reference in the literature, although more explicit mean value
formulas for large classes of $\phi$ were derived by integral
geometric formulas; see \cite{HLW2023,SW16}.  
Equations  \eqref{eq:specmeanlimit} and \eqref{eqn:mean_Euclidean} differ in two
ways. The hyperbolic case involves an integration over horoballs and
requires the additional assumption of continuity of $\phi$. In order
to understand these differences, let us give a sketch of the proof of
\eqref{eqn:mean_Euclidean} and explain why it does not carry over to
the hyperbolic framework. For the $n$-th summand in the exact mean
value formula, which is as at \eqref{eqn:limit_meana}, we have
\begin{align*}
 & \int_{(\mathcal{K}^d)^n} \phi(K_1\cap\hdots\cap K_n\cap \mathbb{B}_R) \, \Lambda^n(\dint(K_1,\hdots,K_n)) \\
& = \gamma\int_{\cK_\sfp^d} \int_{\mathbb{R}^d} \int_{(\mathcal{K}^d)^{n-1}} \phi((x+G)\cap K_2\cap\hdots\cap K_n\cap \mathbb{B}_R) \, \Lambda^{n-1}(\dint(K_2,\hdots,K_n)) \, \dint x\, \BQ(\dint G) \allowdisplaybreaks\\
& = \gamma\int_{\cK_\sfp^d}  \Vol(\{x\in\mathbb{R}^d: x+G\subseteq \mathbb{B}_R \}) \\
&\hspace{2cm} \quad \times\int_{(\mathcal{K}^d)^{n-1}} \phi(G\cap K_2\cap\hdots\cap K_n) \, \Lambda^{n-1}(\dint(K_2,\hdots,K_n)) \, \BQ(\dint G) \allowdisplaybreaks\\
& \quad + \gamma\int_{\cK_\sfp^d} \int_{\mathbb{R}^d} \mathbf{1}\{(x+G)\cap \partial\mathbb{B}_R\neq\varnothing \} \\
& \hspace{2cm} \quad \times \int_{(\mathcal{K}^d)^{n-1}} \phi((x+G)\cap K_2\cap\hdots\cap K_n\cap \mathbb{B}_R)\, \Lambda^{n-1}(\dint(K_2,\hdots,K_n)) \, \dint x \, \BQ(\dint G).
\end{align*}
Since $\Vol(\{x\in\mathbb{R}^d \colon x+G\subseteq \mathbb{B}_R \})/\Vol(\mathbb{B}_R)\to 1$ as $R\to\infty$ for $\BQ$-a.e.\ $G\in\mathcal{K}^d$, the first summand on the right-hand side divided by $\Vol(\mathbb{B}_R)$ converges to the integral on the right-hand side of \eqref{eqn:mean_Euclidean}. On the other hand, the fact that
$$
\Vol(\{x\in\mathbb{R}^d\colon (x+G)\cap \partial\mathbb{B}_R\neq\varnothing \}) \slash \Vol(\mathbb{B}_R)\to 0 \quad \text{as} \quad R\to\infty
$$
for $\BQ$-a.e.\ $K\in\cK^d_\sfp$ can be used to establish that the second term vanishes as $R\to\infty$.

In the hyperbolic framework, the $n$-th term of \eqref{eqn:limit_meana} can be rewritten as
\begin{align*}
& \int_{(\mathcal{K}^d)^n} \phi(K_1\cap\hdots\cap K_n\cap \mathbb{B}_R) \, \Lambda^n(\dint(K_1,\hdots,K_n)) \\
& = \gamma \int_{\mathcal{K}^d_{\mathsf{p}}} \int_{\mathcal{I}_d} \int_{(\mathcal{K}^d)^{n-1}} \phi(\varrho G\cap K_2\cap\hdots\cap K_n\cap \mathbb{B}_R) \, \Lambda^{n-1}(\dint(K_2,\hdots,K_n)) \, \lambda(\dint\varrho) \, \BQ(\dint G).
\end{align*}
Now we have that
$\lambda(\{\varrho\in\mathcal{I}_d\colon \varrho
G\cap\partial\mathbb{B}_R\neq\varnothing \})$ is of the order of the
surface area of $\mathbb{B}_R$ if $G$ has nonempty interior. Since for
$R\to\infty$ the volume and the surface area of $\mathbb{B}_R$ are of
the same order (see \eqref{eq:limit1}),  
the integral involving intersections with the boundary 
does not vanish in the hyperbolic setup.
Instead a finer analysis leads to the integral over the horoballs with respect to
$\mu_{\rm hb}$. This argument also requires the continuity of $\phi$.

 In the following, we briefly write
 $$
 v_i\defeq\E V_i(\TG),\quad i\in\{0,\ldots,d\}.
 $$
 As a consequence of the previous theorem, we obtain mean value formulas for the volume $V_{d}$ and  the surface area $V_{d-1}$. 

\begin{corollary}\label{cor:expectation_volume_surface_area}
If $W\in\cK^d$, then
\begin{align}
\E V_{d}(Z\cap W)&=V_d(W)\left(1-e^{-\gamma v_d}\right),\label{eq4.3a}\\
\E V_{d-1}(Z\cap W)&=V_d(W)\gamma v_{d-1} e^{-\gamma v_d}+V_{d-1}(W)\nonumber
\left(1-e^{-\gamma v_d}\right)
\end{align}
and
\begin{align}
m_{V_d,Z}&=1-e^{-\gamma v_d}, \nonumber\\
m_{V_{d-1},Z}&=\gamma v_{d-1} e^{-\gamma v_d}+(d-1)\left(1-
e^{-\gamma v_d}\right). \label{eqn:mVd-1}
\end{align}
\end{corollary}

Note that the formulas for the volume also follow  from the basic Fubini-type relation
$$
\E\Vol(Z\cap W)=\BP(\sfp\in Z)\Vol(W),
$$
since
$$
\BP(\sfp\in Z)=1-\BP(\sfp\notin Z) =1-e^{-\Lambda([\{\sfp\}])}
$$
and
$$
\Lambda([\{\sfp\}])=\gamma \int_{\mathcal{K}^d_{\sfp}}\int_{\Ih}\I\{\sfp\in \varrho G\}\,
\lambda(\dint\varrho)\, \BQ(\dint G)=\gamma v_d.
$$
This argument is easier than applying Theorem \ref{thm:mean}, but it
heavily relies on the properties of the volume and cannot be used for
general geometric functionals as they are considered in Theorem
\ref{thm:mean}. Note that in case of the volume functional both,  Theorem
\ref{thm:mean} and the Fubini type argument work for general compact particles.

In the following, we use the abbreviation
\begin{align*}
v_j^0\defeq \int_{\mathcal{K}^d_{\sfp}} V^0_j(G)\,\BQ(\dint G)= \E V^0_j(\TG),\quad j\in\{0,\ldots,d\},
\end{align*}
where $v_d=v_d^0$ by definition. Next we present mean value formulas for general intrinsic volumes. In order to simplify the formulas, we consider the renormalisation from \eqref{eq:vk0}.

\begin{theorem}\label{Thmmeanvalues} If $k\in\{0,\ldots,d-1\}$ and $W\in \mathcal{K}^d$, then
\begin{align}\label{eq:infinitev}\notag
\E V^0_k(Z\cap W)&=V^0_k(W)\big(1-e^{-\gamma v_d}\big)\\
&\qquad+e^{-\gamma v_d}\sum^d_{m=k+1}V^0_m(W)\sum^{m-k}_{s=1}\frac{(-1)^{s-1}\gamma^s}{s!}
\sum_{\substack{m_1,\dots ,m_s=k\\ m_1+\cdots +m_s=sd+k-m}}^{d-1} v^0_{m_1}\cdots v^0_{m_s}.
\end{align}
\end{theorem}

The previous theorem leads quickly to the following formula
for the asymptotic density \eqref{eqn:limit_mean} for $\phi=V_k$.

\begin{corollary}\label{corr:asym} If $k\in\{0,\ldots,d-1\}$, then
\begin{align*}
m_{V_k,Z}&=(d-1)\big(1-e^{-\gamma v_d}\big)\\
&\qquad
+e^{-\gamma v_d}\sum^d_{m=k+1} b_{m,k}\sum^{m-k}_{s=1}\frac{(-1)^{s-1}\gamma^s}{s!}
\sum_{\substack{m_1,\dots ,m_s=k\\ m_1+\cdots +m_s=sd+k-m}}^{d-1} v^0_{m_1}\cdots v^0_{m_s},
\end{align*}
where $b_{m,k}\defeq (d-1) \frac{\kappa_m\kappa_{d-1-m}}{\kappa_k\kappa_{d-1-k}}$
for $0\le k<m<d$ and $b_{d,k}\defeq  \frac{d}{\pi}\frac{\kappa_d}{\kappa_k\kappa_{d-1-k} }$.
\end{corollary}
\begin{proof} Rather than using Theorem \ref{thm:mean} (b), we obtain the assertion directly  from Theorem \ref{Thmmeanvalues}  via \eqref{eq:limit2}, keeping in mind the normalization \eqref{eq:vk0} and using relation \eqref{eq:sec2rel}.
\end{proof}

\begin{remark} \rm
    The special case $k=d-1$ of Corollary \ref{corr:asym} recovers \eqref{eqn:mVd-1} provided in Corollary \ref{cor:expectation_volume_surface_area}. For an explicit comparison note that $v^0_m$ in the statement of Corollary \ref{corr:asym} is defined in terms of the normalized intrinsic volumes.
\end{remark}

\begin{remark} \rm
    The formula in Corollary \ref{corr:asym}  is similar to the Euclidean counterpart, provided for instance in \cite[Theorem 9.1.4]{SW08}. However, in the hyperbolic setup the sum over $m\in\{k+1,\ldots,d\}$ does not reduce to just the case $m=d$ and all summands contribute asymptotically.
\end{remark}

\begin{corollary} If $d=2$ and $W\in\mathcal{K}^2$, then
\begin{equation}\label{eq:Miles1}
\E\chi(Z\cap W)=(1-e^{-\gamma v_2})+V_1(W) e^{-\gamma v_2}\frac{\gamma v_1}{2\pi}
+V_2(W) e^{-\gamma v_2}\left( \gamma +\frac{\gamma v_2}{2\pi}
-   \frac{(\gamma v_1)^2}{4\pi}\right).
\end{equation}
\end{corollary}
\begin{proof}
Recall that in the case $d=2$ we have $V_i^0=V_i$ for $i=0,1,2$. Hence also $v^0_i=v_i$. Then, for $k=0$, Equation \eqref{eq:infinitev} simplifies to
\begin{align}\label{eq:5.24}
\E V_0(Z\cap W)&=V_0(W)(1-e^{-\gamma v_2})+V_1(W) e^{-\gamma v_2}\gamma v_1
+V_2(W) e^{-\gamma v_2}\gamma v_0\nonumber\\
&\qquad -V_2(W) e^{-\gamma v_2}\frac{(\gamma v^0_1)^2}{2}.
\end{align}
Relations \eqref{eq:chi2} and \eqref{eq4.3a} yield
$$
2\pi \E\chi(Z\cap W)=\E V_0(Z\cap W)- \E V_2(Z\cap W)=\E V_0(Z\cap W)-V_2(W)(1-e^{-\gamma v_2}).
$$
Inserting  \eqref{eq:5.24} and using that $V_0(W)-V_2(W)=2\pi$ and $v_0 =2\pi+v_2$ by \eqref{eq:chi2}, we obtain the asserted formula \eqref{eq:Miles1}.
\end{proof}

Equations \eqref{eq:Miles1} and \eqref{eq:limit2} show that
the asymptotic density of the Euler characteristic is given by
\begin{align}\label{eq:Miles2}\notag
m_{\chi,Z}=\gamma e^{-\gamma v_2}-\gamma^2 e^{-\gamma v_2}\frac{v_1^2}{4\pi}+\gamma e^{-\gamma v_2}\frac{1}{2\pi}\left(v_1+v_2\right).
\end{align}
This should be compared with a classical formula by Miles \cite{Miles76}
(see also \cite[Section 9.1]{SW08}), stating that in the Euclidean case (recall that our normalization of $V_1$ differs from the one in Euclidean space by a factor $2$ so that $\E V_1(\TG)=\E\mathcal{H}^1(\partial\TG)$  assuming that $\TG$ has nonempty interior a.s.)
\begin{align*}
m_{\chi,Z_{\rm Euc}}=
\gamma e^{-\gamma v_2}-\gamma^2 e^{-\gamma v_2}\frac{\E \mathcal{H}^1(\partial\TG)^2}{4\pi}.
\end{align*}

\section{Covariance formulas}\label{sec:5}

As in the preceding section, let $\eta$ be a stationary Poisson particle process concentrated on convex particles. Let $Z$ denote the induced stationary Boolean model. 
For a measurable, additive, and locally bounded functional $\phi:\mathcal{R}^d\to\mathbb{R}$, we define $\phi^*:\mathcal{R}^d\to\mathbb{R}$ by
\begin{equation}\label{eqn:phi_star}
\phi^*(A)\defeq\E\phi(Z\cap A) - \phi(A)\quad \text{for } A\in \mathcal{R}^d.
\end{equation}

The second order results described in Theorem \ref{Thmcovasymp} require a strengthening of the first order integrability condition $\E\overline{\Vol}(\TG)<\infty$ given in Theorem \ref{thmchar}.

\begin{theorem}\label{Thmcovasymp}
Assume that $\E\overline{\Vol}(\TG)^2<\infty$.
\begin{itemize}
\item [{\rm (a)}] If $\phi,\psi:\mathcal{R}^d\to\mathbb{R}$ are measurable, additive, and locally bounded functionals and $W\in\cK^d$, then
$$\E\phi(Z\cap W)^2,\E\psi(Z\cap W)^2<\infty$$
and
\begin{align}
&\Cov(\phi(Z\cap W),\psi(Z\cap W)) \nonumber \\
& =  \sum_{n=1}^\infty \frac{1}{n!} \int_{(\cK^{d})^n} (\phi^*\cdot\psi^*)(K_1\cap\hdots\cap K_n \cap W) \,\Lambda^{n}(\dint (K_1,\hdots,K_n)). \label{eqn:Exact_Covariance}
\end{align}
\item [{\rm (b)}] For a  measurable, additive, and locally bounded  functional $\phi:\mathcal{R}^d\to\mathbb{R}$,
$$
\limsup_{R\to\infty} \frac{\Var \phi(Z\cap \mathbb{B}_R)}{\Vol(\mathbb{B}_R)}<\infty.
$$
\item [{\rm (c)}] If $\phi,\psi:\mathcal{R}^d\to\mathbb{R}$ are geometric functionals that are continuous on $\mathcal{K}^d$, then
\begin{align*}
& \lim_{R\to\infty} \frac{\Cov(\phi(Z\cap \mathbb{B}_R),\psi(Z\cap \mathbb{B}_R))}{\Vol(\mathbb{B}_R)} \\
& = \gamma\sum_{n=1}^\infty \frac{1}{n!} \int_{\hb}  \int_{\mathcal{K}^d_\sfp} \int_{(\mathcal{K}^d)^{n-1}} (\phi^*\cdot\psi^*)(G\cap K_2\cap\hdots\cap K_n\cap B)  \\
& \qquad \times \Lambda^{n-1}(\dint (K_2,\hdots,K_n)) \, \BQ(\dint G) \, \mu_{\rm hb}(\dint B).
\end{align*}
\end{itemize}
\end{theorem}

For Boolean models in Euclidean space parts (a) and (b) of Theorem \ref{Thmcovasymp} are implicitly provided in \cite[Section 3]{HLS}, while the limit of the covariances is computed in \cite[Theorem 3.1]{HLS}. There one obtains the formula from part (c) without the integration over horoballs and without assuming continuity on $\cK^d$. This means that we observe the same differences  between the hyperbolic and the Euclidean framework for the asymptotic covariances as for the asymptotic mean values, which can be explained as next to Theorem \ref{thm:mean}.

\medskip

Using the Chebyshev inequality, we obtain from Theorem \ref{Thmcovasymp} (b) the following weak law of large numbers, where $\overset{\BP}{\longrightarrow}$ denotes convergence in probability.

\begin{corollary}
Assume that $\E\overline{\Vol}(\TG)^2<\infty$. If $\phi:\mathcal{R}^d\to\mathbb{R}$ is a measurable, additive,  and locally bounded functional such that the limit $m_{\phi,Z}$ in \eqref{eqn:limit_mean} exists, then
$$
\frac{\phi(Z\cap \mathbb{B}_R)}{\Vol(\mathbb{B}_R)} \overset{\BP}{\longrightarrow} m_{\phi,Z} \quad \text{as} \quad R\to\infty.
$$
\end{corollary}

 In the following, we use the \emph{covariogram function} $C_W:\BHd\times\BHd\to [0,\infty)$ of a convex body $W\in\cK^d$, which is defined by
$$
C_W(x,z)\defeq\lambda(\{\varrho\in\Ih\colon \varrho x,\varrho z\in W\}),\quad x,z\in\mathbb{H}^d.
$$
The \emph{mean covariogram of the typical particle} $\TG$ is the function $C:\BHd\times\BHd\to [0,\infty)$ given by
\begin{equation}\label{eqn:covariogram}
C(x,z)\defeq\E C_{\TG}(x,z)=\E\lambda(\{\varrho\in\Ih\colon\varrho x\in \TG,\varrho z\in\TG\})=\frac{1}{\gamma}\int_{\cK^d}\mathbf{1}\{x,z\in K\}\, \Lambda(\dint K)
\end{equation}
for $ x,z\in\BHd$.
Note that $C_W(x,z)=C_W(\tau x,\tau z)$, for $x,z\in\BHd$ and $\tau\in \Ih$, and the same relation holds for $C$. Hence, the covariogram function of $W$ and the mean covariogram of the typical particle both depend  only on the relative position of their arguments. Moreover, both functions depend continuously on their arguments. This  follows from the
fact that isometries are continuous and that for each
$y\in\BH^d$ and each $K\in\mathcal{K}$ we have that
$\varrho y\notin \partial K$ for $\lambda$-a.e.\ $\varrho\in \Ih$. Finally, we point out that the dominated convergence theorem implies that  $C(\sfp,z)\to 0$ as $d_h(\sfp,z)\to\infty$.

\medskip

  We denote by $\operatorname{Uniform}(\mathbb{S}^{d-1}_\sfp)$ the uniform distribution on $\mathbb{S}^{d-1}_\sfp$ and by $\operatorname{Exp}(a)$ the exponential distribution with parameter $a>0$. 

\begin{theorem}\label{cor:variance_volume}
Assume that $\E\overline{\Vol}(\TG)^2<\infty$. Then
\begin{align}\label{eqn:variance_volume}
\Var V_d(Z\cap W)
& =e^{-2\gamma v_d}   \int_{\BHd} \left(e^{\gamma C(\sfp,    z)}-1\right) C_W( \sfp, z )\,
\mathcal{H}^d(\dint z)<\infty
\end{align}
for $W\in\cK^d$,
\begin{equation}\label{eqn:variance_volume_ball}
\Var V_d(Z\cap \BB_R) = e^{-2\gamma v_d} \int_{\BHd} \left(e^{\gamma C(\sfp,    z)}-1\right)\mathcal{H}^d\left(\BB_R\cap \BB(z,R)\right)\,
\mathcal{H}^d(\dint z)
\end{equation}
for $R>0$, and
\begin{align}
\lim_{R\to\infty} \frac{\Var V_d(Z\cap \mathbb{B}_R)}{\Vol(\mathbb{B}_R)} 
&= e^{-2\gamma v_d} \int_{\mathbb{H}^d} \left(e^{\gamma C(\sfp,z)} -1 \right) \int_{\hb}  \mathbf{1}\{\sfp,z\in B\} \, \mu_{\rm hb}(\dint B) \,
 \mathcal{H}^d(\dint z) \label{eqn:asymptotic_variance_volumeproto}\\
&= e^{-2\gamma v_d} \int_{\BHd} \left(e^{\gamma C(\sfp,    z)}-1\right) \BP(z\in \horo_{U,T}) \, \mathcal{H}^d(\dint z)<\infty \label{eqn:asymptotic_variance_volume}
\end{align}
with independent $U\sim\operatorname{Uniform}(\mathbb{S}^{d-1}_\sfp)$ and $-T\sim\operatorname{Exp}(d-1)$.
\end{theorem}

\begin{remark}\label{rem:general_volume}
{\rm Theorem \ref{cor:variance_volume} holds more generally for a stationary Poisson particle process with compact particles and - in case of \eqref{eqn:variance_volume} - a general compact observation window. The proof will be given at the end of Section \ref{sec:12}.}
\end{remark}

 For $i,j\in\{0,\ldots,d\}$ and $W\in\mathcal{K}^d$ we define the local covariances
\begin{align*}
\sigma_{i,j}(W)\defeq\Cov(V_i(Z\cap W),V_j(Z\cap W))
\end{align*}
and the asymptotic covariances
\begin{align*}
\sigma_{i,j}\defeq \lim_{R\to\infty} \frac{\Cov(V_i(Z\cap \mathbb{B}_R),V_j(Z\cap \mathbb{B}_R))}
{V_d(\mathbb{B}_R)}
\end{align*}
whenever the limit exists. In the following we shall determine explicit expressions for 
these covariances for $i,j\in\{d-1,d\}$.

First we need to introduce
some notation.
For $K\in\mathcal{K}^d$, let $C_{d-1}(K,\cdot)$ be the boundary (surface) measure  associated with $K$, that is, the top order curvature measure of $K$ in $\mathbb{H}^d$  (see Section \ref{sec:GF}). If $K^\circ\neq\varnothing$, then
$C_{d-1}(K,\cdot)=\mathcal{H}^{d-1}(\cdot\cap \partial K)$. In addition, we   define $C_d(K,\cdot)\defeq\mathcal{H}^d(\cdot\cap K)$.

For the statement of the results, we define measures $M_{i,j}$   on $(\BHd)^2$ by
\begin{align*}
M_{i,j}&\defeq\E \int_{\mathbb{H}^d}\int_{\mathbb{H}^d} \I\{(x,y)\in\cdot\}\,C_{i}(\TG,\dint x)\,\,C_{j}(\TG,\dint y),
\end{align*}
for $i,j\in\{d-1,d\}$.
The condition $\E\overline{\Vol}(\TG)^2<\infty$ and the Steiner formula \eqref{eq:Steiner} ensure that these measures are finite.

The next theorem provides rather explicit expressions for the local variances and covariances of volume and surface area. In the statements of these results we restrict the integration domains  by an indicator function (if necessary) and do not explicitly write the integration domains if they are determined by the measures.

\begin{theorem}\label{th:localcovariance}
Assume that $\E\overline{\Vol}(\TG)^2<\infty$. If $W\in\mathcal{K}^d$, then 
\begin{align}\label{eq:987} \notag
\sigma_{d-1,d}(W)&=
- e^{-2\gamma v_d} \gamma v_{d-1} \int  \Big(e^{\gamma C(\sfp, z)}-1\Big)
C_W(\sfp,z)\,\mathcal{H}^d(dz)\\\notag
&\quad+ e^{-2\gamma v_d} \gamma\int e^{\gamma C(y,z)}C_W(y,z)
\,M_{d-1,d}(\dint (y,z))\\
&\quad+e^{-2\gamma v_d}\iint \big(e^{\gamma C(y,z)}-1\big)\I\{z\in W\} \,C_{d-1}(W,\dint y)\,\mathcal{H}^d(\dint z),
\end{align}
and, with independent copies $\TG_1,\TG_2$ of the typical particle,
\begin{align}\label{eq:996}\notag
\sigma_{d-1,d-1}(W)\notag
&=e^{-2\gamma v_d}\gamma^2v_{d-1}^2\int  \Big(e^{\gamma C(\sfp, z)}-1\Big)
C_W(\sfp,z) \,\mathcal{H}^d(dz)\\\notag
&\quad -2e^{-2\gamma v_d}\gamma^2 v_{d-1}\int  e^{\gamma C(y, z)}C_W(y,z)\,M_{d-1,d}(\dint(y,z))\\\notag
&\quad -2e^{-2\gamma v_d}\gamma v_{d-1}\iint  \big(e^{\gamma C(y, z)}-1\big)\I\{z\in W\}
\, C_{d-1}(W,\dint y)\,\mathcal{H}^d(dz)\\\notag
&\quad +e^{-2\gamma v_d}\gamma \int e^{\gamma C(y, z)}C_W(y,z)
\,M_{d-1,d-1}(\dint (y,z))\\\notag
&\quad +e^{-2\gamma v_d} \gamma^2 \E \iiint e^{\gamma C(\varrho y, z)}C_W(\varrho y,z)
\I\{\varrho y\in \TG_2,z \in \varrho \TG_1\}\notag\\
&\qquad\quad\times
\,\lambda(\dint\varrho)\,C_{d-1}(\TG_1,\dint y)\,C_{d-1}(\TG_2,\dint z)  \notag \\\notag
&\quad +2e^{-2\gamma v_d}\gamma \E \iiint e^{\gamma C(\varrho y,  z)}\I\{\varrho y\in \TG,z\in \varrho W\}\, \lambda(\dint\varrho)
 \,C_{d-1}(W,\dint y)\,C_{d-1}(\TG,\dint z)\\
&\quad +e^{-2\gamma v_d}\iint \big(e^{\gamma C(y, z)}-1\big)\,C_{d-1}(W,\dint y)\,C_{d-1}(W,\dint z).
\end{align}
Moreover, all integrals involved in \eqref{eq:987} and \eqref{eq:996} are finite.
\end{theorem}

 From Theorem \ref{th:localcovariance}  asymptotic variances and covariances of these functionals could in principle be derived, which are stated in the following theorem. Proceeding in this way, the proof of the relation
 $$
 \lim_{R\to\infty}
 \int \big(e^{\gamma C(\sfp, z)}-1\big)\,C_{d-1}(\mathbb{B}(\exp_{\sfp}(Ru),R) ,\dint z)=\int \big(e^{\gamma C(\sfp, z)}-1\big)\,C_{d-1}(\mathbb{B}_{u,0} ,\dint z),
$$
where $u\in\mathbb{S}^{d-1}_\sfp$,
turns out to be difficult. Here we define $C_{d-1}(\mathbb{B}_{u,0} ,\cdot)\defeq\mathcal{H}^{d-1}(\partial \mathbb{B}_{u,0}\cap\cdot)$.  For this reason, we base the proof of Theorem \ref{th:asympcovariance} on   Theorem \ref{Thmcovasymp} (c) and prepare the proof by some additional integral geometric formulas (see Section \ref{sec:7}) which may be of independent interest.

\begin{theorem}\label{th:asympcovariance}
Assume that $\E\overline{\Vol}(\TG)^2<\infty$. Let $u\in\mathbb{S}^{d-1}_\sfp$ be fixed. Then
\begin{align}\label{eq:987b}
\sigma_{d-1,d}
&=- e^{-2\gamma v_d} \gamma v_{d-1} \iint  \Big(e^{\gamma C(\sfp, z)}-1\Big) \mathbf{1}\{ \sfp,z \in B \}
\,\mathcal{H}^d(\dint z)\, \mu_{\rm hb}(\dint B) \nonumber\\\notag
&\quad+ e^{-2\gamma v_d}\gamma\iint  e^{\gamma C(y, z)}\I\{y,z\in B\}\,M_{d-1,d}(\dint(y,z))
\, \mu_{\rm hb}(\dint B) \\
&\quad+ (d-1)\, e^{-2\gamma v_d}\int \big(e^{\gamma C(\sfp, z)}-1\big)\I\{z\in \mathbb{B}_{u,0}\}
\,\mathcal{H}^d(\dint z)
\end{align}
and, with independent copies $\TG_1,\TG_2$ of the typical particle,
\begin{align}\label{eq:996b}
\sigma_{d-1,d-1}
&=e^{-2\gamma v_d}\gamma^2v_{d-1}^2\iint  \Big(e^{\gamma C(\sfp, z)}-1\Big)
\mathbf{1}\{ \sfp,z \in B \}
\,\mathcal{H}^d(\dint z)\, \mu_{\rm hb}(\dint B)\notag\\\notag
&\quad -2e^{-2\gamma v_d}\gamma^2 v_{d-1}\iint  e^{\gamma C(y, z)}\I\{y,z\in B\}\,M_{d-1,d}(\dint(y,z))
\, \mu_{\rm hb}(\dint B)\\\notag
&\quad -2(d-1)e^{-2\gamma v_d}\gamma v_{d-1}\int \big(e^{\gamma C(\sfp, z)}-1\big)\I\{z\in \mathbb{B}_{u,0}\}
\,\mathcal{H}^d(\dint z)\\\notag
&\quad +e^{-2\gamma v_d}\gamma \iint  e^{\gamma C(y, z)}\I\{y,z\in B\}\,M_{d-1,d-1}(\dint(y,z))
\, \mu_{\rm hb}(\dint B)\notag\\\notag
&\quad +e^{-2\gamma v_d} \gamma^2 \E \iiiint e^{\gamma C(\varrho y, z)} \I\{\varrho y\in \TG_2\cap B,z\in\varrho \TG_1\cap B\}\\
&\quad\qquad\times\,\lambda(\dint\varrho)\,C_{d-1}(\TG_1,\dint y)\,C_{d-1}(\TG_2,\dint z)\, \mu_{\rm hb}(\dint B) \notag \\\notag
&\quad +2(d-1)e^{-2\gamma v_d}\gamma \E \iint e^{\gamma C(\varrho \sfp,  z)}\I\{\varrho \sfp\in \TG, z\in \varrho\mathbb{B}_{u,0}\}\, \lambda(\dint\varrho)\,C_{d-1}(\TG,\dint z)\notag\\
&\quad +(d-1)e^{-2\gamma v_d}\int \big(e^{\gamma C(\sfp, z)}-1\big)\,C_{d-1}(\mathbb{B}_{u,0},\dint z).
\end{align}
Moreover, all integrals involved in \eqref{eq:987b} and \eqref{eq:996b} are finite.
\end{theorem}

A counterpart of the formula for $\sigma_{d-1,d}$ in Euclidean space is stated in \cite[Corollary 6.2]{HLS}. The Euclidean formula involves only two summands which correspond to the first two terms on the right-hand side of \eqref{eq:987b}. However, in the hyperbolic setting both terms involve an additional integration over horoballs. The  formula for $\sigma_{d-1,d-1}$ in Euclidean space (see again \cite[Corollary 6.2]{HLS}) effectively involves four terms (if the second is written as a difference of two integrals in the obvious way). The first, the second, the fourth, and the fifth contribution on the right side of \eqref{eq:996b} represent the hyperbolic analogues. The hyperbolic expressions again require an additional integration over the space of horoballs. In addition, formula \eqref{eq:996b} requires three further contributions  each of which involves a fixed horoball whose boundary contains $\sfp$. The additional summands that arise in hyperbolic space all have the prefactor $d-1$, which is due to some limit of the form \eqref{eq:limit1}.

It is clear that our approach for determining the asymptotic covariances $\sigma_{i,j}$ in principle works for all $i,j\in\{0,\ldots,d\}$. Note, however, that the case $i=0$ no longer simplifies as in the Euclidean setup, since $V_0$ is not proportional to the Euler characteristic. Since the formulas are already complicated in Euclidean space, we refrain from working out the other formulas in detail here. Some results on local covariance formulas for intrinsic volumes in Euclidean space, are derived in Appendix C of the arXiv version (v2) of \cite{HLS}.

\section{Lower bounds for variances}\label{sec:lower_bound_variance}

In order to show that asymptotic variances do not degenerate and to determine the order of variances, which is an essential step for deriving central limit theorems, one is interested in lower bounds for variances as they are established in the following result.

\begin{theorem}\label{thm:lower_bound_variance}
Assume that $\E\overline{\Vol}(\TG)^2<\infty$ and let $\phi:\mathcal{R}^d\to\mathbb{R}$ be a geometric functional. If there exists some $m\in\mathbb{N}_0$ such that
\begin{equation}\label{eqn:assumption_intersection_phi}
\int_{\mathcal{K}^d} \int_{(\mathcal{K}^d)^{m}} \mathbf{1}\{ \phi(G\cap K_1\cap\hdots\cap K_m) \neq 0\} \, \Lambda^{m}(\dint (K_1,\hdots,K_m)) \, \BQ(\dint G)>0,
\end{equation}
then
$$
\liminf_{R\to\infty} \frac{\Var \phi(Z\cap \mathbb{B}_R)}{\Vol(\mathbb{B}_R)} > 0.
$$
\end{theorem}

\begin{remark}\label{rem:lower_bound_variance} \rm
The condition \eqref{eqn:assumption_intersection_phi} is satisfied for $m=0$ if $\mathbf{P}(\phi(\TG)\neq 0)>0$.
\end{remark}

Lower variance bounds for Boolean models in $\mathbb{R}^d$ under equivalent conditions to \eqref{eqn:assumption_intersection_phi} are given in \cite[Theorem 4.6]{HLW2023} and in \cite[Exercise 22.4]{LP18} (where a typo needs to be corrected), while the analogue to the simplified condition in Remark \ref{rem:lower_bound_variance} is stated in \cite[Theorem 22.9]{LP18}. In \cite[Theorem 4.1]{HLS}, the positive definiteness of the asymptotic covariance matrix of the intrinsic volumes is established provided that the typical particle is full dimensional with positive probability.

Note that Theorem \ref{thm:lower_bound_variance} in contrast to Theorem \ref{Thmcovasymp} (c) does not require continuity of $\phi$. Thus, there are cases where one can bound the variance from below, but it is unclear whether the asymptotic variance exists. This means that for the proof of Theorem \ref{thm:lower_bound_variance} one cannot proceed as in the proof of \cite[Theorem 4.1]{HLS}, where the expression for the asymptotic variance is bounded from below. The approach from \cite[Theorem 22.9 and Exercise 22.4]{LP18} cannot be applied directly to the hyperbolic framework since \cite[Equation (22.35)]{LP18} is no longer valid, but we expect that it still holds with some weight function.  Nevertheless we decided to use a completely new approach for the proof of Theorem \ref{thm:lower_bound_variance}, which relies on a lower variance bound from \cite{ST}.

\section{Central limit theorems}\label{sec:6}

For random variables $X$ and $Y$ we define the Wasserstein distance
$$
\mathbf{d}_{\rm Wass}(X,Y) = \sup_{h\in\operatorname{Lip}(1)} |\E h(X) - \E h(Y)|
$$
in case that $\E|X|,\E|Y|<\infty$, where $\operatorname{Lip}(1)$ is the set of all functions $h:\mathbb{R}\to\mathbb{R}$ such that the Lipschitz constant is at most one, while the Kolmogorov distance is given by
$$
\mathbf{d}_{\rm Kol}(X,Y) = \sup_{a\in\mathbb{R}} |\BP(X\leq a) - \BP(Y\leq a)|.
$$

As in Sections \ref{sec:4} and \ref{sec:5}, let  $\eta$ be a stationary Poisson particle process concentrated on convex particles.  The induced stationary Boolean model is denoted by $Z$.

\begin{theorem}\label{thm:CLT_univariate}
Let $\phi: \mathcal{R}^d\to\mathbb{R}$ be a geometric functional such that
\begin{equation}\label{eqn:liminf_var}
\liminf_{R\to\infty} \frac{\Var \phi(Z\cap \mathbb{B}_R)}{\Vol(\mathbb{B}_R)} >0
\end{equation}
and let $N$ denote a standard normal random variable.
\begin{itemize}
\item [{\rm (a)}] If $\E\overline{\Vol}(\TG)^2<\infty$, then
$$
\frac{\phi(Z\cap \mathbb{B}_R) - \E\phi(Z\cap \mathbb{B}_R)}{ \sqrt{\Var \phi(Z\cap \mathbb{B}_R)}} \overset{d}{\longrightarrow} N \quad \text{as} \quad R\to\infty.
$$
\item [{\rm (b)}] If $\E\overline{\Vol}(\TG)^3<\infty$, then there exist constants $C_{\rm Wass},r_0\in(0,\infty)$, depending only on $d$, $\gamma$, $\BQ$, and $\phi$, such that
$$
\mathbf{d}_{\rm Wass}\bigg( \frac{\phi(Z\cap \mathbb{B}_R) - \E\phi(Z\cap \mathbb{B}_R)}{ \sqrt{\Var \phi(Z\cap \mathbb{B}_R)}} , N \bigg) \leq \frac{C_{\rm Wass}}{\sqrt{\Vol(\mathbb{B}_R)}}
$$
for all $R\geq r_0$.
\item [{\rm (c)}] If $\E\overline{\Vol}(\TG)^4<\infty$, then there exist constants $C_{\rm Kol},r_0\in(0,\infty)$, depending only on $d$, $\gamma$, $\BQ$, and $\phi$, such that
$$
\mathbf{d}_{\rm Kol}\bigg( \frac{\phi(Z\cap \mathbb{B}_R) - \E\phi(Z\cap \mathbb{B}_R)}{ \sqrt{\Var \phi(Z\cap \mathbb{B}_R)}} , N \bigg) \leq \frac{C_{\rm Kol}}{\sqrt{\Vol(\mathbb{B}_R)}}
$$
for all $R\geq r_0$.
\end{itemize}
\end{theorem}

\begin{remark}
{\rm In Theorem \ref{thm:CLT_univariate} (a) we have even convergence with respect to the Wasserstein distance.}
\end{remark}

Let us compare Theorem \ref{thm:CLT_univariate} with the existing literature for Boolean models in Euclidean spaces. The analogue of part (a) was first shown in \cite[Theorem 9.3]{HLS}, generalising previous work on volume, surface area and related functionals in e.g.\ \cite{Baddeley,HeinrichMolchanov,Mase,Molchanov1995}. For the Wasserstein distance the rate of convergence as in part (b) was derived under a slightly stronger moment assumption in \cite[Theorem 9.1]{HLS} and with the same moment assumption in \cite[Theorem 22.7]{LP18} and \cite[Theorem 3.5]{BST}, where the latter result even includes Poisson cylinder processes. The counterpart to part (c) is shown in \cite[Theorem 3.5]{BST} under the same assumptions. The first normal approximation bound in Kolmogorov distance for Boolen models in $\mathbb{R}^d$ was established for the volume in \cite{Heinrich2005}.

We can also consider the multivariate situation where one is interested in several geometric functionals simultaneously.

\begin{theorem}\label{thm:CLT_multivariate}
Let $\phi_1,\hdots,\phi_m:\mathcal{R}^d\to\mathbb{R}$ be geometric functionals, $m\in\mathbb{N}$, and assume that there exists a matrix $\Sigma=(\sigma_{ij})_{i,j\in\{1,\hdots,m\}}\in\mathbb{R}^{m\times m}$ such that
\begin{equation}\label{eqn:Assumption_Convergence_Covariances}
\lim_{R\to\infty} \frac{\Cov(\phi_i(Z\cap\mathbb{B}_R), \phi_j(Z\cap\mathbb{B}_R))}{\Vol(\mathbb{B}_R)}=\sigma_{ij}
\end{equation}
for $i,j\in\{1,\hdots,m\}$. Let $\mathbf{N}_\Sigma$ denote an $m$-dimensional centered Gaussian random vector with covariance matrix $\Sigma$. If $\E\overline{\Vol}(\TG)^2<\infty$, then, as $R\to\infty$,
$$
\frac{1}{\sqrt{\Vol(\mathbb{B}_R)}} \big(\phi_1(Z\cap\mathbb{B}_R) - \E\phi_1(Z\cap\mathbb{B}_R),\hdots,\phi_m(Z\cap\mathbb{B}_R) - \E\phi_m(Z\cap\mathbb{B}_R)\big)\overset{d}{\longrightarrow} \mathbf{N}_\Sigma.
$$
\end{theorem}

\begin{remark} \rm
Note that \eqref{eqn:Assumption_Convergence_Covariances} is always satisfied if $\phi_1,\hdots,\phi_m$ are additionally continuous on $\mathcal{K}^d$ because of Theorem \ref{Thmcovasymp} (c).
\end{remark}

For Boolean models in $\mathbb{R}^d$ multivariate central limit theorems were derived in \cite[Theorems 2.1, 9.1]{HLS} and in \cite[Theorem 3.7]{BST}, where the second reference also deals with Poisson cylinder processes. Our multivariate central limit theorem is completely qualitative. For the so-called $d_3$- or $d_{convex}$-distance between the centered and rescaled vector of geometric functionals in Theorem \ref{thm:CLT_multivariate} and a centered Gaussian random vector with the same covariance matrix one can derive upper bounds by evaluating Malliavin-Stein bounds for the multivariate normal approximation from \cite{SY19} similarly as in our proof of Theorem \ref{thm:CLT_univariate}. For Euclidean Boolean models this was accomplished in \cite[Theorem 3.7]{BST} for the $d_3$-distance and in case of intrinsic volumes for the $d_3$- and the $d_{convex}$-distance in \cite{SY19}. The restriction to intrinsic volumes comes from the fact that for the $d_{convex}$- distance the covariance matrix of the Gaussian random vector must have full rank. The latter issue also needs to be taken into account when dealing with the $d_{convex}$-distance for Boolean models in the hyperbolic framework as described above.

A further interesting problem is to provide a rate for the convergence of the vector considered in Theorem \ref{thm:CLT_multivariate} to its limiting multivariate normal distribution. For Boolean models in $\mathbb{R}^d$ such a rate is provided in \cite[Theorem 9.1]{HLS} and in case of intrinsic volumes in \cite[Theorem 4.2]{SY19}. The crucial ingredient for showing these results is to control the speed of convergence of the exact to the asymptotic covariances. For the Euclidean case such upper bounds are provided in \cite[Theorem 3.1]{HLS}. Since the asymptotic covariances in the hyperbolic framework are more involved than in the Euclidean case due to the integration over horoballs, it is unclear if such an upper bound is still valid.

\begin{remark} \rm
We have introduced particle processes and the Boolean model  for compact particles without assuming convexity. In this case the Boolean model intersected with a compact convex observation window is not necessarily an element of the convex ring. Hence, we cannot consider geometric functionals which are merely defined on $\mathcal{R}^d$. However, the volume is still well defined. All our results remain true for the volume of Boolean models in hyperbolic space with compact particles. In the Euclidean setup, central limit theorems for the volume of Boolean models and Poisson cylinder processes without convexity assumption were shown in \cite[Theorem 3.3]{BST}.
\end{remark}

\section{Basic integral geometric formulas}\label{sec:7}

In this section, we provide some useful integral geometric relations. We start with iterated kinematic formulas for the volume and the surface area. Then we establish a crucial asymptotic mean value formula involving the invariant measure on the space of horoballs (see Lemma \ref{lem:LimitEW_New}).

\begin{lemma}\label{lem:intelem}
If $n\in\N$ and $A_0,A_1,\ldots,A_n\subseteq\BHd$ are measurable sets, then
$$
\int_{(\mathcal{I}_d)^n} \Vol(A_0\cap\varrho_1 A_1\cap\ldots\cap \varrho_n A_n)\, \lambda^n(\dint(\varrho_1,\ldots,\varrho_n))
=\Vol(A_0)\Vol(A_1)\cdots\Vol(A_n).
$$
\end{lemma}

\begin{proof}
Fubini's theorem can be applied, since $\lambda$ and $\mathcal{H}^d$ are  $\sigma$-finite measures and $\Vol$ is a nonnegative function. Hence we get
\begin{align*}
&\int_{(\mathcal{I}_d)^n} \Vol(A_0\cap\varrho_1 A_1\cap\ldots\cap \varrho_n A_n)\, \lambda^n(\dint(\varrho_1,\ldots,\varrho_n))\\
&=\int_{(\mathcal{I}_d)^n} \int_{\BHd} \mathbf{1}\{x\in A_0,x\in\varrho_1 A_1,\ldots,x\in \varrho_n A_n\}\, \mathcal{H}^d(\dint x)
\, \lambda^n(\dint(\varrho_1,\ldots,\varrho_n))\\
&=\int_{\BHd}\int_{\mathcal{I}_d}\ldots\int_{\mathcal{I}_d} \mathbf{1}\{\varrho_1^{-1}x\in A_1\}\cdots\mathbf{1}\{\varrho_n^{-1}x\in A_n\}\mathbf{1}\{x\in A_0\}\, \lambda(\dint\varrho_1)\ldots \lambda(\dint\varrho_n)\, \mathcal{H}^d(\dint x)\\
&=\int_{\BHd}\int_{\BHd}\ldots\int_{\BHd} \mathbf{1}\{x_1\in A_1\}\cdots\mathbf{1}\{x_n\in A_n\}\mathbf{1}\{x\in A_0\}\, \, \mathcal{H}^d(\dint x_1)\ldots  \mathcal{H}^d(\dint x_n)\, \mathcal{H}^d(\dint x)\\
&=\Vol(A_0)\Vol(A_1)\cdots\Vol(A_n),
\end{align*}
where we used \eqref{eq:hdlambda} and the fact that $\lambda$ is inversion invariant.
\end{proof}

The following lemma deals with a property of the top order
curvature measure $C_{d-1}(\cdot,\cdot)$ (see Section \ref{sec:GF}, in particular the
comments from \eqref{eq:Steiner} up to \eqref{eq:neua}) of convex
bodies in $\mathbb{H}^d$ that will be useful in the following.  In
Euclidean space, a corresponding more general translative version was
required and proved in \cite[Lemma 5.3]{HLS}.

\begin{lemma}\label{lem:surface} 
Let $\varrho_0$ denote the
identity in $\Ih$. If $n\in\N$ and $K_0,\ldots,K_n\in\mathcal{K}^d$, then
\begin{align*}
C_{d-1}(K_0\cap \varrho_1K_1\cap\ldots \cap \varrho_nK_n,\cdot)
=\sum^n_{i=0}\int\I\bigg\{x\in\cdot\,,x\in \bigcap_{\substack{j=0\\j \ne i}}^n\varrho_jK_j\bigg\}\,C_{d-1}(\varrho_i K_i,\dint x)
\end{align*}
for 
$\lambda^n$-a.e.\ $(\varrho_1,\ldots,\varrho_n)\in (\Ih)^n$.
\end{lemma}

\begin{proof} Recall that  $L^\circ$ denotes the interior of
  $L\in\cK^d$. Let $L_i=\varrho_i K_i$, $i=0,\ldots,n$, and
  $(\varrho_1,\ldots,\varrho_n)\in (\Ih)^n$. Also recall that $C_{d-1}(L,\cdot)$ is concentrated on $\partial L$ (which is true for all curvature measures). Since 
  $$
  \partial ( L_0 \cap\ldots \cap L_n )=\bigcup_{\varnothing \neq I\subseteq \{0,\ldots,n\}}\left(\bigcap_{r\in I}\partial L_r\cap\bigcap_{s\notin I}L_s^\circ\right)
  $$
is a disjoint decomposition  
  and by
  the measure property, we obtain
\begin{align}
  C_{d-1}(L_0 \cap\ldots \cap L_n,\cdot)&=
 \sum_{\varnothing \neq I\subseteq \{0,\ldots,n\}}C_{d-1}
\left(L_0\cap\ldots\cap L_n,\cdot \cap\bigcap_{r\in I}\partial L_r\cap\bigcap_{s\notin I}L_s^\circ\right)\nonumber\\
&=\sum_{\varnothing \neq I\subseteq \{0,\ldots,n\}}C_{d-1}
\left(\bigcap_{r\in I}L_r,\cdot \cap\bigcap_{r\in I}\partial L_r\cap\bigcap_{s\notin I}L_s^\circ\right),\label{eq:curvmeas}
\end{align}
where for the second equality we used the fact that the curvature
measures are locally determined and
$L_0\cap\ldots\cap L_n\cap U=\bigcap_{r\in I} L_r\cap U$ with the open
set $U=\bigcap_{s\notin I}L_s^\circ$ (see \cite[Theorem 14.3.2]{SW08},
\cite[p.~226]{Schneider}). For the top order curvature measure this property obviously follows from \eqref{eq:3:curvd-1} and \eqref{eq:3:curvd-1lower}. For $|I|\ge 2$,  
\cite[Theorem 5.15]{Brothers66} implies that
$$
\mathcal{H}^{d-1}\left(\bigcap_{r\in I}\partial L_r\right)=0
$$
holds for $\lambda^n$-a.e.\ $(\varrho_1,\ldots,\varrho_n)\in (\Ih)^n$.
But then all summands in \eqref{eq:curvmeas} with $|I|\ge 2$ vanish
for $\lambda^n$-a.e.\ $(\varrho_1,\ldots,\varrho_n)\in (\Ih)^n$, since
for $L\in\cK^d$ we have $C_{d-1}(L,\cdot)\le 2\mathcal{H}^{d-1}(\partial L\cap\cdot)$ (see \eqref{eq:3:curvd-1} and \eqref{eq:3:curvd-1lower}).  By the same reason, the values of
the curvature measures in \eqref{eq:curvmeas} remain unchanged if we
replace $L_s^\circ$, for $s\notin I$, by $L_s$, for $\lambda^n$-a.e.\
$(\varrho_1,\ldots,\varrho_n)\in (\Ih)^n$.
\end{proof}

Next we establish a basic integral geometric formula for the functional $V_{d-1}$.

\begin{lemma}\label{lem:intelem2}
If $n\in\N$ and $A_0,A_1,\ldots,A_n\in\cK^d$, 
then
\begin{align*}
\int_{(\mathcal{I}_d)^n} V_{d-1}(A_0\cap\varrho_1 A_1\cap\ldots\cap \varrho_n A_n)\, \lambda^n(\dint(\varrho_1,\ldots,\varrho_n))
=\sum_{j=0}^n V_{d-1}(A_j)\prod_{\substack{i=0\\ i\neq j}}^n V_d(A_i).
\end{align*}
\end{lemma}

\begin{proof}
First we show that
\begin{equation}\label{eq:firsteq}
\int_{\mathcal{I}_d} V_{d-1}(A_0\cap\varrho A_1)\,\lambda(\dint \varrho)=V_{d-1}(A_0)V_d(A_1)+V_{d}(A_0)V_{d-1}(A_1).
\end{equation}
This is a special case of a kinematic formula for curvature measures in hyperbolic space (see \eqref{eq3:kin1} and the literature cited there). Here we provide a short direct proof.

Using Lemma \ref{lem:surface}, Fubini's theorem, the inversion invariance of $\lambda$,  the covariance property of $C_{d-1}(\cdot,\cdot)$ and \eqref{eq:hdlambda}, we get
\begin{align*}
&\int_{\mathcal{I}_d} V_{d-1}(A_0\cap\varrho A_1)\,\lambda(\dint \varrho)
=\int_{\mathcal{I}_d} C_{d-1}(A_0\cap\varrho A_1,\mathbb{H}^d)\,\lambda(\dint \varrho)\\
&=\int_{\mathcal{I}_d}\int_{\mathbb{H}^d}\I\{x\in\varrho A_1\} \, C_{d-1}(A_0,\dint x)\,  \, \lambda(\dint \varrho)
+\int_{\mathcal{I}_d}\int_{\mathbb{H}^d}\I\{x\in  A_0\} \,C_{d-1}(\varrho A_1,\dint x)\,  \, \lambda(\dint \varrho)\\
&=\int_{\mathbb{H}^d}\int_{\mathcal{I}_d}\I\{x\in\varrho A_1\} \, \lambda(\dint \varrho)\,  C_{d-1}(A_0,\dint x)\,  
+\int_{\mathcal{I}_d} \int_{\mathbb{H}^d}\I\{\varrho x\in  A_0\}\, C_{d-1}(  A_1,\dint x)\,  \lambda(\dint \varrho)\\
&=V_d(A_1)C_{d-1}(A_0,\mathbb{H}^d)+ \int_{\mathbb{H}^d}\int_{\mathcal{I}_d}\I\{\varrho x\in  A_0\}\, \lambda(\dint \varrho)\, C_{d-1}(  A_1,\dint x)\\
&=V_d(A_1)V_{d-1}(A_0)+V_d(A_0)V_{d-1}(A_1),
\end{align*}
which proves \eqref{eq:firsteq}.

The proof of the general assertion can now be completed by induction over $n$ and by using Lemma \ref{lem:intelem}.
\end{proof}

As a consequence of the preceding integral geometric formula we obtain the next asymptotic mean value. In contrast to the Euclidean counterpart, the limit involves an additional summand which vanishes in the Euclidean case.

\begin{corollary}\label{lem:LimitEW2}
If $L\in\cK^d$, then
\begin{align*}
\lim_{R\to\infty}\frac{1}{\Vol(\BB_R)}\int_{\BHd}
V_{d-1}(L\cap \BB(x,R))\, \mathcal{H}^d(\dint x)
&=V_{d-1}(L)+(d-1) V_d(L).
\end{align*}
\end{corollary}

\begin{proof} An application of a special case of Lemma \ref{lem:intelem2} and \eqref{eq:hdlambda} yield
\begin{align*}
&\lim_{R\to\infty}\frac{1}{\Vol(\BB_R)}\int_{\BHd}
V_{d-1}(L\cap \BB(x,R))\, \mathcal{H}^d(\dint x)\\
&=\lim_{R\to\infty}\frac{1}{\Vol(\BB_R)}\left(V_{d-1}(L)\Vol(\BB_R)+V_d(L)V_{d-1}(\BB_R)\right)\\
&=V_{d-1}(L)+\lim_{R\to\infty}\frac{V_{d-1}( \BB_R)}{\Vol(\BB_R)}V_d(L),
\end{align*}
hence the assertion follows from \eqref{eq:limit2}.
\end{proof}

In Corollary \ref{lem:LimitEW2} we considered the asymptotic behavior of the integral average of the function $x\mapsto V_{d-1}(L\cap \BB(x,R))$ as $R\to\infty$. The limit could be expressed in terms of volume and surface area of $L$. The following lemma treats the situation where $V_{d-1}(L\cap\cdot)$ is replaced by a more general functional $\varphi$.

\begin{lemma}\label{lem:LimitEW_New}
Let $\varphi\colon \cK^d\cup \hb\to\R$ be a bounded,  measurable functional. Assume that
\begin{equation}\label{eqn:continuous_ae}
\lim_{R\to\infty}\varphi(\BB(\exp_\sfp((t+R)u),R))=\varphi(\horo_{u,t})
\end{equation}
for $\mathcal{H}^{d-1}_\sfp\otimes\mathcal{H}^1$ almost all $(u,t)\in \mathbb{S}^{d-1}_\sfp\times\R$ and that there exists some $r_0\in(0,\infty)$ such that
\begin{equation}\label{eqn:assumption_ball}
\varphi(K)=0 \quad \text{for all } K\in\cK^d \text{ with } K\cap \BB_{r_0}=\varnothing.
\end{equation}
Then
$$
\lim_{R\to\infty}\frac{1}{\Vol(\BB_R)}\int_{\BHd}\varphi(\BB(x,R))\, \mathcal{H}^d(\dint x)=\int_{\hb}\varphi(B)\, \mu_{\rm hb}(\dint B).
$$
\end{lemma}

\begin{proof} If $R\ge 1$ and $t\ge -R$, then  by \cite[Lemma 1 (a)]{HHT2021} we have
\begin{equation}\label{eqn:bound_sinh_vol}
\Vol(\BB_R)^{-1}\sinh^{d-1}(t+R)\le c_{0,d}e^{(d-1)t}.
\end{equation}
From \eqref{volbalr} and by an application of l'Hospital's rule, we obtain
\begin{equation}\label{eq:lab1}
\lim_{R\to\infty}\Vol(\BB_R)^{-1}\sinh^{d-1}(t+R)=\frac{d-1}{\omega_d}e^{(d-1)t}.
\end{equation}
By \eqref{eq:polarcoord} and a substitution we get
\begin{align}
& \Vol(\BB_R)^{-1}\int_{\BHd}\varphi(\BB(x,R))\, \mathcal{H}^d(\dint x) \notag \\
&=\int_{\mathbb{S}^{d-1}_\sfp}\int_0^\infty \varphi(\BB(\exp_\sfp(su),R))\sinh^{d-1}(s){\Vol(\BB_R)}^{-1}\,\dint s\,
\mathcal{H}^{d-1}_\sfp(\dint u) \notag \\
&=\int_{\mathbb{S}^{d-1}_\sfp}\int_{-R}^\infty \varphi(\BB(\exp_\sfp((t+R)u),R))\sinh^{d-1}(t+R){\Vol(\BB_R)}^{-1}\,\dint t\,
\mathcal{H}^{d-1}_\sfp(\dint u). \label{eqn:proof_limit_horoballs}
\end{align}
Since $\varphi$ is bounded, there exists a $M_\varphi\in(0,\infty)$ such that $|\varphi(K)|\leq M_\varphi$ for all $K\in\cK^d$. Together with \eqref{eqn:assumption_ball} and \eqref{eqn:bound_sinh_vol} we obtain, for $R\ge 1$ and $t\ge -R$,
\begin{align*}
&\left|\varphi(\BB(\exp_\sfp((t+R)u),R))\sinh^{d-1}(t+R){\Vol(\BB_R)}^{-1}\right|\\
&\le M_{\varphi} \mathbf{1}\{\BB(\exp_\sfp((t+R)u),R)\cap\BB_{r_0} \neq \varnothing\} c_{0,d}e^{(d-1)t}\\
&\le c_{0,d} M_{\varphi} \left(\mathbf{1}\{t\le 0\}e^{(d-1)t}+\mathbf{1}\{0<t\le r_0\}e^{(d-1)r_0}\right).
\end{align*}
This allows us to apply the dominated convergence theorem in \eqref{eqn:proof_limit_horoballs}. Since the integrand converges to
$\varphi(\BB_{u,t}) (d-1) e^{(d-1)t}/\omega_d$
for almost all $(u,t)\in \mathbb{S}^{d-1}_\sfp\times\R$ by \eqref{eqn:continuous_ae} and \eqref{eq:lab1}, the assertion follows from the definition of $\mu_{\rm hb}$ in \eqref{eq:muhb}.
\end{proof}

The following lemma deals with the special case where $\varphi=\psi(L\cap \cdot)$ with a continuous functional $\psi$ and some $L\in\cK^d$.

\begin{lemma}\label{lem:limits_intersection_horoballs}
If $L\in\cK^d$ and $\psi: \cK^d\to\R$ is a continuous functional satisfying $\psi(\varnothing)=0$, then the map $\varphi\colon \cK^d\cup \hb\to\mathbb{R}$, $K\mapsto \psi(L\cap K)$, is bounded, measurable, and satisfies \eqref{eqn:continuous_ae} and \eqref{eqn:assumption_ball}. Moreover,
\begin{equation}\label{eqn:limit_integral_horoballs}
\lim_{R\to\infty}\frac{1}{\Vol(\BB_R)}\int_{\BHd}\psi(L\cap \BB(x,R))\, \mathcal{H}^d(\dint x)=\int_{\hb}\psi(L\cap B)\, \mu_{\rm hb}(\dint B).
\end{equation}
\end{lemma}

\begin{proof}
Since $\{K\in\cK^d\colon K\subseteq L\}$ is a compact subset of $\cK^d$, the continuity of $\psi$ implies that $\varphi$ is bounded. Measurability is implied by \cite[Theorem 12.2.6 (a)]{SW08}. By the continuity of $\psi$, we have
$$
\lim_{R\to\infty}\varphi(\BB(\exp_\sfp((t+R)u),R))=\lim_{R\to\infty}\psi(L\cap \BB(\exp_\sfp((t+R)u),R))=\psi(L\cap \horo_{u,t})=\varphi(\horo_{u,t})
$$
for $(u,t)\in \mathbb{S}^{d-1}_\sfp\times\R$ provided that $L\cap \horo_{u,t}=\varnothing$ or if $L\cap \horo_{u,t}\neq \varnothing$ and $L,\horo_{u,t}$ cannot be
separated by a hyperbolic hyperplane. For any given $u\in \mathbb{S}^{d-1}_\sfp$,  there is only a single $t\in\R$ such that $L\cap \horo_{u,t}\neq \varnothing$ and  $L,\horo_{u,t}$ can be
separated by a hyperbolic hyperplane, hence this situation  occurs only on a null set, which proves \eqref{eqn:continuous_ae}. There exists some $r(L)\in(0,\infty)$ such that $L\subseteq \BB_{r(L)}$. Together with $\psi(\varnothing)=0$ this shows \eqref{eqn:assumption_ball} with $r_0=r(L)$. Equation \eqref{eqn:limit_integral_horoballs} is an immediate consequence of Lemma \ref{lem:LimitEW_New}.
\end{proof}

\begin{corollary}\label{cor:horint1}
\begin{itemize}
\item [{\rm (a)}] If $L\in\cK^d$, then
$$
\Vol(L)=\int_{\hb}  \Vol(L\cap B)\, \mu_{\rm hb}(\dint B).
$$
\item [{\rm (b)}] If $y,z\in\BHd$ and $\tau\in\mathcal{I}_d$, then
\begin{align*}
 \int_{\hb}  \mathbf{1}\{y,z\in \tau B\} \, \mu_{\rm hb}(\dint B)
 = \int_{\hb}  \mathbf{1}\{y,z\in B\} \, \mu_{\rm hb}(\dint B).
\end{align*}
\item [{\rm (c)}] If $y,z\in\BHd$, then
$$
\lim_{R\to\infty} \frac{\mathcal{H}^d\left(\BB(y,R) \cap \BB(z,R)\right)}{\Vol(\BB_R)} = \int_{\hb}
 \mathbf{1}\{ y,z \in B \} \, \mu_{\rm hb}(\dint B).
$$
\end{itemize}
\end{corollary}

\begin{proof}
By \eqref{eq:hdlambda} and Lemma \ref{lem:intelem}, we have
$$
\frac{1}{\Vol(\BB_R)}\int_{\BHd}\Vol(L\cap \BB(x,R))\, \mathcal{H}^d(\dint x)
=\Vol(L).
$$
Now the assertion (a) is implied by Lemma \ref{lem:limits_intersection_horoballs}.

Assertion (b) is a direct consequence of the isometry invariance of the measure $\mu_{\rm hb}$. Here we provide a simple direct argument based on Lemma \ref{lem:LimitEW_New}. For fixed points $x_1,x_2\in\BH^d$,  $\varphi \colon \cK^d\cup\hb\ni K\mapsto \mathbf{1}\{x_1,x_2\in K\}$ is bounded and satisfies \eqref{eqn:assumption_ball} with $r_0:=\max\{d_h(\sfp,x_1),d_h(\sfp,x_2)\}$. For fixed  $u\in \mathbb{S}^{d-1}_\sfp$,  \eqref{eqn:continuous_ae} is satisfied whenever $x_1\not\in\partial\BB_{u,t}$ and $x_2\notin\partial \BB_{u,t}$, which proves \eqref{eqn:continuous_ae}.
From Lemma \ref{lem:LimitEW_New} and the invariance of the Hausdorff measure it follows that
\begin{align*}
 \int_{\hb}  \mathbf{1}\{y,z\in \tau B\} \,  \mu_{\rm hb}(\dint B)
& = \lim_{R\to\infty} \frac{1}{V_d(\BB_R)} \int_{\BH^d}\mathbf{1}\{y,z\in \tau \BB(x,R)\} \, \mathcal{H}^d(\dint x) \\
& = \lim_{R\to\infty} \frac{1}{V_d(\BB_R)} \int_{\BH^d}\mathbf{1}\{y,z\in \BB(x,R)\} \, \mathcal{H}^d(\dint x) \\
& =  \int_{\hb}  \mathbf{1}\{y,z\in B\} \,  \mu_{\rm hb}(\dint B) ,
\end{align*}
which shows (b), while (c) is the last equality.
\end{proof}

For the analysis of asymptotic covariances involving the surface area, we need Lemma \ref{lem:integration_horoball}. The following lemma prepares the proof of Lemma \ref{lem:integration_horoball} and shows that generically the intersections of the boundaries considered in Lemma \ref{lem:nullsets},  (a) and (b), have measure zero.

\begin{lemma}\label{lem:nullsets}
Let $A\in\cK^d$ and $u\in\mathbb{S}^{d-1}_\sfp$. For $R\geq 0$, let $x_R=\exp_{\sfp}(Ru)$.  Then the following is true.
\begin{enumerate}
\item[{\rm (a)}] $\mathcal{H}^{d-1}(\varrho\partial A\cap \partial
B)=0$ for $\lambda$-a.e. $\varrho\in \mathcal{I}_d$, for every $B\in \hb$.
\item[{\rm (b)}] $\mathcal{H}^{d-1}(\partial A\cap \partial B)=0$ for $\mu_{\rm hb}$-a.e. $B\in\hb$.
\item[{\rm (c)}] $\mathcal{H}^{d-1}(\varrho A\cap \partial \mathbb{B}(x_R,R))\to  \mathcal{H}^{d-1}(\varrho A\cap \partial \mathbb{B}_{u,0})$ as $R\to\infty$, for
$\lambda$-a.e. $\varrho\in\mathcal{I}_d$.
\end{enumerate}
\end{lemma}

\begin{proof}
(a) Clearly, $\partial A$ is $(d-1)$-rectifiable and $\partial B\cap \mathbb{B}_R$ is  $(d-1)$-rectifiable for each $R>0$ and every $B\in \hb$. Hence, for each $R>0$ and every $B\in \hb$, \cite[Theorem  5.15]{Brothers66} yields
$$
\mathcal{H}^{d-2}(\varrho \partial A\cap(\partial B\cap \mathbb{B}_R))<\infty\qquad \text{for $\lambda$-a.e. $\varrho\in \mathcal{I}_d$.}
$$
Therefore, for each $R>0$ and every $B\in \hb$ we get
$$
\mathcal{H}^{d-1}(\varrho \partial A\cap\partial B\cap \mathbb{B}_R)=0\qquad \text{for $\lambda$-a.e. $\varrho\in \mathcal{I}_d$,}
$$
which implies (a) by considering an increasing sequence of radii $R$.

(b) We use the isometry invariance of $\mu_{\rm hb}$ and $\mathcal{H}^{d-1}$ and Fubini's theorem to see that
\begin{align*}
& \int_{\hb}\mathcal{H}^{d-1}(\partial A\cap \partial B)\, \mu_{\rm hb}(\dint B)\\
&=\int_{\mathcal{I}_d}\int_{\hb}V_d(\mathbb{B}_1)^{-1}\mathbf{1}\{\varrho \sfp \in \mathbb{B}_1\}\mathcal{H}^{d-1}(\varrho \partial A\cap \partial B)\, \mu_{\rm hb}(\dint B)\, \lambda(\dint \varrho)\\
&=\int_{\hb}V_d(\mathbb{B}_1)^{-1}\int_{\mathcal{I}_d}
\mathbf{1}\{\varrho \sfp \in \mathbb{B}_1\}\mathcal{H}^{d-1}(\varrho \partial A\cap \partial B)\, \lambda(\dint \varrho)
\, \mu_{\rm hb}(\dint B).
\end{align*}
Since for each $B\in \hb$, we already know that $\mathcal{H}^{d-1}(\varrho \partial A\cap \partial B)=0$ for  $\lambda$-a.e. $\varrho\in \hb$, the inner integral is zero, which implies (b).

(c) We denote by $C_{d-1}(L,\cdot)$ the curvature measure of order $d-1$ of $L\in\cK^d$ (see Section \ref{sec:GF} and \cite[Theorem 2.7]{Kohlmann1991}).  The map $L\mapsto C_{d-1}(L,\cdot)$ is weakly continuous on $\cK^d$ \cite[Lemma 2.3, Theorem 2.4]{Kohlmann1994} and
$C_{d-1}(L,\cdot)=\mathcal{H}^{d-1}(\partial L\cap\cdot)$ if $L\in\cK^d$ has nonempty interior (as can be seen from \cite[Theorem 2.7]{Kohlmann1991}).  If $r>0$, then $\mathbb{B}(x_R,R)\cap \mathbb{B}_r\to \mathbb{B}_{u,0}\cap \mathbb{B}_r$ as $R\to\infty$ with respect to the Hausdorff metric (compare \cite[Theorem 12.2.2, Theorem 12.3.4]{SW08} or \cite[Theorem 1.8.10]{Schneider}, \cite[Lemmma 5.1.1]{SW08}, which remain true in hyperbolic space). Also note that curvature measures are locally determined (see \cite[p.~226, Note 11]{Schneider}).  Let $f:\mathbb{H}^d\to\R$ be continuous with compact support contained in $\mathbb{B}_r^\circ$, for some $r>0$. Then we obtain
\begin{align*}
\int_{\partial \mathbb{B}(x_R,R)}f(x)\, \mathcal{H}^{d-1}(\dint x)&=\int_{\mathbb{H}^d}f(x)\, C_{d-1}(\mathbb{B}(x_R,R)\cap \mathbb{B}_r,\dint x)\\
&\to \int_{\mathbb{H}^d}f(x)\, C_{d-1}(\mathbb{B}_{u,0}\cap \mathbb{B}_r,\dint x)
=\int_{\partial \mathbb{B}_{u,0}}f(x)\, \mathcal{H}^{d-1}(\dint x),
\end{align*}
as $R\to\infty$. If $\varrho\in \mathcal{I}_d$ is chosen such that $\mathcal{H}^{d-1}( \partial \varrho A\cap\partial \mathbb{B}_{u,0})=0$, then the Portmanteau theorem applied with the set $\varrho A$ shows that
$$
\mathcal{H}^{d-1}(\varrho A\cap \partial \mathbb{B}(x_R,R))\to  \mathcal{H}^{d-1}(\varrho A\cap \partial \mathbb{B}_{u,0})
$$
as $R\to\infty$. Hence the assertion follows from (a).
\end{proof}

The following lemma will be a crucial ingredient in the proof of Theorem \ref{th:asympcovariance}.

\begin{lemma}\label{lem:integration_horoball}
Let $u\in\mathbb{S}^{d-1}_\sfp$ be fixed and let $A\in\cK^d$. Then
\begin{align}
&\int_{\hb} V_{d}(A\cap B) V_{d-1}(A\cap B) \, \mu_{\rm hb}(\dint B) \nonumber \\
& = \int_{\hb} V_{d}(A\cap B) C_{d-1}(  A, B) \, \mu_{\rm hb}(\dint B) +  (d-1) \int_{\mathcal{I}_d} \mathbf{1}\{\sfp\in \varrho A\} V_{d}(\varrho A\cap \mathbb{B}_{u,0}) \, \lambda(\dint \varrho) \label{eq:Integration_horoball_1}
\end{align}
and
\begin{align}
&\int_{\hb} V_{d-1}(A\cap B)^2 \, \mu_{\rm hb}(\dint B) \nonumber\\
& = \int_{\hb} C_{d-1}(  A, B)^2 \,\mu_{\rm hb}(\dint B) \nonumber \\
&\qquad \,
 + (d-1) \int_{\mathcal{I}_d} \mathbf{1}\{\sfp\in \varrho A\} \big(2 C_{d-1}(\varrho   A, \mathbb{B}_{u,0}) + C_{d-1}(\mathbb{B}_{u,0},\varrho A )\big) \, \lambda(\dint \varrho). \label{eq:Integration_horoball_2}
\end{align}
\end{lemma}

\begin{proof} 
For $R>0$ we define
$$
I(R)\defeq \int_{\mathbb{H}^d} V_{d}(A\cap \BB(x,R)) V_{d-1}(A\cap \BB(x,R)) \, \mathcal{H}^d(\dint x)
$$
and
$$
J(R)\defeq\int_{\mathbb{H}^d} V_{d-1}(A\cap \BB(x,R))^2 \, \mathcal{H}^d(\dint x) .
$$
Since $V_i$ is continuous and $V_i(\varnothing)=0$, for $i\in\{d-1,d\}$,  Lemma \ref{lem:limits_intersection_horoballs}   implies that
$$ \int_{\hb} V_{d}(A\cap B) V_{d-1}(A\cap B) \, \mu_{\rm hb}(\dint B)  = \lim_{R\to\infty} \frac{I(R)}{V_d(\mathbb{B}_R)}
$$
and
$$
\int_{\hb} V_{d-1}(A\cap B)^2 \,\mu_{\rm hb}(\dint B) = \lim_{R\to\infty} \frac{J(R)}{V_d(\mathbb{B}_R)} .
$$
By Lemma \ref{lem:surface} and \eqref{eq:hdlambda},
$$
V_{d-1}( A\cap \BB(x,R)) = C_{d-1}(  A,\BB(x,R)) + C_{d-1}(\BB(x,R), A)\quad\text{for $\mathcal{H}^d$-a.e.\ }x\in\mathbb{H}^d.
$$
Hence,  $I(R)=I_1(R)+I_2(R)$ with
\begin{align*}
I_1(R)&\defeq
\int_{\mathbb{H}^d} V_{d}(  A\cap \BB(x,R)) C_{d-1}(  A,\BB(x,R))  \, \mathcal{H}^d(\dint x),\\
I_2(R)&\defeq
\int_{\mathbb{H}^d} V_{d}(  A\cap \BB(x,R)) C_{d-1}(\BB(x,R), A) \, \mathcal{H}^d(\dint x),
\end{align*}
and $J(R)=J_1(R)+J_2(R)$ with
\begin{align*}
J_1(R)&\defeq
 \int_{\mathbb{H}^d}  C_{d-1}(  A,\BB(x,R))^2 \,\mathcal{H}^d(\dint x),\\
 J_2(R)&\defeq \int_{\mathbb{H}^d}
C_{d-1}(\BB(x,R), A) \big(
2  C_{d-1}(  A,\BB(x,R))+C_{d-1}(\BB(x,R), A)\big)
 \,\mathcal{H}^d(\dint x).
\end{align*}

First we show that
\begin{align*}
 \lim_{R\to\infty} \frac{I_1(R)}{V_d(\mathbb{B}_R)}
& = \int_{\hb} V_{d}(A\cap B) C_{d-1}( A, B) \, \mu_{\rm hb}(\dint B).
\end{align*}
This follows from Lemma \ref{lem:LimitEW_New} with $\varphi(K)\defeq V_d(A\cap K)C_{d-1}(A,K)$ for $K\in\cK^d\cup \hb$, once we have verified that the assumptions of the lemma are satisfied. Clearly, $\varphi$ is measurable, since $(x,K)\mapsto \I_K(x)$ is measurable (see \cite[Theorem 12.2.7]{SW08}) and by Fubini's theorem (see also \cite[Lemma 12.1.2]{SW08}). Since  $0\le \varphi(K)\le V_d(A)V_{d-1}(A)$, we see that $\varphi$ is bounded. Choosing $r_0>0$ such that $A\subseteq \mathbb{B}_{r_0}$, we get  $\varphi(K)=0$ if $K\cap \mathbb{B}_{r_0}=\varnothing$, hence \eqref{eqn:assumption_ball} holds.
Since $V_d(\partial \mathbb{B}_{u,t})=0$, the dominated (or the monotone) convergence theorem yields that  $V_d(A\cap \BB(\exp_\sfp((t+R)u),R))\uparrow V_d(A\cap\horo_{u,t})$ as $R\to\infty$. Moreover, Lemma \ref{lem:nullsets} (b) and the definition of $\mu_{\rm hb}$ show that, for almost all $(u,t)$, we have  $\mathcal{H}^{d-1}(\partial A\cap  \partial \mathbb{B}_{u,t})=0$, and hence also $C_{d-1}(  A, \partial \mathbb{B}_{u,t})=0$. Therefore the dominated (or the monotone) convergence theorem implies that $C_{d-1}(A,\BB(\exp_\sfp((t+R)u),R))\uparrow C_{d-1}(A,\horo_{u,t})$ as $R\to\infty$. Thus we have verified condition \eqref{eqn:continuous_ae}.

By the same reasoning, we obtain
\begin{align*}
\lim_{R\to\infty} \frac{J_1(R)}{V_d(\mathbb{B}_R)}
& = \int_{\hb} C_{d-1}( A, B)^2 \,\mu_{\rm hb}(\dint B).
\end{align*}
Let $u\in\mathbb{S}^{d-1}_\sfp$ be arbitrarily fixed. For $R\geq 0$ we let $x_R=\exp_{\sfp}(Ru)$.
For every $y\in\partial \mathbb{B}(x_R,R)$ there exists a $\tau_y\in\mathcal{I}_d$ such that $\tau_y y=\sfp$ and $\tau_y x_R=x_R$. By an application of \eqref{eq:hdlambda} with $x=x_R$, using the inversion and left invariance of $\lambda$, the isometry invariance of the volume functional, and $\tau_y\BB(x_R,R)=\BB(x_R,R)$, we get
\begin{align*}
I_2(R)& = \int_{\mathcal{I}_d}V_{d}(\varrho A\cap \BB(x_R,R))C_{d-1}(\BB(x_R,R),\varrho A)\, \lambda(\dint\varrho)\\
&=\int_{\mathbb{H}^d} \int_{\mathcal{I}_d} \mathbf{1}\{y\in \varrho A\} V_{d}(\varrho A\cap \BB(x_R,R)) \, \lambda(\dint \varrho) \, C_{d-1}(\BB(x_R,R),\dint y) \\
& = \int_{\mathbb{H}^d} \int_{\mathcal{I}_d}  \mathbf{1}\{\sfp\in \tau_y\varrho A\} V_{d}(\tau_y\varrho A\cap \BB(x_R,R))  \, \lambda(\dint \varrho) \,
   C_{d-1}(\BB(x_R,R),\dint y) \\
& = \int_{\mathbb{H}^d}\int_{\mathcal{I}_d} \mathbf{1}\{\sfp\in \varrho A\} V_{d}(\varrho A\cap \BB(x_R,R)) \, \lambda(\dint \varrho) \,  C_{d-1}(\BB(x_R,R),\dint y) \\
& = V_{d-1}(\BB(x_R,R)) \int_{\mathcal{I}_d} \mathbf{1}\{\sfp\in \varrho A\} V_{d}(\varrho A\cap \BB(x_R,R)) \, \lambda(\dint \varrho).
\end{align*}
Similarly, we obtain
\begin{align*}
&J_2(R)\\
& = \int_{\mathcal{I}_d} \int_{\mathbb{H}^d} \mathbf{1}\{y\in \varrho A\} \big(2 C_{d-1}(\varrho   A, \BB(x_R,R)) + C_{d-1}(\BB(x_R,R),\varrho A )\big) \, C_{d-1}(\BB(x_R,R),\dint y) \, \lambda(\dint \varrho) \allowdisplaybreaks\\
& = \int_{\mathbb{H}^d} \int_{\mathcal{I}_d} \mathbf{1}\{\sfp\in \varrho A\} \big(2  C_{d-1}(\varrho   A, \BB(x_R,R)) + C_{d-1}(\BB(x_R,R),\varrho A )\big) \, \lambda(\dint \varrho) \,  C_{d-1}(\BB(x_R,R),\dint y)\\
& = V_{d-1}( \BB(x_R,R)) \int_{\mathcal{I}_d} \mathbf{1}\{\sfp\in \varrho A\} \big(2  C_{d-1}(\varrho   A, \BB(x_R,R)) + C_{d-1}(\BB(x_R,R),\varrho A )\big) \, \lambda(\dint \varrho).
\end{align*}
Now it follows from \eqref{eq:limit2}, the dominated (or the monotone) convergence theorem, the fact that $V_d(\varrho A\cap \partial \mathbb{B}_{u,0})=0$  for all $\varrho\in \mathcal{I}_d$, and hence $ V_d(\varrho A\cap \BB(x_R,R) )\uparrow V_d(\varrho A\cap \mathbb{B}_{u,0})$,  as $R\to\infty$, that
\begin{align*}
 \lim_{R\to\infty} \frac{I_2(R)}
{V_d(\mathbb{B}_R)}
& =  (d-1) \lim_{R\to\infty}\int_{\mathcal{I}_d} \mathbf{1}\{\sfp\in \varrho A\} V_{d}(\varrho A\cap \BB(x_R,R)) \, \lambda(\dint \varrho) \\
& = (d-1) \int_{\mathcal{I}_d} \mathbf{1}\{\sfp\in \varrho A\} V_{d}(\varrho A\cap \BB_{u,0}) \, \lambda(\dint \varrho).
\end{align*}
Moreover,
\begin{align*}
\lim_{R\to\infty} \frac{J_2(R)}{V_d(\mathbb{B}_R)}
& =  (d-1) \lim_{R\to\infty}\int_{\mathcal{I}_d} \mathbf{1}\{\sfp\in \varrho A\} \big(2  C_{d-1}(\varrho   A, \BB(x_R,R)) + C_{d-1}(\BB(x_R,R),\varrho A )\big) \, \lambda(\dint \varrho) \\
& = (d-1) \int_{\mathcal{I}_d} \mathbf{1}\{\sfp\in \varrho A\} \big(2 C_{d-1}(\varrho   A, \BB_{u,0}) + C_{d-1}(\BB_{u,0},\varrho A)
\big) \, \lambda(\dint \varrho).
\end{align*}
Here we used Lemma \ref{lem:nullsets} (a) for the first summand and Lemma \ref{lem:nullsets} (c)
for the second summand. The dominated convergence theorem can be applied, since $\{\varrho\in\Ih\colon \sfp\in\varrho A\}$ is compact and hence has finite $\lambda$ measure.

Combination of the preceding limit relations proves \eqref{eq:Integration_horoball_1} and \eqref{eq:Integration_horoball_2}.
\end{proof}

\section{Economic covering of hyperbolic space}\label{sec:8}

Euclidean space can be covered by congruent cubes such that any two cubes have disjoint interiors. Removing the upper right boundaries of the cubes, we can even partition Euclidean space.
This fact is used in a crucial way in various arguments in the study of particle processes and random closed sets in Euclidean space (see, e.g., \cite[Sections 4.1, 4.3, 9.2]{SW08}, \cite[Section 3]{HLS}). In hyperbolic space there is no comparable, simple choice of a regular decomposition of space into congruent cells such that each cell intersects only a fixed number of other cells. In the present section, we provide a collection of congruent balls such that the hyperbolic space is covered by these balls and at the same time the number of balls containing any given point of the hyperbolic space is bounded independently of the given point. This result can be used as a substitute of the (almost disjoint) partitioning of Euclidean space.

In the following, we write $|A|$ for the cardinality of a finite set $A$.

\begin{lemma}\label{lem:covering}
There exist a countable set $\mathcal{M}\subset \BHd$ and a constant $c_d\in\mathbb{N}$ such that
$$
\bigcup_{x\in \mathcal{M}} \mathbb{B}(x,1/2) = \BHd
$$
and
\begin{equation}\label{eqn:M_locally_finite}
|\{x\in \mathcal{M}\colon y\in\mathbb{B}(x,1/2) \}| \leq c_d \quad \text{\rm for all }   y\in\BHd.
\end{equation}
\end{lemma}

For the proof of Lemma \ref{lem:covering} we work with  the Poincar\'e half-space model of $d$-dimensional  hyperbolic space, i.e., we consider the upper half-space $\mathbf{H}^d\defeq\mathbb{R}^{d-1}\times (0,\infty)$ equipped with the metric
$$
{\rm d}_{\mathbf{H}^d}(x,y) \defeq 2 \operatorname{arsinh}\bigg(\frac{\|x-y\|}{2\sqrt{x_d y_d}}\bigg)
$$
for $x=(x_1,\hdots,x_d)$, $y=(y_1,\hdots,y_d)\in\mathbf{H}^d$ with the usual Euclidean norm $\|\cdot\|$ (see e.g.\ \cite[Theorem 4.6.1]{Ratcliffe}). For $x\in \mathbf{H}^d$ and $r>0$ we denote by
$$
\mathbb{B}_{\mathbf{H}^d}(x,r)\defeq\{ y\in\mathbf{H}^d\colon {\rm d}_{\mathbf{H}^d}(x,y) \leq r \}
$$
the ball around $x$ with radius $r$ with respect to ${\rm d}_{\mathbf{H}^d}$. Moreover, we use the abbreviation
$$
C^{d-1}(z,s) \defeq \bigtimes_{i=1}^{d-1} [z_i-s,z_i+s]
$$
for $z=(z_1,\hdots,z_{d-1})\in\mathbb{R}^{d-1}$ and $s>0$.

For the proof of Lemma \ref{lem:covering} it is sufficient to construct a corresponding set $\mathcal{M}$ in $\mathbf{H}^d$. Such a set is provided in Lemma \ref{le:8.3}. As a preparation we first establish an auxiliary result in which a geodesic ball in the half-space model is compared to `rectangular boxes'.

\begin{lemma}\label{lem:Proof_Covering}
For all $z\in\mathbb{R}^{d-1}$ and $u\in(0,\infty)$,
$$
C^{d-1}(z, a u) \times [7u/8,u] \subseteq \mathbb{B}_{\mathbf{H}^d}((z,u), 1/2 )\subseteq C^{d-1}(z, 32 u) \times [u/16^2,16^2u]
$$
with
$$
a\defeq\sqrt{\frac{1}{d-1}\bigg(\frac{7\operatorname{sinh}(1/4)^2}{2}-\frac{1}{64}\bigg)}.
$$
\end{lemma}

\begin{proof}
For $x=(x_1,\hdots,x_d)\in C^{d-1}(z, a u) \times [7u/8,u]$, we have
$$
\|x-(z,u)\|\leq \sqrt{(d-1) a^2 u^2 +\frac{u^2}{64}} = \sqrt{\frac{7}{2}} \operatorname{sinh}(1/4) u
$$
and $x_d\geq 7u/8$ so that
\begin{align*}
{\rm d}_{\mathbf{H}^d}\big(x, (z,u) \big) & = 2 \operatorname{arsinh}\bigg( \frac{\|x-(z,u)\|}{2\sqrt{x_d u}} \bigg) \\
& \leq 2 \operatorname{arsinh}\bigg( \frac{\sqrt{\frac{7}{2}} \operatorname{sinh}(1/4) u}{2\sqrt{\frac{7u}{8} u}} \bigg) = 2 \operatorname{arsinh}( \operatorname{sinh}(1/4) )=\frac{2}{4}=\frac{1}{2},
\end{align*}
which shows the first inclusion of the statement.

In the following let $x=(x_1,\hdots,x_d)\notin C^{d-1}(z, 32 u) \times [u/16^2,16^2u] $. For $x_d< u/16^2$, we have
$$
{\rm d}_{\mathbf{H}^d}\big(x, (z,u) \big) = 2 \operatorname{arsinh}\bigg( \frac{\|x-(z,u)\|}{2\sqrt{x_d u}} \bigg) \geq 2 \operatorname{arsinh}\bigg( \frac{\frac{255}{256}u}{2\sqrt{u \frac{u}{16^2}}} \bigg)= 2\operatorname{arsinh}\bigg( \frac{255}{32} \bigg) > 1,
$$
while for $x_d>16^2 u$, we have $x_d-u\ge \frac{1}{2}x_d$ and
\begin{align*}
{\rm d}_{\mathbf{H}^d}\big(x, (z,u) \big) & = 2 \operatorname{arsinh}\bigg( \frac{\|x-(z,u)\|}{2\sqrt{x_d u}} \bigg)
\geq 2 \operatorname{arsinh}\bigg( \frac{x_d-u}{2\sqrt{ x_d  u}} \bigg) \\
& \geq 2 \operatorname{arsinh}\bigg( \frac{x_d}{4\sqrt{ x_d u}} \bigg) =  2 \operatorname{arsinh}\bigg( \frac{\sqrt{x_d}}{4\sqrt{u}} \bigg) \ge 2\operatorname{arsinh}\bigg( \frac{16}{4} \bigg) > 1.
\end{align*}
In the remaining case, we have $(x_1,\hdots,x_{d-1})\notin C^{d-1}(z, 32 u)$ and $x_d\in [u/16^2,16^2u]$, and thus we obtain
\begin{align*}
{\rm d}_{\mathbf{H}^d}\big(x, (z,u) \big) \ge 2\operatorname{arsinh}\bigg( \frac{\|(x_1,\hdots,x_{d-1})-z\|}{2\sqrt{ x_d u}} \bigg) \geq 2\operatorname{arsinh}\bigg( \frac{32u}{2\sqrt{16^2u^2}} \bigg)=2\operatorname{arsinh}(1) > 1.
\end{align*}
Combining the previous three inequalities proves the second inclusion.
\end{proof}

\begin{lemma}\label{le:8.3}
Let the constant $a\in (0,\infty)$ be as in Lemma \ref{lem:Proof_Covering}. Then
$$
\bigcup_{k\in\mathbb{Z}^{d-1}, \ell\in\mathbb{Z}} \mathbb{B}_{\mathbf{H}^d}\bigg(\bigg(\frac{7}{8}\bigg)^\ell (2ak,1),\frac{1}{2}\bigg) = \mathbf{H}^d
$$
and there exists a constant $N_0\in\mathbb{N}$ such that
$$
\bigg|\bigg\{ k\in\mathbb{Z}^{d-1}, \ell\in\mathbb{Z}\colon y \in \mathbb{B}_{\mathbf{H}^d}\bigg( \bigg(\frac{7}{8}\bigg)^\ell (2ak,1),\frac{1}{2}\bigg) \bigg\}\bigg|\leq N_0
$$
for all $y\in\mathbf{H}^d$.
\end{lemma}

\begin{proof}
By Lemma \ref{lem:Proof_Covering} we have
\begin{align*}
& \bigcup_{k\in\mathbb{Z}^{d-1}, \ell\in\mathbb{Z}} \mathbb{B}_{\mathbf{H}^d}\bigg( \bigg(\frac{7}{8}\bigg)^\ell (2ak,1),\frac{1}{2}\bigg) \\
& \supseteq \bigcup_{k\in\mathbb{Z}^{d-1}, \ell\in\mathbb{Z}} C^{d-1}\bigg(2a \bigg(\frac{7}{8}\bigg)^\ell k, a\bigg(\frac{7}{8}\bigg)^\ell\bigg) \times \bigg[\bigg(\frac{7}{8}\bigg)^{\ell+1}, \bigg(\frac{7}{8}\bigg)^\ell \bigg]  = \mathbf{H}^d,
\end{align*}
which is the first statement. For $y=(y_1,\hdots,y_d)\in\mathbf{H}^d$, it follows from Lemma \ref{lem:Proof_Covering} that
\begin{align*}
& \bigg|\bigg\{ k\in\mathbb{Z}^{d-1}, \ell\in\mathbb{Z}\colon y \in \mathbb{B}_{\mathbf{H}^d}\bigg( \bigg(\frac{7}{8}\bigg)^\ell (2ak,1),\frac{1}{2}\bigg) \bigg\}\bigg| \\
& \leq \bigg|\bigg\{ k\in\mathbb{Z}^{d-1}, \ell\in\mathbb{Z}\colon y \in C^{d-1}\bigg( 2a\bigg(\frac{7}{8}\bigg)^\ell k, 32\bigg(\frac{7}{8}\bigg)^\ell \bigg) \times \bigg[\frac{1}{16^2} \bigg(\frac{7}{8}\bigg)^\ell, 16^2 \bigg(\frac{7}{8}\bigg)^\ell \bigg] \bigg\}\bigg|.
\end{align*}
Because of $\big[\frac{1}{16^2} \big(\frac{7}{8}\big)^\ell, 16^2 \big(\frac{7}{8}\big)^\ell \big] \subseteq \big[\big(\frac{7}{8}\big)^{\ell+50}, \big(\frac{7}{8}\big)^{\ell-50} \big]$ the number of choices for $\ell$ such that $y_d \in \big[\frac{1}{16^2} \big(\frac{7}{8}\big)^\ell, 16^2 \big(\frac{7}{8}\big)^\ell \big]$ is bounded by 101. For each of these choices for $\ell$ we have
\begin{align*}
& \bigg|\bigg\{ k\in\mathbb{Z}^{d-1}\colon (y_1,\hdots,y_{d-1}) \in C^{d-1}\bigg( 2a \bigg(\frac{7}{8}\bigg)^\ell k, 32 \bigg(\frac{7}{8}\bigg)^\ell \bigg)  \bigg\}\bigg| \\
& \leq \sup_{v\in\mathbb{R}^{d-1}} \bigg|\bigg\{ k\in\mathbb{Z}^{d-1}\colon v \in C^{d-1}\bigg( 2a \bigg(\frac{7}{8}\bigg)^\ell k, 32 \bigg(\frac{7}{8}\bigg)^\ell \bigg)  \bigg\}\bigg| \\
& = \sup_{v\in\mathbb{R}^{d-1}} \big|\big\{ k\in\mathbb{Z}^{d-1}\colon v \in C^{d-1}( 2ak, 32)  \big\}\big|.
\end{align*}
The observation that the right-hand side is finite and does not depend on $\ell$ completes the proof.
\end{proof}

\section{Auxiliary bounds for additive functionals}\label{sec:9}

In this section, we establish several auxiliary bounds for a measurable, additive, and locally bounded functional $\phi: \mathcal{R}^d\to\mathbb{R}$. Hence, in particular these results apply to all geometric functionals. We employ the same arguments as in \cite{BST,HLS} for the Euclidean case, but instead of covering $\mathbb{R}^d$ by translated unit cubes we rely on Lemma \ref{lem:covering}. Recall that $Z$ is a Boolean model generated by a stationary Poisson particle process $\eta$ on $\BH^d$ concentrated on convex particles.

Throughout this section, $\mathcal{M}$ and $c_d$ are as in Lemma \ref{lem:covering}. For $K\in\mathcal{K}^d$, we define
\begin{equation*}
\mathcal{M}_K\defeq\{z\in \mathcal{M}\colon \mathbb{B}(z,1/2)\cap K\neq \varnothing \}.
\end{equation*}
As before, we write $|A|$ for the cardinality of a set $A$.

\begin{lemma}\label{lem:bound_covering_K}
If $K\in\mathcal{K}^d$, then
\begin{equation*}
|\mathcal{M}_K| \leq \frac{c_d}{\Vol(\mathbb{B}_{1/2})}\, \overline{\Vol}(K).
\end{equation*}
\end{lemma}

\begin{proof}
Let $K\in\mathcal{K}^d$ be given.
 If  $x\in \mathcal{M}_K$, then $\mathbb{B}(x,1/2)\subseteq \mathbb{B}(K,1)$. Hence,
\begin{align*}
|\mathcal{M}_K| & = \frac{1}{\Vol(\mathbb{B}_{1/2})} \sum_{x\in \mathcal{M}_K} \Vol(\mathbb{B}(x,1/2)) \\
&= \frac{1}{\Vol(\mathbb{B}_{1/2})} \int_{\mathbb{B}(K,1)} \sum_{x\in \mathcal{M}_K} \mathbf{1}\{y\in\mathbb{B}(x,1/2)\} \, \mathcal{H}^d(\dint y) \\
& \leq \frac{1}{\Vol(\mathbb{B}_{1/2})} \int_{\mathbb{B}(K,1)} \sum_{x\in \mathcal{M}} \mathbf{1}\{y\in\mathbb{B}(x,1/2)\} \, \mathcal{H}^d(\dint y)\\
& \leq\frac{1}{\Vol(\mathbb{B}_{1/2})} \int_{\mathbb{B}(K,1)} c_d \, \mathcal{H}^d(\dint y)
 = \frac{c_d}{\Vol(\mathbb{B}_{1/2})} \mathcal{H}^d(\mathbb{B}(K,1))\\
& = \frac{c_d}{\Vol(\mathbb{B}_{1/2})} \,\overline{\Vol}(K),
\end{align*}
where we used \eqref{eqn:M_locally_finite} in the second inequality.
\end{proof}

Recall that $[K]=\{L\in\mathcal{K}^d\colon L\cap K\neq\varnothing\}$ for $K\in\mathcal{K}^d$. For a given map $\phi: \mathcal{R}^d\to\mathbb{R}$, let  $M(\phi)$ be defined as in \eqref{eqn:M_phi}.

\begin{lemma}\label{lem:bound_phi}
Let $\phi: \mathcal{R}^d\to\mathbb{R}$ be measurable, additive, and locally bounded. Then there exist constants $c_{1,d},c_{2,d}\in(0,\infty)$, depending only on $d$, such that
\begin{equation}\label{eq:lem:bound_phi:1}
|\phi(K)| \leq c_{1,d} M(\phi) \,\overline{\Vol}(K)
\end{equation}
and
$$
|\phi(K\cap Z)| \leq c_{2,d} M(\phi) \sum_{x\in \mathcal{M}_K} 2^{\eta([\mathbb{B}(x,1/2)])}
$$
for $K\in\mathcal{K}^d$. Moreover, one has $c_{1,d}=c_{2,d}c_d/\Vol(\mathbb{B}_{1/2})$.
\end{lemma}

\begin{proof}
From Lemma \ref{lem:bound_covering_K} it follows that
\begin{equation}\label{eqn:bound_intersections_covering}
|\mathcal{M}_{\mathbb{B}(x,1/2)}| \leq \frac{c_d}{\Vol(\mathbb{B}_{1/2})} \,\overline{\Vol}(\mathbb{B}(x,1/2)) = \frac{c_d \Vol(\mathbb{B}_{3/2})}{\Vol(\mathbb{B}_{1/2})}=: c_0
\end{equation}
for all $x\in \mathcal{M}$.

Clearly,  $|\mathcal{M}_K|<\infty$ and $K\subseteq \bigcup_{x\in \mathcal{M}_K} \mathbb{B}(x,1/2)$. Thus, by the additivity of $\phi$,
$$
\phi(K) = \phi\left( K \cap \bigcup_{x\in \mathcal{M}_K} \mathbb{B}(x,1/2)\right) = \sum_{\varnothing\neq U\subseteq \mathcal{M}_K} (-1)^{|U|-1} \phi\left(K\cap \bigcap_{x\in U} \mathbb{B}(x,1/2)\right).
$$
Since $\phi$ is locally bounded, this implies
$$
|\phi(K)| \leq M(\phi) \sum_{\varnothing\neq U\subseteq \mathcal{M}} \mathbf{1}\left\{ \bigcap_{x\in U} \mathbb{B}(x,1/2)\cap K \neq\varnothing \right\}.
$$
From \eqref{eqn:bound_intersections_covering} we know that every element of $\mathcal{M}$ is contained in at most $2^{c_0}$ subsets $U$ of $\mathcal{M}$ with $\bigcap_{x\in U} \mathbb{B}(x,1/2)\neq\varnothing$ so that
\begin{align*}
\sum_{\varnothing\neq U\subseteq \mathcal{M}}  \mathbf{1}\left\{ \bigcap_{x\in U} \mathbb{B}(x,1/2)\cap K \neq\varnothing \right\} & \leq 2^{c_0} \sum_{x\in \mathcal{M}}  \mathbf{1}\{ \mathbb{B}(x,1/2)\cap K \neq\varnothing \} = 2^{c_0} |\mathcal{M}_K| \\
& \leq \frac{2^{c_0} c_d }{\Vol(\mathbb{B}_{1/2})}\, \overline{\Vol}(K),
\end{align*}
where in the last step Lemma \ref{lem:bound_covering_K} was used. This leads to
$$
|\phi(K)| \leq \frac{2^{c_0} c_d }{\Vol(\mathbb{B}_{1/2})}\, M(\phi) \overline{\Vol}(K)=:c_{1,d}  M(\phi) \overline{\Vol}(K),
$$
which proves the first inequality.

As above, the additivity of $\phi$ yields
\begin{align*}
\phi(K\cap Z) &= \phi\left( K\cap Z \cap \bigcup_{x\in \mathcal{M}_K} \mathbb{B}(x,1/2)\right) \\
&= \sum_{\varnothing\neq U\subseteq \mathcal{M}_K} (-1)^{|U|-1} \phi\left(K\cap Z\cap \bigcap_{x\in U} \mathbb{B}(x,1/2)\right).
\end{align*}
For every $\varnothing\neq U\subseteq \mathcal{M}_K$, by the additivity and the local boundedness of $\phi$, we get
\begin{align*}
\left|\phi\left(K\cap Z\cap \bigcap_{x\in U} \mathbb{B}(x,1/2) \right)\right| & = \left| \sum_{\varnothing\neq I\subseteq \eta, |I|<\infty} (-1)^{|I|-1} \phi\left(K\cap \bigcap_{L\in I} L \cap \bigcap_{x\in U} \mathbb{B}(x,1/2)\right) \right| \\
& \leq M(\phi) \sum_{\varnothing\neq I\subseteq \eta, |I|<\infty} \mathbf{1}\left\{ \bigcap_{L\in I} L \cap \bigcap_{x\in U} \mathbb{B}(x,1/2) \neq \varnothing \right\} \\
& \leq M(\phi) \min_{x\in U} 2^{\eta([\mathbb{B}(x,1/2)])}.
\end{align*}
Together with the fact that each element of $\mathcal{M}$ is contained in at most $2^{c_0}$ subsets $U$ of $\mathcal{M}$ with $\bigcap_{x\in U} \mathbb{B}(x,1/2)\neq\varnothing$, we see that
\begin{align*}
|\phi(K\cap Z)| &\leq M(\phi) 2^{c_0} \sum_{x\in \mathcal{M}} 2^{\eta([\mathbb{B}(x,1/2)])} \mathbf{1}\left\{ \mathbb{B}(x,1/2)\cap K \neq\varnothing \right\}\\
&=c_{2,d} M(\phi) \sum_{x\in \mathcal{M}_K} 2^{\eta([\mathbb{B}(x,1/2)])},
\end{align*}
where $c_{2,d}\defeq 2^{c_0} $, which completes the proof.
\end{proof}

For a geometric functional $\phi$ and $W\in\mathcal{K}^d$, the random variable $\phi(Z\cap W)$ is a function of the Poisson process $\eta$. Thus, its $n$-th order difference operator is given by
$$
D^n_{K_1,\hdots,K_n}\phi(Z\cap W) = \sum_{I\subseteq \{1,\hdots,n\}} (-1)^{n-|I|} \phi\bigg(\big(Z\cup\bigcup_{i\in I}K_i\big)\cap W\bigg)
$$
for $K_1,\hdots,K_n\in\mathcal{K}^d$ and $n\in\mathbb{N}$ (see e.g.\ \cite[Equation (18.3)]{LP18}). The iterated difference operators will play a crucial role for our analysis of the variance and covariance structure and for the proof of the univariate central limit theorem. Since $\phi$ is additive and one has the recursive formula $D^n_{K_1,\hdots,K_n}\phi(Z\cap W)=D_{K_n}(D^{n-1}_{K_1,\hdots,K_{n-1}}\phi(Z\cap W))$ for $n\geq 2$, one can show by induction that
\begin{equation}\label{eqn:formula_D_n}
D^n_{K_1,\hdots,K_n}\phi(Z\cap W) = (-1)^n \big( \phi(Z\cap K_1\cap\hdots\cap K_n \cap W) - \phi(K_1 \cap \hdots\cap K_n \cap W) \big).
\end{equation}

\begin{lemma}\label{Lemma5.5}
Let $\phi: \mathcal{R}^d\to\mathbb{R}$ be measurable, additive, and locally bounded. Then the following is true for $\phi^*$ as defined in \eqref{eqn:phi_star}.
\begin{itemize}
\item [{\rm (a)}] If $W\in\mathcal{K}^d$, $n\in\N$,  and $K_1,\hdots,K_n\in\mathcal{K}^d$, then
$$
\E D^n_{K_1,\hdots,K_n}\phi(Z\cap W) = (-1)^{n} \phi^*(K_1\cap\hdots\cap K_n \cap W).
$$
\item [{\rm (b)}] There exists a constant $C\in(0,\infty)$, depending only on $d$, $\gamma$, and $\BQ$, such that
$$
|\phi^*(K)|\leq C M(\phi)\, \overline{\Vol}(K)
$$
for all  $K\in\mathcal{K}^d$.
\end{itemize}
\end{lemma}

\begin{proof}
Part (a) follows directly by taking the expectation in \eqref{eqn:formula_D_n} from the definition of $\phi^*$. For $K\in\mathcal{K}^d$ we deduce from Lemma \ref{lem:bound_phi} and Lemma \ref{lem:bound_covering_K} that
\begin{align*}
|\phi^*(K)| & = | \E\phi(Z\cap K) - \phi(K)| \\
& \leq \E|\phi(Z\cap K)| + |\phi(K)| \\
& \leq M(\phi) \left( c_{2,d} \sum_{x\in \mathcal{M}_K} \E 2^{\eta([\mathbb{B}(x,1/2)])}
 + c_{1,d} \overline{\Vol}(K)\right) \\
& \leq M(\phi) \left( \frac{ c_{2,d} c_d}{\Vol(\mathbb{B}_{1/2})} \E 2^{\eta([\mathbb{B}_{1/2}])} + c_{1,d}\right) \overline{\Vol}(K)\\
&=c_{1,d} \left(   \E 2^{\eta([\mathbb{B}_{1/2}])} + 1\right)M(\phi) \overline{\Vol}(K).
\end{align*} 
This finally proves part (b) of the lemma.
\end{proof}

\begin{lemma}\label{lem:moments_D}
Let $\phi: \mathcal{R}^d\to\mathbb{R}$ be measurable, additive, and locally bounded. Let $m\in\N$. Then there exists a constant $C_{m}\in(0,\infty)$, depending only on $m$, $d$, $\gamma$, and $\BQ$, such that,
\begin{itemize}
\item[{\rm (a)}] for all  $W\in\mathcal{K}^d$,
$$
\E |\phi(Z\cap W)|^m\le c_{1,d}^mM(\phi)^m \E 2^{m\eta([\mathbb{B}_{1/2}])}\,\overline{\Vol}( W)^m,
$$

\item[{\rm (b)}] for all  $K,W\in\mathcal{K}^d$,
$$
\E |D_{K}\phi(Z\cap W)|^m \leq C_{m} M(\phi)^m \, \overline{\Vol}(K\cap W)^m,
$$

\item[{\rm (c)}] for all  $K_1,K_2,W\in\mathcal{K}^d$,
$$
\E |D^2_{K_1,K_2}\phi(Z\cap W)|^m \leq C_{m} M(\phi)^m \, \overline{\Vol}(K_1\cap K_2\cap W)^m.
$$
\end{itemize}
\end{lemma}

\begin{proof}
From  Lemma \ref{lem:bound_phi}, Jensen's inequality, and Lemma \ref{lem:bound_covering_K} we obtain
\begin{align*}
 \E |\phi(Z\cap W)|^m
& \le c_{2,d}^m M(\phi)^m \E\left(\sum_{x\in \mathcal{M}_{W}} 2^{\eta([\mathbb{B}(x,1/2)])} \right)^m\\
& \le c_{2,d}^m M(\phi)^m |\mathcal{M}_{W}|^{m-1}\, \E \sum_{x\in \mathcal{M}_{W}} 2^{m\eta([\mathbb{B}(x,1/2)])} \\
& = c_{2,d}^m M(\phi)^m |\mathcal{M}_{W}|^m \, \E   2^{m\eta([\mathbb{B}_{1/2}])} \\
&\le c_{1,d}^m M(\phi)^m \, \E 2^{m\eta([\mathbb{B}_{1/2}])}\, \overline{\Vol}(W)^m,
\end{align*}
which proves (a).

From \eqref{eqn:formula_D_n}, Jensen's inequality, part (a), and Lemma \ref{lem:bound_phi} it follows that
\begin{align*}
\E |D_{K}\phi(Z\cap W)|^m
& = \E |\phi(Z\cap K\cap W) - \phi(K\cap W)|^m \\
& \leq 2^{m-1} \E |\phi(Z\cap K\cap W)|^m + 2^{m-1} |\phi(K\cap W)|^m \\
& \leq 2^{m-1} c_{1,d}^m M(\phi)^m \, \E 2^{m\eta([\mathbb{B}_{1/2}])}\, \overline{\Vol}(K\cap W)^m \\
& \quad + 2^{m-1} c_{1,d}^m M(\phi)^m\, \overline{\Vol}(K\cap W)^m \\
& = 2^{m-1}c_{1,d}^m\left(   \E 2^{m\eta([\mathbb{B}_{1/2}])} +1  \right) M(\phi)^m \, \overline{\Vol}(K\cap W)^m,
\end{align*}
which proves (b), and similarly
\begin{align*}
\E |D^2_{K_1,K_2}\phi(Z\cap W)|^m
& = \E |\phi(Z\cap K_1\cap K_2 \cap W) - \phi(K_1\cap K_2 \cap W)|^m \\
& \leq  2^{m-1}c_{1,d}^m\left(   \E 2^{m\eta([\mathbb{B}_{1/2}])} +1  \right) M(\phi)^m \,\overline{\Vol}(K_1\cap K_2 \cap W)^m,
\end{align*}
which proves the assertion in (c).
\end{proof}

\section{Proofs of the results for mean values}\label{sec:10}

\begin{proof}[Proof of Theorem \ref{thm:mean}]
Let $W\in\cK^d$ be a fixed observation window. It follows from the additivity of $\phi$ and the inclusion-exclusion formula that
\begin{equation}\label{eq:premeanvalue}
\phi(Z\cap W)=\sum_{n=1}^\infty\frac{(-1)^{n-1}}{n!}\int_{(\mathcal{K}^d)^n} \phi(K_1\cap\ldots\cap K_n\cap W)\, \eta^n_{\neq}(\dint (K_1,\ldots,K_n)),
\end{equation}
where the factorial moment measure $\eta^n_{\neq}$ corresponds to the summation over all $n$-tuples of distinct particles of $\eta$. Since $W$ is compact, for each realisation of the Poisson particle process $\eta$, only finitely many particles hit $W$, that is, $\eta([W])<\infty$ holds $\BP$-a.s. Therefore, the sum on the right side of \eqref{eq:premeanvalue} has $\BP$-a.s.~only finitely many non-zero summands. In the following, we use the shorthand notation
$$
A^{(1)}\defeq \mathbb{B}(A,1)\quad \text{for $A\in\mathcal{K}^d$.}
$$
For $\varrho\in \calI_d$ we have $(\varrho A)^{(1)}=\varrho A^{(1)}$. Let $\TG_1,\ldots,\TG_n$ denote independent copies of the typical particle with distribution $\BQ$. By \eqref{eq:lem:bound_phi:1}, Theorem \ref{thmdisintegration}, and Lemma \ref{lem:intelem} we obtain
\begin{align}
&\int_{(\mathcal{K}^d)^n} |\phi(K_1\cap\ldots\cap K_n\cap W)|\, \Lambda^n(\dint (K_1,\ldots,K_n))\nonumber\\
&\le c_{1,d}M(\phi)\int_{(\mathcal{K}^d)^n} \overline{\Vol}(K_1\cap\ldots\cap K_n\cap W) \, \Lambda^n(\dint (K_1,\ldots,K_n))\nonumber\\
&\le c_{1,d}M(\phi) \E \int_{(\mathcal{I}_d)^n} \Vol\left(\varrho_1 \TG_1^{(1)}\cap\ldots\cap\varrho_n \TG_n^{(1)}\cap W^{(1)}\right)\, \lambda^n(\dint(\varrho_1,\ldots,\varrho_n))\nonumber\\
&=c_{1,d}M(\phi)\left(\E \overline{\Vol}(\TG)\right)^n\overline{\Vol}(W).\label{eq:bound1}
\end{align}
Since $\E\eta^n_{\neq}=\Lambda^n$ by the multivariate Mecke formula, summation and expected value can be interchanged when $\E$ is applied to \eqref{eq:premeanvalue}. Thus we get
the assertion of Theorem \ref{thm:mean} (a).

\medskip

For the proof of Theorem \ref{thm:mean} (b), we assume that $\phi$ is continuous and choose $W=\BB_R$ for $R>0$. There exists a constant $c_{3,d}$ only depending on $d$ such that $\overline{\Vol}(\BB_R)/ {\Vol}(\BB_R)\le c_{3,d}$ for all $R\ge 1$. Thus, it follows from \eqref{eq:bound1} with $W=\BB_R$ that summation and limit $R\to\infty$ can be interchanged (provided the limit on the right-hand side exists) so that
\begin{align*}
&\lim_{R\to\infty}
\frac{\E \phi(Z\cap \BB_R)}{\Vol(\BB_R)}\\
&=\sum_{n=1}^\infty\frac{(-1)^{n-1}}{n!}\lim_{R\to\infty}\frac{1}{\Vol(\BB_R)}\int_{(\mathcal{K}^d)^n} \phi(K_1\cap\ldots\cap K_n\cap \BB_R)\, \Lambda^n(\dint (K_1,\ldots,K_n)).
\end{align*}

In the following, we confirm that the limit exists and determine its value. Next we repeatedly apply  Fubini's theorem (without further mention), which is justified by \eqref{eq:bound1} and use Theorem \ref{thmdisintegration}, the invariance of $\Lambda$, the invariance of $\phi$, the inversion invariance of $\lambda$, and \eqref{eq:hdlambda} to get
\begin{align}
&\frac{1}{\Vol(\BB_R)}\int_{(\mathcal{K}^d)^n} \phi(K_1\cap\ldots\cap K_n\cap \BB_R)\, \Lambda^n(\dint (K_1,\ldots,K_n))\nonumber\\
&=\frac{\gamma}{\Vol(\BB_R)}\int_{(\mathcal{K}^d)^{n-1}} \int_{\mathcal{K}^d_{\mathsf{p}}} \int_{\mathcal{I}_d} \phi(\varrho G\cap K_2\cap\ldots\cap K_n\cap \BB_R)\, \lambda(\dint \varrho)\, \BQ(\dint G)\, \Lambda^{n-1}(\dint (K_2,\ldots,K_n))\nonumber \allowdisplaybreaks\\
&=\frac{\gamma}{\Vol(\BB_R)} \int_{\mathcal{I}_d} \int_{(\mathcal{K}^d)^{n-1}} \int_{\mathcal{K}^d_{\mathsf{p}}} \phi(\varrho G\cap \varrho K_2\cap\ldots\cap \varrho K_n\cap \BB_R)\, \BQ(\dint G)\, \Lambda^{n-1}(\dint (K_2,\ldots,K_n))\, \lambda(\dint \varrho)\nonumber \allowdisplaybreaks\\
&=\frac{\gamma}{\Vol(\BB_R)}\int_{(\mathcal{K}^d)^{n-1}} \int_{\mathcal{K}^d_{\mathsf{p}}} \int_{\mathcal{I}_d} \phi(  G\cap   K_2\cap\ldots\cap  K_n\cap \varrho^{-1}\BB_R)\, \lambda(\dint \varrho)\, \BQ(\dint G)\, \Lambda^{n-1}(\dint (K_2,\ldots,K_n))\nonumber \allowdisplaybreaks\\
&=\frac{\gamma}{\Vol(\BB_R)}\int_{(\mathcal{K}^d)^{n-1}}  \int_{\mathcal{K}^d_{\mathsf{p}}} \int_{\mathcal{I}_d} \phi(  G\cap   K_2\cap\ldots\cap  K_n\cap \BB(\varrho\sfp,R))\, \lambda(\dint \varrho)\, \BQ(\dint G)\, \Lambda^{n-1}(\dint (K_2,\ldots,K_n))\nonumber\\
&=\frac{\gamma}{\Vol(\BB_R)}\int_{(\mathcal{K}^d)^{n}} \int_{\BHd} \phi(G\cap   K_2\cap\ldots\cap  K_n  \cap  \BB(x,R))\, \mathcal{H}^d(\dint x)\, (\BQ\otimes\Lambda^{n-1})(\dint (G,K_2,\ldots,K_n))\nonumber.
\end{align}
In order to apply the dominated convergence theorem to interchange the limit for $R\to\infty$ and the first integral, we provide an integrable upper bound for
$$
(G,K_2,\ldots,K_n)\mapsto  \int_{\BHd} \frac{1}{\Vol(\BB_R)}|\phi(G\cap   K_2\cap\ldots\cap  K_n \cap  \BB(x,R))|\, \mathcal{H}^d(\dint x)
$$
which is independent of $R$. Using the shorthand notation $L= G\cap   K_2\cap\ldots\cap  K_n$, \eqref{eq:hdlambda}, \eqref{eq:lem:bound_phi:1}, and Lemma \ref{lem:intelem}, we get
\begin{align}
&\int_{\BHd} \frac{1}{\Vol(\BB_R)}|\phi(G\cap   K_2\cap\ldots\cap K_n  \cap  \BB(x,R))|\, \mathcal{H}^d(\dint x)\nonumber\\
&\le c_{1,d}M(\phi)\int_{\mathcal{I}_d}  {\Vol(\BB_R)}^{-1}{\Vol}(L^{(1)}  \cap  \varrho\BB_{R+1})\, \lambda(\dint \varrho)\nonumber\\
&=c_{1,d}M(\phi)  \Vol(L^{(1)})\Vol(\BB_{R+1})\Vol(\BB_R)^{-1}\nonumber\\
&\le c_{1,d} c_{3,d}M(\phi)\overline{\Vol}(G\cap   K_2\cap\ldots\cap  K_n )\label{eq:modhelp2}
\end{align}
with the constant $c_{3,d}$ from above.
The integrability with respect to $\BQ\otimes \Lambda^{n-1}$ now follows as in the derivation of \eqref{eq:bound1}. Thus, we have shown that
\begin{align*}
&\lim_{R\to\infty}\frac{1}{\Vol(\BB_R)}\int_{(\mathcal{K}^d)^n} \phi(K_1\cap\ldots\cap K_n\cap \BB_R)\, \Lambda^n(\dint (K_1,\ldots,K_n)) \\
&=\int_{(\mathcal{K}^d)^n} \lim_{R\to\infty} \frac{\gamma}{\Vol(\BB_R)} \int_{\BHd} \phi(G\cap   K_2\cap\ldots\cap  K_n  \cap  \BB(x,R))\,  \mathcal{H}^d(\dint x)\\
& \qquad\qquad \times \,  (\BQ\otimes\Lambda^{n-1})(\dint (G,K_2,\ldots,K_n)).
\end{align*}
Applying Lemma \ref{lem:limits_intersection_horoballs} with $L= G\cap   K_2\cap\ldots\cap  K_n$ yields
\begin{align*}
& \lim_{R\to\infty}\frac{1}{\Vol(\BB_R)}\int_{(\mathcal{K}^d)^n} \phi(K_1\cap\ldots\cap K_n\cap \BB_R)\, \Lambda^n(\dint (K_1,\ldots,K_n)) \\
& = \gamma \int_{(\mathcal{K}^d)^n} \int_{\hb} \phi(G\cap K_2 \cap \hdots\cap K_n\cap B) \, \mu_{\rm hb}(\dint B) \, (\BQ\otimes\Lambda^{n-1})(\dint(G,K_2,\hdots,K_n)).
\end{align*}

Finally, note that the previous identity remains true with $\phi$ replaced by $|\phi|$ by the same arguments as above. In combination with \eqref{eq:modhelp2}, it therefore follows that
\begin{align*}
 &\int_{(\mathcal{K}^d)^{n-1}}\int_{\mathcal{K}^d_\sfp}  \int_{\hb}
  \left|\phi(K_1\cap K_2 \cap \hdots\cap K_n\cap B)\right|
  \, \mu_{\rm hb}(\dint B) \, \BQ(\dint K_1)\,
 \Lambda^{n-1}(\dint(K_2,\hdots,K_n))<\infty .
\end{align*}
Hence, Fubini's theorem can be applied to change the order in the iterated integral. This observation completes the proof of Theorem \ref{thm:mean}.
\end{proof}

\begin{proof}[Proof of Corollary \ref{cor:expectation_volume_surface_area}]
Theorem \ref{thmdisintegration}, Fubini's theorem, and Lemma \ref{lem:intelem} imply
\begin{align*}
& \int_{(\mathcal{K}^d)^n} \Vol(K_1\cap\hdots\cap K_n\cap W) \, \Lambda^n(\dint(K_1,\hdots,K_n)) \\
& = \gamma^n \int_{(\mathcal{K}_{\sfp}^d)^n} \int_{(\mathcal{I}_d)^n} \Vol(\varrho_1 G_1\cap\hdots\cap \varrho_n G_n\cap W) \, \lambda^n(\dint(\varrho_1,\hdots,\varrho_n)) \, \BQ^n(\dint(G_1,\hdots,G_n))\\
&  = \gamma^n (\E\Vol(\TG))^n \Vol(W)=V_d(W)(\gamma v_d)^n.
\end{align*}
Thus the mean value formula for the volume involving a bounded observation window follows from  Theorem \ref{thm:mean}.  The asymptotic formula is a direct consequence.

Now we aim at   the formulas for the expected value for $V_{d-1}$.
 By Lemma \ref{lem:intelem2},
 \begin{align*}
&\int_{(\mathcal{K}^d)^{n}} V_{d-1}(K_1\cap\ldots\cap K_n\cap W)\, \Lambda^n(\dint(K_1,\ldots,K_n))\\
&=\gamma^n\int_{(\mathcal{K}^d_\sfp)^{n}}\int_{(\Ih)^n}
V_{d-1}(\varrho_1 G_1\cap\ldots\cap\varrho_n G_n\cap W)\,
\lambda^n(\dint(\varrho_1,\ldots,\varrho_n))\, \BQ^n(\dint(G_1,\ldots,G_n))\\
&=\gamma^n\int_{(\mathcal{K}^d_\sfp)^{n}} \left(
V_{d-1}(W)V_d(G_1)\cdots V_d(G_n)+nV_{d-1}(G_1)V_d(W)V_d(G_2)\cdots V_d(G_n)
\right)\\
&\qquad\qquad\qquad \times \, \BQ^n(\dint(G_1,\ldots,G_n))\\
&=\gamma^nV_{d-1}(W)(\E V_d(\TG))^n+nV_d(W)\gamma\E V_{d-1}(\TG) (\gamma\E V_d(\TG))^{n-1}\\
&=V_{d-1}(W)(\gamma v_d)^n+V_d(W)n\gamma v_{d-1}(\gamma v_d)^{n-1}.
\end{align*}
Plugging this into \eqref{eqn:limit_meana}, we obtain the assertion for $\E V_{d-1}(Z\cap W)$. The asymptotic formula is now implied by \eqref{eq:limit2}.
\end{proof}

\begin{remark}{\rm The asymptotic formulas in Corollary \ref{cor:expectation_volume_surface_area} can also be derived directly from Theorem \ref{thm:mean} (b). In that case, Corollary \ref{cor:horint1} (a) is useful in the proof of the formula for $m_{V_d,Z}$, whereas Corollary  \ref{lem:LimitEW2} and Lemma \ref{lem:limits_intersection_horoballs} can be used in the proof of the formula for $m_{V_{d-1},Z}$.}
\end{remark}

\begin{proof}[Proof of Theorem  \ref{Thmmeanvalues}]
  We wish  to apply Theorem \ref{thm:mean} (a). For $n\in\N$ and $K_0,\ldots,K_n\in\mathcal{K}^d$, iterated application of the the kinematic formula \eqref{eq:kinematic} yields
\begin{align*}
\int_{(\mathcal{I}_d)^n} V_k^0(K_0\cap\varrho_1K_1\cap\hdots\cap \varrho_nK_n)
\, \lambda^n(\dint(\varrho_1,\hdots,\varrho_n))
=\sum_{\substack{m_0,\dots ,m_n=k\\ m_0+\cdots +m_n=nd+k}}^{d}
V_{m_0}^0(K_0)\cdots V_{m_n}^0(K_n).
\end{align*}
Therefore,
\begin{align*}
&\int_{(\mathcal{K}^d)^n} V_k^0(K_1\cap\hdots\cap K_n\cap W) \, \Lambda^n(\dint(K_1,\hdots,K_n))\\
&=\gamma^n \int_{(\mathcal{K}^d_{\sfp})^n} \int_{(\mathcal{I}_d)^n} V_k^0(\varrho_1G_1\cap\hdots\cap \varrho_nG_n\cap W) \, \lambda^n(\dint(\varrho_1,\hdots,\varrho_n))
\, \BQ^n(\dint(G_1,\hdots,G_n))\\
&=\gamma^n\sum_{\substack{m_0,\dots ,m_n=k\\ m_0+\cdots +m_n=nd+k}}^{d}
V^0_{m_0}(W)c_{m_1}\cdots c_{m_n}.
\end{align*}
Hence we obtain that
\begin{align*}
S&\defeq \sum^\infty_{n=1}\frac{(-1)^{n-1}}{n!}\int_{(\mathcal{K}^d)^n} V_k^0(K_1\cap\hdots\cap K_n\cap W) \, \Lambda^n(\dint(K_1,\hdots,K_n))\\
&=\sum^d_{m=k}V^0_m(W)\sum^\infty_{n=1}\frac{(-1)^{n-1}\gamma^n}{n!}
\sum_{\substack{m_1,\dots ,m_n=k\\ m_1+\cdots +m_n=nd+k-m}}^{d} c_{m_1}\cdots c_{m_n}.
\end{align*}
We can now follow the proof of Theorem 9.1.3 in \cite{SW08}. For $m=k$ the inner
series above equals
\begin{align*}
\sum^\infty_{n=1}\frac{(-1)^{n-1}\gamma^n}{n!}c_d^n=1-e^{-\gamma c_d}.
\end{align*}
Next we consider $m\in\{k+1,\ldots,d\}$ and let $s\in\{1,\ldots,m-k\}$
denote the number of indices (in the inner series) smaller than $d$ (no other values of $s$ can occur). Then
this series equals
\begin{align*}
\sum^\infty_{n=1}&\frac{(-1)^{n-1}\gamma^n}{n!}\sum^{(m-k)\wedge n}_{s=1}\binom{n}{s}c^{n-s}_d
\sum_{\substack{m_1,\dots ,m_s=k\\ m_1+\cdots +m_s=sd+k-m}}^{d-1} c_{m_1}\cdots c_{m_s}\\
&=\sum^{m-k}_{s=1}\frac{(-1)^{s-1}\gamma^s}{s!}\sum^\infty_{n=s}\frac{(-1)^{n-s}\gamma^{n-s}}{(n-s)!}c_d^{n-s}
\sum_{\substack{m_1,\dots ,m_s=k\\ m_1+\cdots +m_s=sd+k-m}}^{d-1} c_{m_1}\cdots c_{m_s}.
\end{align*}
Together with $c_d=v_d$ this yields the asserted formula for $S$.
\end{proof}

\section{Proofs of the results for variances and covariances}

\subsection{Proofs of the covariance formulas for general geometric functionals}\label{sec:11}

\begin{proof}[Proof of Theorem \ref{Thmcovasymp}]
Recall that $K^{(1)}=\mathbb{B}(K,1)$ denotes the parallel set of $K\in\cK^d$ with distance parameter $1$.
Lemma \ref{lem:moments_D} (a) yields that $\E \phi(Z\cap W)^2, \E \psi(Z\cap W)^2<\infty$. Hence, Theorem 18.6 and (18.8) in \cite{LP18} together with Lemma \ref{Lemma5.5} (a) show that
\begin{equation}\label{eq:standard}
 \Cov(\phi(Z\cap W),\psi(Z\cap W))
 =  \sum_{n=1}^\infty \frac{1}{n!}     \int_{(\mathcal{K}^d)^{n}} (\phi^*\cdot\psi^*)(K_1\cap\hdots\cap K_n \cap W) \,\Lambda^{n}(\dint (K_1,\hdots,K_n)),
\end{equation}
which is Theorem \ref{Thmcovasymp} (a).

For the proof of Theorem \ref{Thmcovasymp} (b), we first observe that
\begin{equation}\label{trivialbound}
 \overline{\Vol}(\varrho G\cap K_2\cap \hdots\cap K_n\cap W )^2\le \overline{\Vol}(G)\Vol(\varrho G^{(1)}\cap K_2^{(1)}\cap\ldots\cap  K_n^{(1)}\cap W^{(1)}),
\end{equation}
since $\overline{\Vol}$ is isometry invariant and increasing (with respect to set inclusion).
Using   Lemma \ref{Lemma5.5} (b), Theorem \ref{thmdisintegration}, \eqref{trivialbound}, and Lemma \ref{lem:intelem},
we get
\begin{align*}
& \int_{(\mathcal{K}^d)^{n}} |\phi^*\cdot\psi^*|(K_1\cap\hdots\cap K_n\cap W ) \, \Lambda^{n}(\dint (K_1,\hdots,K_n))\\
&\le \gamma    \int_{\Ih}\int_{\mathcal{K}^d_{\sfp}} \int_{(\mathcal{K}^d)^{n-1}} C^2M(\phi) M(\psi)\, \overline{\Vol}(\varrho G\cap K_2\cap \hdots\cap K_n\cap W )^2\\
&\qquad \times \Lambda^{n-1}(\dint (K_2,\hdots,K_n))\, \BQ(\dint G)\, \lambda(\dint\varrho)  \allowdisplaybreaks\\
&\le C^2M(\phi) M(\psi) \gamma \int_{\Ih}\int_{\mathcal{K}^d_{\sfp}} \int_{(\mathcal{K}^d)^{n-1}} \, \overline{\Vol}( G )\Vol(\varrho G^{(1)}\cap K_2^{(1)}\cap \hdots\cap K_n^{(1)}\cap W^{(1)} )\\
&\qquad \times \Lambda^{n-1}(\dint (K_2,\hdots,K_n))\, \BQ(\dint G)\, \lambda(\dint\varrho) \allowdisplaybreaks\\
& = C^2M(\phi) M(\psi) \gamma^n \left(\E\overline{\Vol}(\TG)\right)^{n-1}\E \overline{\Vol}(\TG)^2\,\overline{\Vol}(W) .
\end{align*}
Summation yields
\begin{align}
 &\sum_{n=1}^\infty \frac{1}{n!}     \int_{(\mathcal{K}^d)^{n}} |\phi^*\cdot\psi^*|(K_1\cap\hdots\cap K_n\cap W )
\, \Lambda^{n}(\dint (K_1,\hdots,K_n))\nonumber\\
  &\le C^2M(\phi) M(\psi) \sum_{n=1}^\infty \frac{\gamma^n}{n!}\left(\E\overline{\Vol}(\TG)\right)^{n-1}\E \overline{\Vol}(\TG)^2\,\overline{\Vol}(W)\nonumber\\
  &\le C^2M(\phi) M(\psi)\left(e^{\gamma\E\overline{\Vol}(\TG)}-1\right)
    \left(\E\overline{\Vol}(\TG)\right)^{-1}{\E\overline{\Vol}(\TG)^2}\, \overline{\Vol}(W).\label{eq11:upperbound}
\end{align}
From this the assertion in (b) follows with the choices $\psi=\phi$ and $W=\BB_R$.

For the proof of (c), we first fix $n\in\N$. By Fubini's theorem, Theorem \ref{thmdisintegration},
and the isometry invariance of $\Lambda$, $\phi^*$, and $\psi^*$ we have (similarly as before)
\begin{align}\label{e:865}
& \int_{(\mathcal{K}^d)^{n}} (\phi^*\cdot\psi^*)(K_1\cap\hdots\cap K_n\cap W ) \, \Lambda^{n}(\dint (K_1,\hdots,K_n))\\
\notag
&=\gamma\int_{(\mathcal{K}^d)^{n-1}} \int_{\mathcal{K}^d_{\sfp}}\int_{\Ih}
(\phi^*\cdot\psi^*)(G\cap K_2\cap\ldots\cap K_n\cap  \varrho^{-1}W)
\, \lambda(\dint\varrho)\,\BQ(\dint G)\,   \Lambda^{n-1}(\dint (K_2,\hdots,K_n)).
\end{align}
Now we choose $W=\BB_R$ in \eqref{eq:standard}. The
upper bound \eqref{eq11:upperbound} and the boundedness of $\overline{\Vol}(\mathbb{B}_R)/\Vol(\mathbb{B}_R)$ for $R\ge 1$ allow us to interchange limit
$R\to\infty$ and summation over $n$ if we divide by $\Vol(\mathbb{B}_R)$ in \eqref{eq:standard}.
Using \eqref{e:865} and also the inversion invariance of $\lambda$
and \eqref{eq:hdlambda}, we hence arrive at
\begin{align}
& \lim_{R\to\infty} \frac{\Cov(\phi(Z\cap \mathbb{B}_R),\psi(Z\cap \mathbb{B}_R))}{\Vol(\mathbb{B}_R)} \nonumber \\
&=\sum_{n=1}^\infty\frac{\gamma}{n!}\lim_{R\to\infty}\frac{1}{\Vol(\BB_R)}
\int_{(\mathcal{K}^d)^{n-1}} \int_{\mathcal{K}^d_{\sfp}}\int_{\BHd} (\phi^*\cdot\psi^*)(G\cap K_2\cap\ldots\cap K_n\cap  \BB(x,R)) \label{eqn:Intermediate_Asymptotic_Variance}\\
&\qquad\times \mathcal{H}^d(\dint x)\, \BQ(\dint G)\,   \Lambda^{n-1}(\dint (K_2,\hdots,K_n)). \nonumber
\end{align}

To justify that the limit $R\to\infty$ can be applied directly to the
inner integral over $\BHd$ in
\eqref{eqn:Intermediate_Asymptotic_Variance}, we can argue as at
\eqref{eq:modhelp2}, using in addition a trivial upper bound as in
\eqref{trivialbound} and the second moment assumption.

Recall that
$\phi$ and $\psi$ are assumed to be continuous for part (c). Because
of \eqref{eqn:limit_meana} and the absolute convergence of
\eqref{eqn:limit_meana}, due to the dominated convergence theorem
(recall \eqref{eq:bound1}), the maps
$\cK^d\ni K\mapsto \E\nu(Z\cap G\cap K_2\cap \hdots\cap K_n \cap K)$
and $\cK^d\ni K\mapsto \nu(G\cap K_2\cap \hdots\cap K_n \cap K)$ with
$\nu\in\{\phi,\psi\}$ satisfy \eqref{eqn:continuous_ae} (confer the
proof of Lemma \ref{lem:limits_intersection_horoballs}) and
\eqref{eqn:assumption_ball}, where $r_0$ can be chosen as the radius
of the smallest ball centered at $\sfp$ and containing
$G\cap K_2\cap\hdots\cap K_n$. By \eqref{eqn:phi_star} it follows that
$\cK^d\cup \hb\ni K\mapsto \phi^*(G\cap K_2\cap\hdots\cap K_n\cap K)$ and
$\cK^d\cup \hb\ni K\mapsto \psi^*(G\cap K_2\cap\hdots\cap K_n\cap K)$ also
satisfy \eqref{eqn:continuous_ae} and \eqref{eqn:assumption_ball} and
are bounded due to Lemma \ref{Lemma5.5} (b). Hence an application of
Lemma \ref{lem:LimitEW_New} completes the proof of (c).
\end{proof}

\subsection{Proofs of the variance formulas for the volume}\label{sec:12}

In this section, we establish the variance formulas for the volume. We will first establish the local variance formula, from which we derive the asymptotic formula. In addition, we describe how the asymptotic formula can be proved by a direct approach and show how the local variance formula for the volume functional can be verified for Boolean models with compact particles and a  general compact observation window $W$.

\begin{proof}[Proof of Theorem \ref{cor:variance_volume}]
By \eqref{eqn:phi_star} and Corollary \ref{cor:expectation_volume_surface_area}, we have
$$
V_d^*(K)= \E V_d(Z\cap K) - V_d(K) =-V_d(K)e^{-\gamma v_d}
$$
for $K\in\cK^d$. Using the covariogram function $C$ introduced in \eqref{eqn:covariogram} and denoting by $(\TG_i)_{i\in\mathbb{N}}$ independent copies of the typical particle $\TG$, we get from \eqref{eqn:Exact_Covariance}, \eqref{eq:hdlambda}, the inversion invariance of $\lambda$, the isometry invariance of the mean covariogram, and the isometry invariance of $\mathcal{H}^d$ that
\begin{align}\label{eq:reftoproof5.4}
&\Var V_d(Z\cap W)\notag\\
&= e^{-2\gamma v_d}\sum_{n=1}^\infty \frac{\gamma^n}{n!}    \E \int_{(\Ih)^{n} } V_d(\varrho_1 \TG_1\cap\hdots\cap \varrho_n \TG_n \cap W)^2 \,\lambda^{n}(\dint (\varrho_1,\hdots,\varrho_n))\notag\\
&= e^{-2\gamma v_d}\sum_{n=1}^\infty \frac{\gamma^n}{n!}    \E \int_{\BHd} \int_{\BHd} \prod_{i=1}^n\lambda(\{\varrho_i\in\Ih:\varrho_i x\in \TG_i,\varrho_i z\in\TG_i\})\mathbf{1}\{x,z\in W\}\,
\mathcal{H}^d(\dint x)\, \mathcal{H}^d(\dint z) \allowdisplaybreaks\notag\\
&=e^{-2\gamma v_d}\sum_{n=1}^\infty \frac{\gamma^n}{n!} \int_{\BHd} \int_{\BHd} C(x,z)^n \mathbf{1}\{x,z\in W\}\,
\mathcal{H}^d(\dint x)\, \mathcal{H}^d(\dint z) \allowdisplaybreaks \notag\\
&=e^{-2\gamma v_d} \int_{\BHd} \int_{\BHd} \left(e^{\gamma C(x,z)}-1\right)\mathbf{1}\{x,z\in W\}\,
\mathcal{H}^d(\dint x)\, \mathcal{H}^d(\dint z) \allowdisplaybreaks \notag\\
&=e^{-2\gamma v_d} \int_{\Ih} \int_{\BHd} \left(e^{\gamma C(\tau^{-1}\sfp,  z)}-1\right)\mathbf{1}\{\tau^{-1}\sfp,z \in W \}\,\mathcal{H}^d(\dint z)\, \lambda(\dint \tau) \allowdisplaybreaks \notag\\
&=e^{-2\gamma v_d} \int_{\Ih} \int_{\BHd} \left(e^{\gamma C(\sfp, \tau  z)}-1\right)\mathbf{1}\{\tau^{-1}\sfp,z \in W \}\,
\mathcal{H}^d(\dint z) \,\lambda(\dint \tau) \allowdisplaybreaks\notag\\
&=e^{-2\gamma v_d} \int_{\Ih} \int_{\BHd} \left(e^{\gamma C(\sfp,    z)}-1\right)\mathbf{1}\{\tau^{-1}\sfp,\tau^{-1}z \in W \}\,
\mathcal{H}^d(\dint z) \,\lambda(\dint \tau)\notag\\
&=e^{-2\gamma v_d}  \int_{\BHd} \int_{\Ih}\left(e^{\gamma C(\sfp,    z)}-1\right)\mathbf{1}\{\tau \sfp, \tau z \in  W \}\,\lambda(\dint \tau)\,
\mathcal{H}^d(\dint z) ,
\end{align}
which is \eqref{eqn:variance_volume}.

\medskip

The special choice $W=\BB_R=\BB(\sfp,R)$ in \eqref{eqn:variance_volume} and the fact that
for $z\in\BH^d$ we have
\begin{align*}
\int\I\{ \tau\sfp \in \mathbb{B}_R, \tau z\in \mathbb{B}_R \}\,\lambda(\dint \tau) & = \int\I\{ \sfp \in \tau \mathbb{B}_R, z\in \tau \mathbb{B}_R \}\,\lambda(\dint \tau) \\
    &=\int\I\{\sfp \in \mathbb{B}(x,R),z \in \mathbb{B}(x,R) \}\,\mathcal{H}^d(\dint x) =V_d(\mathbb{B}(\sfp,R)\cap \mathbb{B}(z,R))
\end{align*}
yield relation \eqref{eqn:variance_volume_ball}.

\medskip

Note that for all $x,z\in\BHd$ we have
$$
C(x,    z) = \E \lambda(\{\varrho\in \Ih: \varrho x\in\TG, \varrho z\in\TG\}) \leq \E \lambda(\{\varrho\in \Ih: \varrho x\in\TG\})=\E\Vol(\TG)=v_d
$$
so that
\begin{align*}
0 \leq \left( e^{\gamma C(\sfp,    z)}-1 \right) \frac{\mathcal{H}^d\left(\BB_R\cap \BB(z,R)\right)}{\Vol(\BB_R)} & \leq e^{\gamma C(\sfp,    z)}-1 = \sum_{n=1}^\infty \frac{\gamma^n C(\sfp,z)^n}{n!} \nonumber\\
& \leq \gamma C(\sfp,z) \sum_{n=1}^\infty \frac{(\gamma v_d)^{n-1}}{n!}\nonumber\\
&\leq \gamma C(\sfp,z) e^{\gamma v_d}.
\end{align*}
Since for any $x\in\mathbb{H}^d$ we have
\begin{align*}
\int_{\BHd} C(x,z) \, \mathcal{H}^d(\dint z) & = \int_{\BHd} \E \lambda(\{\varrho\in \Ih: \varrho x\in\TG, \varrho z\in\TG\}) \, \mathcal{H}^d(\dint z) \nonumber\\
& = \E \int_{\Ih} \int_{\BHd}  \mathbf{1}\{\varrho x\in\TG\} \mathbf{1}\{z\in\varrho^{-1} \TG\}) \, \mathcal{H}^d(\dint z) \, \lambda(\dint\varrho)\nonumber\\
&= \E \Vol(\TG)^2,
\end{align*}
it follows from \eqref{eqn:variance_volume_ball}, the dominated convergence theorem, and Corollary \ref{cor:horint1} (c) that
\begin{align*}
\lim_{R\to\infty} \frac{\Var V_d(Z\cap \mathbb{B}_R)}{\Vol(\mathbb{B}_R)}
& = e^{-2 v_d}  \int_{\BHd}  \left(e^{\gamma C(\sfp, z)}-1\right) \int_{\hb} \mathbf{1}\{ \sfp,z \in B \} \, \mu_{\rm hb}(\dint B) \, \mathcal{H}^d(\dint z).
\end{align*}
By the definition of $\mu_{\rm hb}$ in \eqref{eq:muhb}, the right-hand side can be rewritten as
\begin{align*}
& \frac{(d-1) e^{-2\gamma v_d}}{\omega_d}  \int_{\BHd}  \left(e^{\gamma C(\sfp, z)}-1\right) \int_{\mathbb{S}^{d-1}_\sfp}\int_\R \mathbf{1}\{ \sfp,z \in \horo_{u,t} \} e^{(d-1)t}\, \dint t\, \mathcal{H}^{d-1}_\sfp(\dint u) \, \mathcal{H}^d(\dint z)\\
& = e^{-2\gamma v_d} \int_{\BHd} \left(e^{\gamma C(\sfp, z)}-1\right) \int_{\mathbb{S}^{d-1}_\sfp}\int_{-\infty}^0  \mathbf{1}\{ z \in \horo_{u,t} \} \frac{(d-1) }{\omega_d} e^{(d-1)t}\, \dint t\, \mathcal{H}^{d-1}_\sfp(\dint u) \, \mathcal{H}^d(\dint z),
\end{align*}
where we used  that $\sfp\in \horo_{u,t}$ if and only if $t\leq 0$. Since  $1/\omega_d$ is the density of the uniform distribution on $\mathbb{S}^{d-1}_\sfp$ and   $t\mapsto \mathbf{1}\{t\leq 0\} e^{(d-1)t}$ is the density of an exponentially distributed random variable with parameter $d-1$ multiplied by $-1$,  relation \eqref{eqn:asymptotic_variance_volume} follows.
\end{proof}

The formula for the asymptotic variance of the volume can be also derived from the general formula for asymptotic covariances in Theorem \ref{Thmcovasymp} (c).

\begin{proof}[Alternative proof of \eqref{eqn:asymptotic_variance_volumeproto}]
We obtain from Theorem \ref{Thmcovasymp} (c) and rewriting the expressions similarly as in the proof of Theorem  \ref{cor:variance_volume} that
\begin{align*}
&\lim_{R\to\infty} \frac{\Var V_d(Z\cap \mathbb{B}_R)}{\Vol(\mathbb{B}_R)}\notag\\
& = e^{-2\gamma v_d} \sum_{n=1}^\infty \frac{\gamma}{n!} \int_{\hb} \int_{\mathcal{K}^d_\sfp} \int_{(\mathcal{K}^d)^{n-1}} V_d(G\cap K_2\cap\hdots\cap K_n\cap B)^2  \notag\\
& \hspace{4cm} \times \Lambda^{n-1}(\dint (K_2,\hdots,K_n)) \, \BQ(\dint G) \, \mu_{\rm hb}(\dint B)\notag\\
& = e^{-2\gamma v_d}\sum_{n=1}^\infty \frac{\gamma^n}{n!} \int_{\hb} \int_{\mathbb{H}^d} \int_{\mathbb{H}^d} \int_{\cK^d_{\sfp}} C(x,z)^{n-1} \mathbf{1}\{x,z\in G\cap B\}\,
 \BQ(\dint G)\, \mathcal{H}^d(\dint x)\, \mathcal{H}^d(\dint z) \,  \mu_{\rm hb}(\dint B) \notag\\
& = e^{-2\gamma v_d}\sum_{n=1}^\infty \frac{\gamma^n}{n!} \int_{\hb} \int_{\mathcal{I}_d} \int_{\mathbb{H}^d} \int_{\cK^d_\sfp} C(\sfp,z)^{n-1} \mathbf{1}\{\sfp,z\in \tau G\cap \tau B\}\,
\BQ(\dint G)\,
\mathcal{H}^d(\dint z)\,  \lambda(\dint \tau) \, \mu_{\rm hb}(\dint B) \allowdisplaybreaks \notag\\
& =  e^{-2\gamma v_d}\sum_{n=1}^\infty \frac{\gamma^n}{n!} \int_{\mathbb{H}^d} C(\sfp,z)^{n-1} \int_{\mathcal{I}_d} \int_{\cK^d_\sfp} \mathbf{1}\{\sfp,z\in \tau G\}\, \BQ(\dint G)\notag\\
&\hspace{4cm} \times  \int_{\hb}  \mathbf{1}\{\sfp,z\in \tau B\}  \, \mu_{\rm hb}(\dint B) \,
\lambda(\dint \tau) \, \mathcal{H}^d(\dint z).
\end{align*}
Finally, an application of Corollary \ref{cor:horint1} (b) (that is, the isometry invariance of $\mu_{\rm hb}$) yields
\begin{align}
\lim_{R\to\infty} \frac{\Var V_d(Z\cap \mathbb{B}_R)}{\Vol(\mathbb{B}_R)}
& = e^{-2\gamma v_d}\sum_{n=1}^\infty \frac{\gamma^n}{n!} \int_{\mathbb{H}^d} C(\sfp,z)^{n-1} \int_{\mathcal{I}_d} \int_{\cK^d_\sfp} \mathbf{1}\{\sfp,z\in \tau G\}\, \BQ(\dint G)\notag\\
&\hspace{4cm} \times  \int_{\hb}  \mathbf{1}\{\sfp,z\in B\} \, \mu_{\rm hb}(\dint B) \,
\lambda(\dint \tau) \, \mathcal{H}^d(\dint z) \allowdisplaybreaks \notag\\
& = e^{-2\gamma v_d}\sum_{n=1}^\infty \frac{\gamma^n}{n!} \int_{\mathbb{H}^d} C(\sfp,z)^{n} \int_{\hb}  \mathbf{1}\{\sfp,z\in B\} \, \mu_{\rm hb}(\dint B) \,
 \mathcal{H}^d(\dint z) \notag\\
& = e^{-2\gamma v_d} \int_{\mathbb{H}^d} \left(e^{\gamma C(\sfp,z)} -1 \right) \int_{\hb}  \mathbf{1}\{\sfp,z\in B\} \, \mu_{\rm hb}(\dint B) \,
 \mathcal{H}^d(\dint z),\label{eq:citelater3}
\end{align}
which completes the proof.
\end{proof}

The formula for the variance of the volume within an arbitrary compact observation window $W$ can be derived as follows for general compact particles, as announced in Remark \ref{rem:general_volume}.

\begin{proof}[Proof of Theorem \ref{cor:variance_volume} for compact sets $W\subset \mathbb{H}^d$ and compact particles]
First, we have
\begin{align*}
\E V_d(Z\cap W)^2
&=\E \int_W\int_W \1\{x,y\in Z\}\, \mathcal{H}^d(\dint x)\, \mathcal{H}^d(\dint y)\\
&=\E \int_W\int_W \left(\1\{x\in Z\}+\1\{y\in Z\}-1+\1\{\{x,y\}\cap Z=\varnothing\}\right)\, \mathcal{H}^d(\dint x)\, \mathcal{H}^d(\dint y)\\
&=2V_d(W)\E V_d(Z\cap W)-V_d(W)^2+\int_W\int_W\BP(\{x,y\}\cap Z=\varnothing)\, \mathcal{H}^d(\dint x)\, \mathcal{H}^d(\dint y).
\end{align*}
Second, by the void probability of the Poisson process $\eta$ and Theorem \ref{thmdisintegration},
\begin{align*}
\BP(\{x,y\}\cap Z=\varnothing)
&=\exp\left(-\gamma\int_{\cK^d_\sfp} \int_{\Ih} \1\{\{x,y\}\cap\varrho G\neq\varnothing\}\, \lambda(\dint \varrho)\, \BQ(\dint G)\right)\\
&=\exp\left(-\gamma\int_{\cK^d_\sfp} \int_{\Ih} \left(\1\{x\in \varrho G\} + \1\{y\in \varrho G\} - \1\{x,y\in \varrho G\} \right)\, \lambda(\dint \varrho)\, \BQ(\dint G)\right)\\
&=e^{-2\gamma v_d}e^{\gamma C(x,y)}.
\end{align*}
Moreover, we obtain from Corollary \ref{cor:expectation_volume_surface_area} and the discussion thereafter
$$
\E V_d(Z\cap W)=V_d(W)\left(1-e^{-\gamma v_d}\right).
$$
Combining these ingredients, we get
\begin{align*}
\Var V_d(Z\cap W)
&=\E V_d(Z\cap W)^2-\left(\E V_d(Z\cap W)\right)^2\\
&=2V_d(W)V_d(W)\left(1-e^{-\gamma v_d}\right)-V_d(W)^2-V_d(W)^2\left(1-e^{-\gamma v_d}\right)^2\\
&\quad +
e^{-2\gamma v_d}\int_W\int_We^{\gamma C(x,y)}\, \mathcal{H}^d(\dint x)\, \mathcal{H}^d(\dint y)\\
&=-V_d(W)^2e^{-2\gamma v_d}+
e^{-2\gamma v_d}\int_W\int_We^{\gamma C(x,y)}\, \mathcal{H}^d(\dint x)\, \mathcal{H}^d(\dint y)\\
&=e^{-2\gamma v_d}\int_W\int_W\left(e^{\gamma C(x,y)}-1\right)\, \mathcal{H}^d(\dint x)\, \mathcal{H}^d(\dint y), 
\end{align*}
which proves  \eqref{eqn:variance_volume}. From this point onwards the asymptotic formulas for increasing balls can be derived as in the proof of Theorem  \ref{cor:variance_volume} for compact convex particles.
\end{proof}

\subsection{Proofs of Theorems \ref{th:localcovariance} and  \ref{th:asympcovariance}}\label{sec:cov_vol_surf}

Next we turn to the proofs of the local and the asymptotic variances and covariances of volume and surface area.

\begin{proof}[Proof of Theorem \ref{th:localcovariance}]
By \eqref{eqn:phi_star}  and Corollary \ref{cor:expectation_volume_surface_area}, for $A\in\cK^d$ we have
\begin{align}\label{eq:Vstar}
V_d^*(A)=-V_d(A)e^{-\gamma v_d},\quad V_{d-1}^*(A)=V_d(A)\gamma v_{d-1}e^{-\gamma v_d}-V_{d-1}(A)e^{-\gamma v_d} .
\end{align}
Let $W\in\cK^d$. By Theorem \ref{Thmcovasymp} (a) and the proof of Theorem \ref{cor:variance_volume}
we need to consider
\begin{align*}
S(W)&\defeq\sum_{n=1}^\infty \frac{1}{n!}     \int_{(\cK^d)^n} V_{d-1}(K_1\cap\hdots\cap K_n \cap W)V_d(K_1\cap\hdots\cap K_n \cap W)
\,\Lambda^{n}(\dint (K_1,\hdots,K_n))\\
&=\sum_{n=1}^\infty \frac{1}{n!}     \int_{\mathbb{H}^d}\int_{(\cK^d)^n} \int_{\mathbb{H}^d} \I\{y,z\in K_1\cap\hdots\cap K_n \cap W)\\
&\qquad\qquad\qquad\qquad  \times C_{d-1}(K_1\cap\hdots\cap K_n\cap W,dy)
\,\Lambda^{n}(\dint (K_1,\hdots,K_n))\,\mathcal{H}^d(\dint z).
\end{align*}
In the following, we repeatedly apply Theorem \ref{thmdisintegration} (without further mention). By Lemma \ref{lem:surface}
and a symmetry argument we have $S(W)=S_1(W)+S_2(W)$,
where
\begin{align*}
S_1(W)\defeq\sum_{n=1}^\infty &\frac{1}{(n-1)!}    \int_{\mathbb{H}^d}\int_{\cK^d}\int_{(\cK^d)^{n-1}} \int_{\mathbb{H}^d}  \I\{z\in K\cap W\}\\
&\qquad \qquad \times\,
\I\{y\in K_{1}\cap\hdots\cap K_{n-1}\cap W,z\in K_{1}\cap\cdots\cap K_{n-1})\\
&\qquad\qquad\qquad \times\, C_{d-1}(K,\dint y)\,
\,\Lambda^{n-1}(\dint (K_1,\hdots,K_{n-1}))\,\Lambda(\dint K)\,\mathcal{H}^d(\dint z)
\end{align*}
and
\begin{align*}
S_2(W)\defeq\sum_{n=1}^\infty &\frac{1}{n!}     \int_{\mathbb{H}^d}\int_{(\cK^d)^n} \int_{\mathbb{H}^d} \I\{z\in W\}
\I\{y\in K_{1}\cap\hdots\cap K_{n},z\in K_{1}\cap\hdots\cap K_n)\\
&\qquad\qquad\qquad\qquad \times\, C_{d-1}(W,\dint y)\,
\,\Lambda^n(\dint (K_1,\hdots,K_{n}))\,\mathcal{H}^d(\dint z).
\end{align*}
Thus we get
\begin{align*}
S_1(W)&=\sum_{n=1}^\infty \frac{\gamma^{n-1}}{(n-1)!}   \int_{\mathbb{H}^d}\int_{\cK^d}\int_{(\cK^d_\sfp)^{n-1}} \int_{(\Ih)^{n-1}}\int_{\mathbb{H}^d}  \I\{y\in W,z\in K\cap W\}\\
&\qquad\times\,
\I\{y\in\varrho_1 G_{1}\cap\hdots\cap \varrho_{n-1} G_{n-1}\}\I\{z\in \varrho_1G_{1}\cap\hdots\cap \varrho_{n-1}G_{n-1})\\
&\qquad\qquad\times\,
\,C_{d-1}(K,\dint y)\,
\,\lambda^{n-1}(\dint (\varrho_1,\hdots,\varrho_{n-1}))\,\BQ^{n-1}(\dint (G_1,\hdots,G_{n-1}))\,\Lambda(\dint K)\,\mathcal{H}^d(\dint z).
\end{align*}
Since $\lambda$ is inversion invariant, we obtain as in the proof of Theorem \ref{cor:variance_volume} that
\begin{align*}
S_1(W)=\sum_{n=1}^\infty \frac{\gamma^{n-1}}{(n-1)!}\int_{\mathbb{H}^d} \int_{\cK^d}\int_{\mathbb{H}^d}C(y,z)^{n-1}  \I\{y\in W,z\in K\cap W\}
\,C_{d-1}(K,\dint y)\, \Lambda(\dint K)\,\mathcal{H}^d(\dint z).
\end{align*}
Interchanging summation and integration and using that $C(\cdot,\cdot)$ is covariant under isometries, we get
\begin{align*}
S_1(W)&=  \int_{\mathbb{H}^d}\int_{\cK^d}\int_{\mathbb{H}^d} e^{\gamma C(y,z)}\I\{y\in W, z\in K\cap W\}
\,C_{d-1}(K,\dint y)\,\Lambda(\dint K)\,\mathcal{H}^d(\dint z)\\
&= \gamma\int_{\mathbb{H}^d}\int_{\cK^d_\sfp}\int_{\Ih}\int_{\mathbb{H}^d} e^{\gamma C(y,z)}\I\{y\in W, z\in \varrho G,z\in W\}
\,C_{d-1}(\varrho G,\dint y)\,\lambda(\dint \varrho)\,\BQ(\dint G)\,\mathcal{H}^d(\dint z)\\
&= \gamma\int_{\mathbb{H}^d}\int_{\cK^d_\sfp}\int_{\Ih}\int_{\mathbb{H}^d} e^{\gamma C(\varrho y,z)}\I\{\varrho y \in W,z\in W,z\in \varrho G\}
\,C_{d-1}(G,\dint y)\,\lambda(\dint \varrho)\,\BQ(\dint G)\,\mathcal{H}^d(\dint z).
\end{align*}
By Fubini's theorem, the isometry invariance of
the volume measure and the definition of the covariogram function $C_W$   we arrive at
\begin{align}\label{eq:localcov}
S_1(W)=\gamma \int_{\cK^d_\sfp}\int_{\mathbb{H}^d}\int_{\mathbb{H}^d}  e^{\gamma C(y,z)}C_W(y,z) \I\{z\in  G\}
\,C_{d-1}(G,\dint y)\,\mathcal{H}^d(\dint z)\,\BQ(\dint G).
\end{align}
Up to a factor this is the second term on the right-hand side of \eqref{eq:987}.

Similarly,  for the term $S_2(W)$ we obtain that
\begin{align}\label{eq:sec13A}
S_2(W)&= \int_{\mathbb{H}^d}\int_{\mathbb{H}^d}  \big(e^{\gamma C(y,z)}-1\big)\I\{z\in W\} \,C_{d-1}(W,\dint y)\,\mathcal{H}^d(\dint z).
\end{align}
Combination of \eqref{eq:Vstar}, \eqref{eq:localcov},  \eqref{eq:sec13A}, and \eqref{eq:reftoproof5.4} yields relation  \eqref{eq:987}. 

For the proof of \eqref{eq:996} we need to consider
\begin{align*}
V^*_{d-1}(K)^2=e^{-2\gamma v_d}\gamma^2v_{d-1}^2V_d(K)^2
-2 e^{-2\gamma v_d}\gamma v_{d-1}V_{d-1}(K)V_d(K)
+e^{-2\gamma v_d}V_{d-1}(K)^2
\end{align*}
for $K=K_1\cap\ldots \cap K_n\cap W$ and then integrate with respect to $\Lambda^n$.
Our previous calculations lead to the first three terms on the right-hand side
of \eqref{eq:996}. It remains to determine
\begin{align*}
\sum_{n=1}^\infty \frac{1}{n!}     \int_{(\cK^d)^n} V_{d-1}(K_1\cap\hdots\cap K_n \cap W)^2
\,\Lambda^{n}(\dint (K_1,\hdots,K_n)).
\end{align*}
We apply Lemma \ref{lem:surface} and symmetry arguments. This leads to the following four cases. We use repeatedly Theorem \ref{thmdisintegration} and invariance (and covariance) arguments as well as Fubini's theorem.

First, we have
\begin{align*}
&\sum_{n=1}^\infty \frac{n}{n!}\int_{\cK^d}\int_{(\cK^d)^{n-1}}\int_{\mathbb{H}^d}\int_{\mathbb{H}^d}
  \I\{y,z\in W\}\I\{y,z\in K_{1}\cap\hdots\cap K_{n-1}\}\\
&\qquad\qquad\times\,
\,C_{d-1}(K,\dint y)\,C_{d-1}(K,\dint z)\,\Lambda^{n-1}(\dint (K_1,\hdots,K_{n-1}))
\,\Lambda(\dint K)\\
&=\sum_{n=1}^\infty \frac{1}{(n-1)!}\int_{\cK^d}\int_{\mathbb{H}^d}\int_{\mathbb{H}^d} (\gamma C(y,z))^{n-1}\I\{y,z\in W\}\,C_{d-1}(K,\dint y)\,C_{d-1}(K,\dint z)\, \Lambda(\dint K)\\
&=\gamma\int_{\cK^d_\sfp}\int_{\mathbb{H}^d}\int_{\mathbb{H}^d} e^{\gamma C(y,z)}\int_{\Ih}\I\{\varrho y,\varrho z\in W\}
 \, \lambda(\dint\varrho)   \,C_{d-1}(G,\dint y)\,C_{d-1}(G,\dint z)\, \BQ(\dint G).
\end{align*}
Up to the factor $e^{-2\gamma v_d}$, this equals the fourth summand on the right-hand side of \eqref{eq:996}.

For the derivation of the fifth term on the right side of \eqref{eq:996}, we denote by $\TG_1$, $\TG_2$   independent copies of the typical particle. Then we consider\begin{align*}
&\sum_{n=2}^\infty \frac{n(n-1)}{n!}\int_{(\cK^d)^2}\int_{(\cK^d)^{n-2}}
 \int_{\mathbb{H}^d}\int_{\mathbb{H}^d}
  \I\{y,z\in W\}\I\{y,z\in K_{1}\cap\hdots\cap K_{n-2}\}
\I\{y\in L,z\in K\}\\
&\qquad\qquad\qquad \times\,C_{d-1}(K,\dint y)\,C_{d-1}(L,\dint z)\,\Lambda^{n-2}(\dint (K_1,\hdots,K_{n-2}))
\,\Lambda^2(\dint (K,L))\\
&= \int_{(\cK^d)^2} \int_{\mathbb{H}^d}\int_{\mathbb{H}^d}e^{\gamma C(y,z)}\I\{y,z\in W,y\in L,z\in K\}
\,C_{d-1}(K,\dint y)\,C_{d-1}(L,\dint z)\,\,\Lambda^2(\dint (K,L))\\
&= \gamma^2 \E \int_{(\Ih)^2}\int_{\mathbb{H}^d}\int_{\mathbb{H}^d} e^{\gamma C(y,z)}\I\{y,z\in W,y\in \tau \TG_2,z\in \varrho \TG_1\}\\
&\qquad\qquad\qquad \times\,
\,C_{d-1}(\varrho \TG_1,\dint y)\,C_{d-1}(\tau \TG_2,\dint z)\,\lambda^2(\dint (\varrho,\tau))\\
&= \gamma^2 \E \int_{(\Ih)^2}\int_{\mathbb{H}^d}\int_{\mathbb{H}^d} e^{\gamma C(\varrho y,\tau z)}
\I\{\varrho y,\tau z\in W,\varrho y\in \tau \TG_2,\tau z\in \varrho \TG_1\}\\
&\qquad\qquad\qquad \times\,
\,C_{d-1}(\TG_1,\dint y)\,C_{d-1}(\TG_2,\dint z)\,\lambda^2(\dint (\varrho,\tau))\\
&= \gamma^2 \E \int_{(\Ih)^2}\int_{\mathbb{H}^d}\int_{\mathbb{H}^d} e^{\gamma C(\varrho y,z)}
\I\{\tau \varrho y,\tau z\in W,\varrho y\in \TG_2,z\in \varrho \TG_1\}\\
&\qquad\qquad\qquad \times\,
\,C_{d-1}(\TG_1,\dint y)\,C_{d-1}(\TG_2,\dint z)\,\lambda^2(\dint (\varrho,\tau)),
\end{align*}
where we used Fubini's theorem and the left invariance of $\lambda$ in the last step.
Up to the factor $e^{-2\gamma v_d}$, this equals the fifth term on the right-hand side of \eqref{eq:996}.

The seventh summand (up to the exponential factor) comes from
\begin{align*}
&\sum_{n=1}^\infty \frac{1}{n!}
\int_{(\cK^d)^n}\int_{\mathbb{H}^d}\int_{\mathbb{H}^d} \I\{y,z\in K_{1}\cap\ldots\cap K_{n}\}
\,C_{d-1}(W,\dint y)\,C_{d-1}(W,\dint z)\,\Lambda^{n}(\dint (K_1,\hdots,K_{n}))
\end{align*}
and can be derived by similar arguments as above.

The remaining sixth summand is due to
\begin{align*}
&2\sum_{n=1}^\infty \frac{n}{n!}\int_{\cK^d}\int_{(\cK^d)^{n-1}}\int_{\mathbb{H}^d}\int_{\mathbb{H}^d} \I\{y,z\in K_1\cap\ldots\cap K_{n-1}\}\I\{y\in K ,z\in W \} \\
&\qquad\qquad\times\, C_{d-1}(W,\dint y)\,  C_{d-1}(K,\dint z)\,\Lambda^{n-1}(\dint(K_1,\ldots,K_{n-1}))\, \Lambda(\dint K),
\end{align*}
which can be dealt with as before.
\end{proof}

\bigskip

\begin{proof}[Proof of Theorem \ref{th:asympcovariance}]
The proof will be based on Theorem \ref{Thmcovasymp} (c). In view of relation \eqref{eq:Vstar}, it is sufficient to determine for $i,j\in\{d-1,d\}$ the quantities
\begin{align*}
 \widehat\sigma_{i,j}
& \defeq \sum_{n=1}^\infty \frac{\gamma}{n!} \int_{\hb}  \int_{\mathcal{K}^d_\sfp} \int_{(\mathcal{K}^d)^{n-1}} (V_i\cdot V_j)(G\cap K_2\cap\hdots\cap K_n\cap B)  \\
& \qquad \times \Lambda^{n-1}(\dint (K_2,\hdots,K_n)) \, \BQ(\dint G) \, \mu_{\rm hb}(\dint B).
\end{align*}
In the following, we repeatedly use Fubini's theorem, the invariance properties of $\lambda$,  $\Lambda$, and $\mu_{\rm hb}$, and Theorem \ref{thmdisintegration}.

By the derivation of \eqref{eq:citelater3}, we have
\begin{align*}
\widehat{\sigma}_{d,d}& = \int_{\hb} \int_{\BH^d} \mathbf{1}\{\sfp,z\in B\} (e^{\gamma C(\sfp,z)}-1) \, \mathcal{H}^{d}(\dint z) \, \mu_{\rm hb}(\dint B)<\infty.
\end{align*}

Next we turn to $ \widehat\sigma_{d-1,d}$. Let $u\in\mathbb{S}^{d-1}_\sfp$ be fixed. For $n\in\mathbb{N}$, \eqref{eq:Integration_horoball_1} in Lemma \ref{lem:integration_horoball} implies that
\begin{align*}
& \int_{\hb} \int_{\cK_\sfp^d} \int_{(\cK^d)^{n-1}} V_{d}(G\cap K_2\cap \hdots\cap K_n\cap B) V_{d-1}(G\cap K_2\cap \hdots\cap K_n\cap B)\\
& \hspace{3cm} \times \Lambda^{n-1}(\dint(K_2,\hdots,K_n)) \, \BQ(\dint G) \, \mu_{\rm hb}(\dint B) =I_1(n)+I_2(n),
\end{align*}
where
\begin{align*}
I_1(n)\defeq& \int_{\hb} \int_{\cK_\sfp^d} \int_{(\cK^d)^{n-1}} V_{d}(G\cap K_2\cap \hdots\cap K_n\cap B) C_{d-1}( G\cap K_2\cap \hdots\cap K_n, B) \\
& \hspace{3cm} \times \Lambda^{n-1}(\dint(K_2,\hdots,K_n)) \, \BQ(\dint G)\, \mu_{\rm hb}(\dint B),\\
I_2(n)\defeq&  (d-1) \int_{\mathcal{I}_d} \int_{\cK_\sfp^d} \int_{(\cK^d)^{n-1}} \mathbf{1}\{\sfp\in \varrho (G\cap K_2\cap\hdots\cap K_n)\} V_{d}(\varrho (G\cap K_2\cap\hdots\cap K_n)\cap \BB_{u,0}) \\ &\hspace{3cm} \times \Lambda^{n-1}(\dint(K_2,\hdots,K_n)) \, \BQ(\dint G) \, \lambda(\dint \varrho).
\end{align*}
Using Lemma \ref{lem:surface}, the isometry invariance of $\mu_{\rm hb}$, $\Lambda$, and $V_d$, as well as the isometry covariance of $C_{d-1}$, we can rewrite $I_1(n)$ as
\begin{align*}
&\int_{\hb} \int_{\cK_\sfp^d} \int_{(\cK^d)^{n-1}} \left(C_{d-1}(  G, K_2\cap \hdots\cap K_n\cap B) + \cdots + C_{d-1}(K_n,G\cap K_2\cap \hdots\cap   K_{n-1}\cap B) \right) \\
& \hspace{3cm} \times V_{d}(G\cap K_2\cap \hdots\cap K_n\cap B) \, \Lambda^{n-1}(\dint(K_2,\hdots,K_n)) \, \BQ(\dint G)\, \mu_{\rm hb}(\dint B) \\
& = n \int_{\hb} \int_{\cK_\sfp^d} \int_{(\cK^d)^{n-1}} C_{d-1}(  G, K_2\cap \hdots\cap K_n\cap B) V_{d}(G\cap K_2\cap \hdots\cap K_n\cap B) \\
& \hspace{3cm} \times  \Lambda^{n-1}(\dint(K_2,\hdots,K_n)) \, \BQ(\dint G)\, \mu_{\rm hb}(\dint B) \\
& = n \E \int_{\hb}  \int_{\mathbb{H}^d} \int_{  {\rm G}}\int_{(\cK^d)^{n-1}} \mathbf{1}\{y,z\in B\} \mathbf{1}\{ y,z \in K_2\cap \hdots\cap K_n\} \\
& \hspace{3cm} \times  \Lambda^{n-1}(\dint(K_2,\hdots,K_n)) \, \mathcal{H}^d(\dint z) \, C_{d-1}(\TG,\dint y) \, \mu_{\rm hb}(\dint B) \\
& = n \E \int_{\hb} \int_{\mathbb{H}^d} \int_{{\rm G}} \mathbf{1}\{y,z\in B\} \bigg(\int_{\cK^d} \mathbf{1}\{y,z\in K\} \, \Lambda(\dint K) \bigg)^{n-1} \, \mathcal{H}^d(\dint z) \, C_{d-1}(\TG,\dint y) \, \mu_{\rm hb}(\dint B) ,
\end{align*}
hence
$$
I_1(n)
 = n \int_{\hb} \int_{(\BH^d)^2}  \mathbf{1}\{y,z\in B\} (\gamma C(y,z))^{n-1} \, M_{d-1,d}(\dint(y,z)) \, \mu_{\rm hb}(\dint B).
$$

For $I_2(n)$ we get
\begin{align*}
I_2(n) & = \frac{d-1}{\gamma} \int_{(\cK^d)^{n}} \mathbf{1}\{\sfp\in K_1\cap\hdots\cap K_n\} V_{d}(K_1\cap\hdots\cap K_n\cap \BB_{u,0}) \, \Lambda^{n}(\dint(K_1,\hdots,K_n)) \\
& = \frac{d-1}{\gamma} \int_{\BB_{u,0}} \int_{(\cK^d)^{n}} \mathbf{1}\{\sfp,y\in K_1\cap\hdots\cap K_n\} \, \Lambda^{n}(\dint(K_1,\hdots,K_n)) \, \mathcal{H}^d(\dint y) \\
& = \frac{d-1}{\gamma} \int_{\BB_{u,0}} \bigg(\int_{\cK^d} \mathbf{1}\{\sfp,y\in K\} \, \Lambda(\dint K) \bigg)^n \, \mathcal{H}^d(\dint y) = \frac{d-1}{\gamma} \int_{\BB_{u,0}} (\gamma C(\sfp,y))^n \, \mathcal{H}^d(\dint y).
\end{align*}
Summation over $n\in\N$ and addition of these two contributions yields
\begin{align}
\widehat{\sigma}_{d-1,d}
& = \gamma \int_{\hb} \int_{(\BH^d)^2}  \mathbf{1}\{y,z\in B\} e^{\gamma C(y,z)} \, M_{d-1,d}(\dint(y,z)) \, \mu_{\rm hb}(\dint B)\notag\\
&\qquad + (d-1) \int_{\BB_{u,0}} (e^{\gamma C(\sfp,y)}-1) \, \mathcal{H}^d(\dint y).
\label{summandsrhs}
\end{align}
Since, by \eqref{eq:Vstar},
$$\sigma_{d-1,d}=-\gamma v_{d-1}e^{-2\gamma v_d}\widehat{\sigma}_{d,d}+e^{-2\gamma v_d}\widehat{\sigma}_{d-1,d},$$
the asserted formula for $\sigma_{d-1,d}$ is easy to deduce from $\widehat{\sigma}_{d-1,d}$ and $\widehat{\sigma}_{d,d}$.
Moreover, since $\widehat{\sigma}_{d,d}<\infty$ and $ \sigma_{d-1,d}<\infty$, we conclude that $\widehat{\sigma}_{d-1,d}<\infty$ and both (nonnegative) summands on the right side of \eqref{summandsrhs} are  finite as well.

For the determination of
$$\sigma_{d-1,d-1}=\gamma^2 v_{d-1}^2e^{-2\gamma v_d}\widehat{\sigma}_{d,d}-2\gamma v_{d-1}e^{-2\gamma v_d}\widehat{\sigma}_{d-1,d}+e^{-2\gamma v_d}\widehat{\sigma}_{d-1,d-1},
$$
it remains to consider $\widehat{\sigma}_{d-1,d-1}$. The preceding equation already shows that $\widehat{\sigma}_{d-1,d-1}<\infty$. 
Let $n\in\N$. From \eqref{eq:Integration_horoball_2} in Lemma \ref{lem:integration_horoball}  we deduce  that
\begin{align*}
&\int_{\hb} \int_{\cK_\sfp^d} \int_{(\cK^d)^{n-1}} V_{d-1}(G\cap K_2\cap \hdots\cap K_n\cap B)^2 \, \Lambda^{n-1}(\dint(K_2,\hdots,K_n)) \, \BQ(\dint G) \, \mu_{\rm hb}(\dint B)\\
&\qquad=J_1(n)+J_2(n),
\end{align*}
where
\begin{align*}
    J_1(n)&\defeq
 \int_{\cK_\sfp^d} \int_{(\cK^d)^{n-1}} \int_{\hb} C_{d-1}(  G\cap K_2\cap\hdots\cap K_n,B)^2 \,\mu_{\rm hb}(\dint B)  \, \Lambda^{n-1}(\dint(K_2,\hdots,K_n)) \, \BQ(\dint G), \\
J_2(n)&\defeq   (d-1) \int_{\cK_\sfp^d} \int_{(\cK^d)^{n-1}} \int_{\mathcal{I}_d} \mathbf{1}\{\sfp\in \varrho (G\cap K_2\cap\hdots\cap K_n)\} \\
& \qquad \times  \left(2 C_{d-1}(\varrho  (G\cap K_2\cap\hdots\cap K_n), \BB_{u,0}) + C_{d-1}(\BB_{u,0},\varrho (G\cap K_2\cap\hdots\cap K_n) )\right) \\
& \qquad  \times \lambda(\dint \varrho)  \, \Lambda^{n-1}(\dint(K_2,\hdots,K_n)) \, \BQ(\dint G) .
\end{align*}
By Lemma \ref{lem:surface}, the first integral on the right-hand side can be rewritten as
\begin{align*}
&J_1(n)\\
& = \int_{\hb} \int_{\cK_\sfp^d} \int_{(\cK^d)^{n-1}}
C_{d-1}(  G, K_2\cap\hdots\cap K_n\cap B)^2  \\
& \quad  +  (n-1) C_{d-1}(K_2,G\cap  K_3\cap\hdots\cap K_n\cap B)^2\\
& \quad  + 2(n-1) C_{d-1}(  G, K_2\cap\hdots\cap K_n\cap B)C_{d-1}(K_2,G\cap   K_3\cap\hdots\cap K_n\cap B)\\
& \quad  + (n-1)(n-2) C_{d-1}(K_2,G\cap  K_3\cap\hdots\cap K_n\cap B) C_{d-1}(K_3,G\cap K_2\cap K_4\cap\hdots\cap K_n\cap B)\\
& \quad\quad \times \Lambda^{n-1}(\dint(K_2,\hdots,K_n)) \, \BQ(\dint G) \, \mu_{\rm hb}(\dint B) \allowdisplaybreaks \\
& = \int_{\hb} \int_{\cK_\sfp^d} \int_{(\cK^d)^{n-1}} nC_{d-1}(  G, K_2\cap\hdots\cap K_n\cap B)^2 \\
& \quad  + n (n-1) C_{d-1}(  G, K_2\cap\hdots\cap K_n\cap B)C_{d-1}(K_2,G\cap   K_3\cap\hdots\cap K_n\cap B)\\
&\quad\quad \times \Lambda^{n-1}(\dint(K_2,\hdots,K_n)) \, \BQ(\dint G) \, \mu_{\rm hb}(\dint B) \allowdisplaybreaks \\
& = n \int_{\hb} \int_{\cK_\sfp^d} \int_{(\mathbb{H}^d)^2} \mathbf{1}\{y,z\in B\} \left( \int_{\cK^d} \mathbf{1}\{ y,z\in K \} \, \Lambda(\dint K)\right)^{n-1} \, C_{d-1}(G,\cdot)^2(\dint(y,z)) \, \BQ(\dint G) \, \mu_{\rm hb}(\dint B) \\
& \quad + n (n-1) \gamma \int_{\hb} \int_{(\cK_\sfp^d)^2} \int_{\mathcal{I}_d} \int_{\mathbb{H}^d} \int_{\mathbb{H}^d} \mathbf{1}\{y\in B\cap \varrho G_2, \varrho z\in B\cap G_1\} \\
& \quad\quad\times \left( \int_{\cK^d} \mathbf{1}\{ y,\varrho z\in K \} \, \Lambda(\dint K) \right)^{n-2} \,C_{d-1}(G_2,\dint z) \, \,C_{d-1}(G_1,\dint y) \, \lambda(\dint \varrho) \, \BQ^2(\dint (G_1,G_2)) \, \mu_{\rm hb}(\dint B) \allowdisplaybreaks \\
& = n \int_{\hb} \int_{(\mathbb{H}^d)^2} \mathbf{1}\{y,z\in B\} (\gamma C(y,z))^{n-1} \, M_{d-1,d-1}(\dint(y,z)) \,  \mu_{\rm hb}(\dint B)\\
& \quad+ n (n-1) \gamma \E \int_{\hb} \int_{\mathbb{H}^d} \int_{\mathbb{H}^d} \int_{\mathcal{I}_d} \mathbf{1}\{y\in B\cap \varrho {\rm G_2}, \varrho z\in B\cap {\rm G_1}\} (\gamma C(y,\varrho z))^{n-2}\\
& \quad\quad \times  \lambda(\dint \varrho) \, C_{d-1}({\rm G_2},\dint z) \, C_{d-1}({\rm G_1},\dint y) \, \mu_{\rm hb}(\dint B),
\end{align*}
where $\TG_1$ and $\TG_2$ are independent copies of the typical particle.

From Theorem \ref{thmdisintegration} and Lemma \ref{lem:surface} we obtain
\begin{align*}
J_2(n)& =  \frac{d-1}{\gamma}\int_{(\cK^d)^{n}}  \mathbf{1}\{\sfp\in K_1\cap\hdots\cap K_n\} \\
& \qquad \times  \left(2 C_{d-1}( K_1\cap\hdots\cap K_n, \BB_{u,0}) + C_{d-1}( \BB_{u,0},K_1\cap\hdots\cap K_n )\right) \, \Lambda^{n}(\dint(K_1,\hdots,K_n)) \\
& =  \frac{d-1}{\gamma}\int_{(\cK^d)^{n}} \mathbf{1}\{\sfp\in K_1\cap\hdots\cap K_n\} \big(2n C_{d-1}( K_1,K_2\cap\hdots\cap K_n\cap \BB_{u,0}) \\
& \hspace{1.5cm}  + C_{d-1}( \BB_{u,0},K_1\cap\hdots\cap K_n )\big) \, \Lambda^{n}(\dint(K_1,\hdots,K_n)).
\end{align*}
Since
\begin{align*}
& \int_{(\cK^d)^{n}} \mathbf{1}\{\sfp\in K_1\cap\hdots\cap K_n\} C_{d-1}(  K_1, K_2\cap \hdots\cap K_n\cap \BB_{u,0}) \, \Lambda^{n}(\dint(K_1,\hdots,K_n))  \\
& = \gamma \E  \int_{\mathbb{H}^d} \int_{\mathcal{I}_d} \mathbf{1}\{\sfp\in \varrho {\rm G}, \varrho y\in \BB_{u,0}\} \left(\int_{\cK^d} \mathbf{1}\{\sfp\in K, \varrho y \in K\} \, \Lambda(\dint K) \right)^{n-1} \, \lambda(\dint\varrho) \, C_{d-1}(\TG,\dint y)  \\
& = \gamma \E \int_{\mathbb{H}^d} \int_{\mathcal{I}_d} \mathbf{1}\{\varrho \sfp\in {\rm G}, y\in \varrho \BB_{u,0}\} (\gamma C(\varrho \sfp, y))^{n-1} \, \lambda(\dint\varrho) \, C_{d-1}({\rm G},\dint y)
\end{align*}
and
\begin{align*}
& \int_{(\cK^d)^{n}} \mathbf{1}\{\sfp\in K_1\cap\hdots\cap K_n\} C_{d-1}(\mathbb{B}_{u,0},K_1\cap\hdots\cap K_n ) \, \Lambda^{n}(\dint(K_1,\hdots,K_n)) \\
& = \int_{\mathbb{H}^d} \bigg( \int_{\cK^d} \mathbf{1}\{\sfp\in K, z\in K\} \, \Lambda(\dint K) \bigg)^n \, C_{d-1}(\BB_{u,0},\dint z) \\
& = \int_{\mathbb{H}^d} (\gamma C(\sfp,z))^n \, C_{d-1}(\BB_{u,0},\dint z),
\end{align*}
we obtain
\begin{align*}
J_2(n) & = 2(d-1)n \E \int_{\mathbb{H}^d} \int_{\mathcal{I}_d} \mathbf{1}\{\varrho \sfp\in {\rm G}, y\in \varrho \BB_{u,0}\} (\gamma C(\varrho \sfp, y))^{n-1} \, \lambda(\dint\varrho) \, C_{d-1}({\rm G},\dint y) \\
& \quad + \frac{d-1}{\gamma} \int_{\mathbb{H}^d} (\gamma C(\sfp,z))^n \, C_{d-1}(\BB_{u,0},\dint z).
\end{align*}
This implies that
\begin{align*}
\widehat{\sigma}_{d-1,d-1}& = \gamma \int_{\hb} \int_{(\mathbb{H}^d)^2} \mathbf{1}\{y,z\in B\} e^{\gamma C(y,z)} \, M_{d-1,d-1}(\dint(y,z)) \,  \mu_{\rm hb}(\dint B)\\
& \quad + \gamma^2 \E \int_{\hb} \int_{\BH^d} \int_{\BH^d} \int_{\mathcal{I}_d} \mathbf{1}\{y\in B\cap \varrho {\rm G_2}, \varrho z\in B\cap {\rm G_1}\} e^{\gamma C(y,\varrho z)}\\
& \hspace{5cm} \times  \lambda(\dint \varrho) \, C_{d-1}({\rm G_2},\dint z) \, C_{d-1}({\rm G_1},\dint y) \, \mu_{\rm hb}(\dint B) \\
& \quad +2 (d-1) \gamma \E \int_{\BH^d} \int_{\mathcal{I}_d} \mathbf{1}\{\varrho \sfp\in {\rm G}, y\in \varrho \mathbb{B}_{u,0}\} e^{\gamma C(\varrho \sfp, y)}   \, \lambda(\dint\varrho) \, C_{d-1}({\rm G},\dint y) \\
& \quad + (d-1) \int_{\BH^d} (e^{\gamma C(\sfp,z)} -1)  \, C_{d-1}(\BB_{u,0},\dint z).
\end{align*}
Since $\widehat{\sigma}_{d-1,d-1}<\infty$ and the four summands are nonnegative, each of them is finite. 
Now the asserted expression for $\sigma_{d-1,d-1}$ can easily be confirmed.
\end{proof}

\section{Proof of the lower variance bound}\label{sec:14}

\begin{proof}[Proof of Theorem \ref{thm:lower_bound_variance}]
 By the assumption \eqref{eqn:assumption_intersection_phi}, we can choose $r_0\in(0,\infty)$ such that
\begin{equation}\label{eqn:assumption_intersection_phi_r_0}
\int_{\mathcal{K}^d} \int_{(\mathcal{K}^d)^{m}} \mathbf{1}\{K_0\subseteq\mathbb{B}_{r_0}, \phi(K_0\cap K_1\cap\hdots\cap K_m)\neq 0 \} \, \Lambda^{m}(\dint (K_1,\hdots,K_m)) \, \BQ(\dint K_0)>0.
\end{equation}
Moreover, we can assume without loss of generality that $m$ is the smallest nonnegative integer such that \eqref{eqn:assumption_intersection_phi_r_0} holds for the chosen $r_0>0$.
In the sequel, we always assume that $R\geq 2r_0$. 

Since we want to apply \cite[Theorem 1.1]{ST},  we first show that the hypothesis of that theorem is satisfied. 
From \eqref{eqn:formula_D_n} and the invariance of $\phi$ and $Z$ we see that
$$
\E (D_{\varrho K_0}\phi(Z\cap\mathbb{B}_R))^2 = \E (D_{K_0}\phi(Z\cap\mathbb{B}_R))^2
$$
for all $K_0\in\mathcal{K}^d$ and $\varrho\in\mathcal{I}_d$ such that $\varrho K_0\subseteq\mathbb{B}_R$ and $K_0\subseteq \mathbb{B}_R$. Thus, we obtain, by Theorem \ref{thmdisintegration},
\begin{align}\label{eq:sec14A}
J_R  &\defeq\int_{\mathcal{K}^d} \E (D_{K}\phi(Z\cap\mathbb{B}_R))^2 \, \Lambda(\dint K) \notag\\
& = \gamma \int_{\mathcal{I}_d} \int_{\mathcal{K}^d_\sfp} \E (D_{\varrho K_0}\phi(Z\cap\mathbb{B}_R))^2 \, \BQ(\dint K_0)\, \lambda(\dint\varrho) \allowdisplaybreaks \notag\\
& \geq \gamma \int_{\mathcal{K}^d_\sfp} \int_{\mathcal{I}_d} \mathbf{1}\{ K_0\subseteq\mathbb{B}_{r_0}, \varrho K_0\subseteq \mathbb{B}_R \} \E (D_{\varrho K_0}\phi(Z\cap\mathbb{B}_R))^2 \, \lambda(\dint\varrho) \, \BQ(\dint K_0)\notag \\
& \geq \gamma \int_{\mathcal{K}^d_\sfp} \int_{\mathcal{I}_d} \mathbf{1}\{ K_0\subseteq\mathbb{B}_{r_0}, \varrho \mathbb{B}_{r_0}\subseteq \mathbb{B}_R \} \E (D_{K_0}\phi(Z\cap\mathbb{B}_R))^2 \, \lambda(\dint\varrho) \, \BQ(\dint K_0) \allowdisplaybreaks \notag\\
& = \gamma \int_{\mathcal{K}^d_\sfp} \mathbf{1}\{ K_0\subseteq\mathbb{B}_{r_0} \} \lambda(\{\varrho\in\mathcal{I}_d\colon \varrho \mathbb{B}_{r_0}\subseteq  \mathbb{B}_R \})  \E (D_{K_0}\phi(Z\cap\mathbb{B}_R))^2 \, \BQ(\dint K_0) \notag\\
& = \gamma \Vol(\mathbb{B}_{R-r_0}) \int_{\mathcal{K}^d_\sfp} \mathbf{1}\{ K_0\subseteq\mathbb{B}_{r_0} \} \E (D_{K_0}\phi(Z\cap\mathbb{B}_R))^2 \, \BQ(\dint K_0) =: \gamma \Vol(\mathbb{B}_{R-r_0}) \widetilde{J}_R.
\end{align}
Let $K_0\in\mathcal{K}^d$ satisfy $K_0\subseteq\mathbb{B}_R$. If $m\in\mathbb{N}$ and $K_0\in\mathcal{K}^d$ is only intersected by $K_1,\hdots,K_m\in\eta$, we obtain from \eqref{eqn:formula_D_n} that
$$
D_{K_0}\phi(Z\cap \mathbb{B}_R) = \phi(K_0 \cap \mathbb{B}_R) - \phi(K_0 \cap Z\cap \mathbb{B}_R) = \phi(K_0) - \phi(K_0 \cap \bigcup_{j=1}^m K_j )
$$
so that
\begin{align*}
\E (D_{K_0}\phi(Z\cap\mathbb{B}_R))^2
& \geq \E \mathbf{1}\{ \eta([K_0])=m \} (D_{K_0}\phi(Z\cap\mathbb{B}_R))^2 \\
& = \frac{e^{-\Lambda([K_0])}}{m!} \int_{[K_0]^{m}} \big(\phi(K_0) - \phi(K_0 \cap \bigcup_{j=1}^m K_j )\big)^2 \, \Lambda^m(\dint(K_1,\hdots,K_m)),
\end{align*}
where \cite[Theorem 3.2.3]{SW08} was used.
This leads to
\begin{align*}
\widetilde{J}_R& \geq \int_{\mathcal{K}^d_\sfp} \mathbf{1}\{K_0\subseteq \mathbb{B}_{r_0}\} \frac{e^{-\Lambda([K_0])}}{m!} \int_{[K_0]^{m}} \big(\phi(K_0) - \phi(K_0 \cap \bigcup_{j=1}^m K_j )\big)^2 \\
&\qquad \times \Lambda^m(\dint(K_1,\hdots,K_m))   \, \BQ(\dint K_0).
\end{align*}
The additivity of $\phi$ yields
$$
\phi(K_0) - \phi(K_0 \cap \bigcup_{j=1}^m K_j ) = \phi(K_0) - \sum_{\varnothing\neq I\subseteq \{1,\hdots,m\}} (-1)^{|I|-1} \phi(K_0\cap\bigcap_{i\in I}K_i).
$$
By the assumption that $m$ is the minimal number satisfying \eqref{eqn:assumption_intersection_phi_r_0}, the terms of the right-hand side with the exception of the one for $I=\{1,\hdots,m\}$ become zero for $\BQ\otimes\Lambda^{m}$-a.a.\ $(K_0,K_1,\hdots,K_m)\in \cK^d_\sfp\times (\mathcal{K}^d)^{m}$ such that $K_0\subseteq  \mathbb{B}_{r_0}$. For $m=0$, we have
$$
\E (D_{K_0}\phi(Z\cap\mathbb{B}_R))^2 \geq \E \mathbf{1}\{ \eta([K_0])=0 \} (D_{K_0}\phi(Z\cap\mathbb{B}_R))^2 = e^{-\Lambda([K_0])} \phi(K_0)^2.
$$
Hence we get
\begin{align}\label{eq:sec14B}
\widetilde{J}_R &\geq \int_{\mathcal{K}^d_\sfp} \mathbf{1}\{K_0\subseteq \mathbb{B}_{r_0}\} \frac{e^{-\Lambda([K_0])}}{m!} \int_{(\mathcal{K}^d)^{m}} \phi(K_0\cap K_1\cap\hdots\cap K_m)^2 \notag\\
&\qquad \times \Lambda^m(\dint(K_1,\hdots,K_m))   \, \BQ(\dint K_0) =: v_{\phi,\BQ}.
\end{align}
From \eqref{eq:sec14A} and \eqref{eq:sec14B} it follows that 
\begin{equation}\label{eq:sec14C}
J_R =\int_{\mathcal{K}^d} \E (D_{K}\phi(Z\cap\mathbb{B}_R))^2 \, \Lambda(\dint K)\ge \gamma v_{\phi,\BQ}\Vol(\mathbb{B}_{R-r_0}).
\end{equation}

Let $\TG_1$ and $\TG_2$ be independent copies of the typical particle $\TG$. From Lemma \ref{lem:moments_D} (c), Theorem \ref{thmdisintegration}, and an application of Lemma \ref{lem:intelem} it follows that
\begin{align}\label{eq:sec14D}
& \int_{(\mathcal{K}^d)^2} \E (D^2_{K_1,K_2}\phi(Z\cap\mathbb{B}_R))^2 \, \Lambda^2(\dint(K_1,K_2)) \notag\\
& \leq  C_{2} M(\phi)^2 \int_{(\mathcal{K}^d)^2} \overline{\Vol}(K_1\cap K_2 \cap \mathbb{B}_R)^2 \, \Lambda^2(\dint(K_1,K_2)) 
\notag\\
& \leq \gamma^2 C_{2} M(\phi)^2 \E \overline{\Vol}(\TG_1) \int_{(\mathcal{I}_d)^2} \Vol(\varrho_1 \mathbb{B}(\TG_1,1) \cap \varrho_2 \mathbb{B}(\TG_2,1) \cap \mathbb{B}_{R+1}) \, \lambda^2(\dint(\varrho_1,\varrho_2)) \notag\\
& = \gamma^2 C_{2} M(\phi)^2 \E \overline{\Vol}(\TG_1) \Vol(\mathbb{B}(\TG_1,1)) \Vol(\mathbb{B}(G_2,1)) \Vol(\mathbb{B}_{R+1}) \notag\\
& = \gamma^2 C_{2} M(\phi)^2 \E \overline{\Vol}(\TG)^2 \E\overline{\Vol}(\TG) \Vol(\mathbb{B}_{R+1}).
\end{align}
Relations \eqref{eq:sec14C} and \eqref{eq:sec14D} imply that there exists a constant $\overline{C}\in(0,\infty)$ such that
$$
\int_{(\mathcal{K}^d)^2} \E (D^2_{K_1,K_2}\phi(Z\cap\mathbb{B}_R))^2 \, \Lambda^2(\dint(K_1,K_2)) \leq \overline{C} \int_{\mathcal{K}^d} \E (D_{K}\phi(Z\cap\mathbb{B}_R))^2 \, \Lambda(\dint K)
$$
for $R$ sufficiently large.

Hence, by an application of  \cite[Theorem 1.1]{ST} and using \eqref{eq:sec14C}, we obtain  
$$
\liminf_{R\to\infty}\frac{\Var \phi(Z\cap \mathbb{B}_R)}{V_d(\mathbb{B}_R)}\ge \frac{4}{(\overline{C}+2)^2} \liminf_{R\to\infty}\frac{J_R}{V_d(\mathbb{B}_R)}\ge \frac{4}{(\overline{C}+2)^2}
  \gamma e^{-(d-1)r_0}v_{\phi,\BQ}>0,$$
which proves the assertion of Theorem \ref{thm:lower_bound_variance}.
\end{proof}

\section{Proofs of the central limit theorems}\label{sec:15}

The following proof adapts the proof of \cite[Theorem 3.5]{BST} for geometric functionals of Poisson cylinder processes in $\mathbb{R}^d$ to Boolean models in hyperbolic space.

\begin{proof}[Proof of Theorem \ref{thm:CLT_univariate}]
Throughout this proof we use the abbreviations
$$\sigma_R\defeq\sqrt{\Var \phi(Z\cap \mathbb{B}_R)}\qquad \text{and}\qquad F_R\defeq \frac{\phi(Z\cap \mathbb{B}_R) - \E \phi(Z\cap \mathbb{B}_R)}{\sigma_R}.
$$
We first show part (b). It follows from \cite[Theorem 1.1]{LPS} that
$$
\mathbf{d}_{\rm Wass}(F_R,N) \leq T_1+T_2+T_3
$$
with
\begin{align*}
T_1 & = 2 \bigg( \int_{(\mathcal{K}^d)^3} \sqrt{\E (D_{K_1}F_R)^2 (D_{K_2}F_R)^2 } \sqrt{\E (D^2_{K_1,K_3}F_R)^2 (D^2_{K_2,K_3}F_R)^2 } \, \Lambda^3(\dint(K_1,K_2,K_3)) \bigg)^{1/2} ,\\
T_2 & = \bigg( \int_{(\mathcal{K}^d)^3} \E (D^2_{K_1,K_3}F_R)^2 (D^2_{K_2,K_3}F_R)^2  \, \Lambda^3(\dint(K_1,K_2,K_3)) \bigg)^{1/2} ,\\
T_3 & = \int_{\mathcal{K}^d} \E |D_KF_R|^3 \, \Lambda(\dint K).
\end{align*}
From the Cauchy--Schwarz inequality and Lemma \ref{lem:moments_D} (b), (c) it follows that
\begin{equation}\label{eqn:intermediate_bound_T_1}
\begin{split}
T_1^2 \leq  \frac{4   C_{4} M(\phi)^4}{\sigma_R^4} \int_{(\mathcal{K}^d)^3} & \overline{\Vol}(K_1\cap \mathbb{B}_R) \overline{\Vol}(K_2\cap \mathbb{B}_R) \overline{\Vol}(K_1\cap K_3\cap \mathbb{B}_R)\\
& \times \overline{\Vol}(K_2\cap K_3\cap \mathbb{B}_R) \, \Lambda^3(\dint(K_1,K_2,K_3)).
\end{split}
\end{equation}
By Theorem \ref{thmdisintegration}, which will be also applied in the sequel without further mention, the integral on the right-hand side can be rewritten as
\begin{align*}
J_1\defeq \gamma^3\E\int_{(\mathcal{I}_d)^3} & \overline{\Vol}(\varrho_1 \TG_1\cap \mathbb{B}_R) \overline{\Vol}(\varrho_2 \TG_2\cap \mathbb{B}_R) \overline{\Vol}(\varrho_1 \TG_1\cap \varrho_3 \TG_3\cap \mathbb{B}_R) \overline{\Vol}(\varrho_2 \TG_2\cap \varrho_3 \TG_3\cap \mathbb{B}_R)\\
& \times
 \, \lambda^3(\dint(\varrho_1,\varrho_2,\varrho_3))
\end{align*}
with independent copies $\TG_1$, $\TG_2$ and $\TG_3$ of the typical particle. Recall the shorthand notation $A^{(1)}=\mathbb{B}(A,1)$ for $A\in\mathcal{K}^d$. From the monotonicity and invariance of $\overline{\Vol}$ we deduce
\begin{align*}
J_1 & \leq \gamma^3\E \overline{\Vol}(\TG_1) \overline{\Vol}(\TG_2) \int_{\mathcal{I}_d^3} \overline{\Vol}(\varrho_1 \TG_1\cap \varrho_3 \TG_3\cap \mathbb{B}_R) \overline{\Vol}(\varrho_2 \TG_2\cap \varrho_3 \TG_3\cap \mathbb{B}_R) \, \lambda^3(\dint(\varrho_1,\varrho_2,\varrho_3)) \\
& \leq \gamma^3\E \overline{\Vol}(\TG_1) \overline{\Vol}(\TG_2) \int_{\mathcal{I}_d} \int_{\mathcal{I}_d} \Vol(\varrho_1 \TG_1^{(1)}\cap \varrho_3 \TG_3^{(1)}\cap \mathbb{B}_{R+1}) \, \lambda(\dint\varrho_1) \\
& \hspace{4cm} \times \int_{\mathcal{I}_d} \Vol(\varrho_2 \TG_2^{(1)}\cap \varrho_3 \TG_3^{(1)}\cap \mathbb{B}_{R+1}) \, \lambda(\dint\varrho_2) \, \lambda(\dint\varrho_3) .
\end{align*}
Applying Lemma \ref{lem:intelem} (three times) and using the monotonicity of $\overline{\Vol}$, we get
\begin{align*}
J_1 & \leq \gamma^3\E \overline{\Vol}(\TG_1) \overline{\Vol}(\TG_2) \int_{\mathcal{I}_d} \Vol(\TG_1^{(1)}) \Vol(\varrho_3 \TG_3^{(1)}\cap \mathbb{B}_{R+1}) \Vol(\TG_2^{(1)}) \Vol(\varrho_3 \TG_3^{(1)}\cap \mathbb{B}_{R+1}) \, \lambda(\dint\varrho_3) \\
& \leq \gamma^3\E \overline{\Vol}(\TG_1)^2 \overline{\Vol}(\TG_2)^2 \overline{\Vol}(\TG_3) \int_{\mathcal{I}_d} \Vol(\varrho_3 \TG_3^{(1)}\cap \mathbb{B}_{R+1}) \, \lambda(\dint\varrho_3) \\
& = \gamma^3\E \overline{\Vol}(G_1)^2 \overline{\Vol}(\TG_2)^2 \overline{\Vol}(\TG_3)^2 \overline{\Vol}(\mathbb{B}_{R}).
\end{align*}
Thus we have shown that
$$
T_1^2 \leq \frac{4 \gamma^3  C_{4}  M(\phi)^4}{\sigma_R^4} (\E \overline{\Vol}(\TG)^2)^3 \, \overline{\Vol}(\mathbb{B}_{R}).
$$
By the assumption \eqref{eqn:liminf_var} there exist constants $\underline{\sigma},r_0\in(0,\infty)$ such that
\begin{equation}\label{eqn:lower_bound_sigma_R}
\sigma_R^2 \geq \underline{\sigma} \Vol(\mathbb{B}_R)
\end{equation}
for all $R\geq r_0$. Together with
\begin{equation}\label{eq:ratio_volumes_balls}
\lim_{R\to\infty} \frac{\overline{\Vol}(\mathbb{B}_R)}{\Vol(\mathbb{B}_R)}=e^{d-1},
\end{equation}
there exists a constant $C_{T_1}\in(0,\infty)$ such that
\begin{equation}\label{eqn:bound_T_1}
T_1 \leq \frac{C_{T_1}}{\sqrt{\Vol(\mathbb{B}_R)}}
\end{equation}
for $R\geq r_0$. Applying the Cauchy--Schwarz inequality and Lemma \ref{lem:moments_D} (c) to $T_2^2$, and using the monotonicity of $ \overline{\Vol}$, we see that the right-hand side of \eqref{eqn:intermediate_bound_T_1} without the factor $4$ is an upper bound for $T_2^2$, whence
\begin{equation}\label{eqn:bound_T_2}
T_2 \leq \frac{C_{T_1}}{2\sqrt{\Vol(\mathbb{B}_R)}}
\end{equation}
for $R\geq r_0$. Instead of bounding $T_3$ directly, we study the more general term
$$
T_{m,q} = \int_{\mathcal{K}^d} (\E |D_KF_R|^m)^q \, \Lambda(\dint K)
$$
for $m\in\mathbb{N}$ and $q\in(0,1]$ with $mq\geq 1$. From Lemma \ref{lem:moments_D} (b) and by Lemma \ref{lem:intelem} we obtain
\begin{align*}
T_{m,q} & \leq \frac{ C_{m}^q M(\phi)^{mq}}{\sigma_R^{mq}} \int_{\mathcal{K}^d} \overline{\Vol}(K\cap \mathbb{B}_R)^{mq} \, \Lambda(\dint K) \\
& = \frac{\gamma C_{m}^q M(\phi)^{mq}}{\sigma_R^{mq}} \E \int_{\mathcal{I}_d} \overline{\Vol}(\varrho \TG\cap \mathbb{B}_R)^{mq} \, \lambda(\dint \varrho) \\
& \leq \frac{\gamma C_{m}^q M(\phi)^{mq}}{\sigma_R^{mq}} \E \overline{\Vol}(\TG)^{mq-1} \int_{\mathcal{I}_d} \Vol(\varrho \TG^{(1)}\cap \mathbb{B}_{R+1}) \, \lambda(\dint \varrho) \\
& = \frac{\gamma C_{m}^q M(\phi)^{mq}}{\sigma_R^{mq}} \E \overline{\Vol}(\TG)^{mq}\, \overline{\Vol}(\mathbb{B}_R).
\end{align*}
Consequently, there exists a constant $C_{T_{m,q}}\in(0,\infty)$ such that
\begin{equation}\label{eqn:bound_T_M_q}
T_{m,q} \leq \frac{C_{T_{m,q}}}{\Vol(\mathbb{B}_R)^{mq/2-1}}
\end{equation}
for $R\geq r_0$. For $m=3$ and $q=1$ this means that there exists a constant $C_{T_3}\in(0,\infty)$ such that
\begin{equation}\label{eqn:bound_T_3}
T_3 \leq \frac{C_{T_3}}{\sqrt{\Vol(\mathbb{B}_R)}}
\end{equation}
for $R\geq r_0$. Combination of \eqref{eqn:bound_T_1}, \eqref{eqn:bound_T_2}, and \eqref{eqn:bound_T_3} proves part (b).

\medskip

Next we establish part (a). From \cite[Theorem 3.1]{BPT} it follows that
$$
\mathbf{d}_{\rm Wass}(F_R,N) \leq T_1+T_2+T_3'
$$
with $T_1$ and $T_2$ as above and
$$
T_3' = \int_{\mathcal{K}^d} (\E |D_KF_R|^3 )^{1/3} (\E \min\{  |D_KF_R|^3 ,\sqrt{8} |D_KF_R|^{3/2} \})^{2/3} \, \Lambda(\dint K).
$$
For a fixed $u>0$, the H\"older inequality and Lemma \ref{lem:moments_D} imply
\begin{align}\label{eq:sec15A}
T_3' & \leq \int_{\mathcal{K}^d} \mathbf{1}\{\overline{\Vol}(K)\leq u\} \E |D_KF_R|^3 + 8^{1/3} \mathbf{1}\{\overline{\Vol}(K) > u\} (\E |D_KF_R|^3)^{2/3} \, \Lambda(\dint K) \notag\\
& \leq \frac{ C_{3} M(\phi)^3}{\sigma_R^3} \int_{\mathcal{K}^d} \mathbf{1}\{\overline{\Vol}(K)\leq u\} \overline{\Vol}(K\cap\mathbb{B}_R)^3 \, \Lambda(\dint K)\notag \\
& \qquad + \frac{8^{1/3}  C_{3}^{2/3} M(\phi)^2}{\sigma_R^2} \int_{\mathcal{K}^d} \mathbf{1}\{\overline{\Vol}(K) > u\} \overline{\Vol}(K\cap\mathbb{B}_R)^2 \, \Lambda(\dint K).
\end{align}
Furthermore, we have
\begin{align}\label{eq:sec15B}
& \int_{\mathcal{K}^d} \mathbf{1}\{\overline{\Vol}(K)\leq u\} \overline{\Vol}(K\cap\mathbb{B}_R)^3  \Lambda(\dint K) = \gamma \E \int_{\mathcal{I}_d} \mathbf{1}\{\overline{\Vol}(\TG)\leq u\} \overline{\Vol}(\varrho \TG\cap\mathbb{B}_R)^3 \, \lambda(\dint \varrho) \notag\\
& \leq \gamma \E \mathbf{1}\{\overline{\Vol}(\TG)\leq u\} \overline{\Vol}(\TG)^2 \int_{\mathcal{I}_d} \Vol(\varrho \TG^{(1)}\cap\mathbb{B}_{R+1}) \, \lambda(\dint \varrho) \notag\\
& = \gamma \E \mathbf{1}\{\overline{\Vol}(\TG)\leq u\}  \overline{\Vol}(\TG)^3\, \overline{\Vol}(\mathbb{B}_R) \leq \gamma u^3 \, \overline{\Vol}(\mathbb{B}_R)
\end{align}
and, similarly, 
\begin{align}\label{eq:sec15C}
\int_{\mathcal{K}^d} \mathbf{1}\{\overline{\Vol}(K) > u\} \overline{\Vol}(K\cap\mathbb{B}_R)^2 \, \Lambda(\dint K) \leq \gamma \E \mathbf{1}\{\overline{\Vol}(\TG)> u\} \overline{\Vol}(\TG)^2\, \overline{\Vol}(\mathbb{B}_R).
\end{align}
Combining \eqref{eq:sec15A}, \eqref{eq:sec15B},  \eqref{eq:sec15C} and using \eqref{eqn:lower_bound_sigma_R} as well as \eqref{eq:ratio_volumes_balls}, we obtain 
$$
\limsup_{R\to\infty} T_3' \leq \frac{8^{1/3} e^{d-1} \gamma C_{3}^{2/3} M(\phi)^2}{\underline{\sigma}^2} \E \mathbf{1}\{\overline{\Vol}(\TG)> u\} \overline{\Vol}(\TG)^2.
$$
As the right-hand side vanishes for $u\to\infty$, we see that $\lim_{R\to\infty} T_3'=0$. Note that this argument requires only $\E\overline{\Vol}(\TG)^2<\infty$. Together with \eqref{eqn:bound_T_1} and \eqref{eqn:bound_T_2}, which rely also only on $\E \overline{\Vol}(\TG)^2<\infty$, we derive $\lim_{R\to\infty} \mathbf{d}_{\rm Wass}(F_R,N)=0$. Since convergence in Wasserstein distance implies convergence in distribution, part (a) follows.

For the Kolmogorov distance in part (c), \cite[Theorem 1.2 and Lemma 4.2]{LPS} yield
$$
\mathbf{d}_{\rm Kol}(F_R,N) \leq T_1+T_2+T_3+T_4+T_5+T_6
$$
with $T_1$, $T_2$, and $T_3$ as above and
\begin{align*}
T_4 & = \frac{1}{2} \max\bigg\{ 4 \bigg( \int_{\mathcal{K}^d} (\E (D_KF_R)^4 )^{1/2} \, \Lambda(\dint K) \bigg)^{1/2}, \bigg( 4 \int_{\mathcal{K}^d} \E (D_KF_R)^4 \, \Lambda(\dint K) + 2 \bigg)^{1/4} \bigg\} \\
& \quad \times \int_{\mathcal{K}^d} (\E (D_KF_R)^4 )^{3/4} \, \Lambda(\dint K), \\
T_5 & = \bigg( \int_{\mathcal{K}^d} \E (D_KF_R)^4 \, \Lambda(\dint K) \bigg)^{1/2}, \\
T_6 & = \bigg( \int_{(\mathcal{K}^d)^2} 6 (\E (D_{K_1}F_R)^4)^{1/2} (\E (D^2_{K_1,K_2}F_R)^4)^{1/2} + 3 \E (D^2_{K_1,K_2}F_R)^4 \, \Lambda^2(\dint (K_1,K_2)) \bigg)^{1/2}.
\end{align*}
From \eqref{eqn:bound_T_M_q} we derive that there exist constants $C_{T_{4}},C_{T_5}\in(0,\infty)$ such that
\begin{equation}\label{eqn:bound_T_4_and_T_5}
T_4 \leq \frac{C_{T_4}}{\sqrt{\Vol(\mathbb{B}_R)}} \quad \text{and} \quad T_5 \leq \frac{C_{T_5}}{\sqrt{\Vol(\mathbb{B}_R)}}
\end{equation}
for $R\geq r_0$. From Lemma \ref{lem:moments_D} (b), (c) and the monotonicity of $\overline{\Vol}$, we obtain
\begin{align*}
T_6^2 &\leq  \frac{(6C_{4}+3C_{4}) M(\phi)^4}{\sigma_R^4} \\
& \qquad\times \int_{(\mathcal{K}^d)^2} \overline{\Vol}(K_1\cap\mathbb{B}_R)^2 \,\overline{\Vol}(K_1\cap K_2\cap\mathbb{B}_R)^2 \, \Lambda^2(\dint(K_1,K_2)).
\end{align*}
Since the integral on the right-hand side can be bounded from above by
\begin{align*}
& \gamma^2 \E \overline{\Vol}(\TG_1)^2 \,\overline{\Vol}(\TG_2) \int_{(\mathcal{I}_d)^2} \Vol(\varrho_1 \TG_1^{(1)}\cap\varrho_2 \TG_2^{(1)} \cap \mathbb{B}_{R+1}) \, \lambda^2(\dint (\varrho_1,\varrho_2))\\ & = \gamma^2 \E \overline{\Vol}(\TG_1)^3 \,\overline{\Vol}(\TG_2)^2 \,\overline{\Vol}(\mathbb{B}_{R})
= \gamma^2 (\E \overline{\Vol}(\TG)^3) (\E \overline{\Vol}(\TG)^2) \,\overline{\Vol}(\mathbb{B}_{R}),
\end{align*}
with independent copies $\TG_1$ and $\TG_2$ of the typical particle $\TG$, there exists a constant $C_{T_6}\in(0,\infty)$ such that
\begin{equation}\label{eqn:bound_T_6}
T_6 \leq \frac{C_{T_6}}{\sqrt{\Vol(\mathbb{B}_R)}}
\end{equation}
for $R\geq r_0$. Combination of \eqref{eqn:bound_T_1}, \eqref{eqn:bound_T_2}, \eqref{eqn:bound_T_3}, \eqref{eqn:bound_T_4_and_T_5}, and \eqref{eqn:bound_T_6} completes the proof of part (c).
\end{proof}

\begin{proof}[Proof of Theorem \ref{thm:CLT_multivariate}]
For $\boldsymbol{\alpha}=(\alpha_1,\hdots,\alpha_m)\in\mathbb{R}^m$ let
$$
S_{\boldsymbol{\alpha},R}\defeq 
\sum_{i=1}^m \alpha_i (\phi_i(Z\cap\mathbb{B}_R) - \E \phi_i(Z\cap\mathbb{B}_R) )
$$
and $\psi\defeq\sum_{i=1}^m \alpha_i \phi_i$. Note that $\psi$ is a geometric functional and that
$$
 S_{\boldsymbol{\alpha},R} = \psi(Z\cap\mathbb{B}_R)- \E \psi(Z\cap\mathbb{B}_R).
$$
Because of
$$
\Var S_{\boldsymbol{\alpha},R} = \sum_{i,j=1}^m \alpha_i \alpha_j \Cov(\phi_i(Z\cap\mathbb{B}_R),\phi_j(Z\cap\mathbb{B}_R)),
$$
the assumption \eqref{eqn:Assumption_Convergence_Covariances} implies
$$
\lim_{R\to\infty} \frac{\Var S_{\boldsymbol{\alpha},R}}{\Vol(\mathbb{B}_R)} = \boldsymbol{\alpha}^\top \Sigma \boldsymbol{\alpha}.
$$
If $\boldsymbol{\alpha}^\top \Sigma \boldsymbol{\alpha}>0$, we obtain from Theorem \ref{thm:CLT_univariate} that
$$
\frac{S_{\boldsymbol{\alpha},R}}{\sqrt{\Vol(\mathbb{B}_R)}} \overset{d}{\longrightarrow} \langle \boldsymbol{\alpha}, \mathbf{N}_\Sigma \rangle \quad \text{as} \quad R\to\infty,
$$
where $\langle\cdot,\cdot\rangle$ denotes the inner product in $\mathbb{R}^m$. In case that $\boldsymbol{\alpha}^\top \Sigma \boldsymbol{\alpha}=0$, we have $S_{\boldsymbol{\alpha},R}/\sqrt{\Vol(\mathbb{B}_R)} \overset{d}{\longrightarrow} 0$ as $R\to\infty$ and $\langle \boldsymbol{\alpha}, \mathbf{N}_\Sigma \rangle=0$, so that the previous identity is still valid. Thus, the assertion follows from the Cram\'er--Wold device.
\end{proof}

\subsection*{Acknowledgements}
DH, GL, and MS have been funded by the German Research Foundation (DFG) through the Research Unit ``Geometry and Physics of Spatial Random Systems'' and the Priority Programme ``Random Geometric Systems''.  DH and MS  express their gratitude to the Hausdorff Research Institute for Mathematics in Bonn, Germany, where this work has been presented in Spring 2024 within the Dual Trimester Program: ``Synergies between modern probability, geometric analysis and stochastic geometry".

\end{document}